\documentclass[11pt]{amsart}
\usepackage[english]{babel}
\usepackage[dvips]{graphicx}
\usepackage{amsmath,amsfonts,amsthm,amssymb,bbm,comment}
\usepackage[usenames, dvipsnames]{color}
\usepackage{hyperref}
 \newcommand{\ind}{\mathbbm{1}}

\usepackage[width=6.5in,height=21.5cm]{geometry}

\newtheorem{theo}{Theorem}[section]

\newtheorem{ass}{Assumption}[section]
\newtheorem{lem}{Lemma}[section]
\newtheorem{cor}{Corollary}[section]
\newtheorem{rem}{Remark}[section]
\newtheorem{example}{Example}[section]

\title[Empirical means of interacting RSPs]{\bf Networks of reinforced stochastic processes:\\ asymptotics for the empirical means}

\author[G. Aletti]{Giacomo Aletti}
\address{ADAMSS Center,
  Universit\`a degli Studi di Milano, Milan, Italy}
\email{giacomo.aletti@unimi.it}

\author[I. Crimaldi]{Irene Crimaldi}
\address{IMT School for Advanced Studies, Lucca, Italy}
\email{irene.crimaldi@imtlucca.it}

\author[A. Ghiglietti]{Andrea Ghiglietti}
\address{Universit\`a degli Studi di Milano, Milan, Italy}
\email{andrea.ghiglietti@unimi.it (Corresponding author)}

\date{\today}

\begin{document}

\maketitle

\begin{abstract}
This work deals with systems of {\em interacting reinforced stochastic
  processes}, where each process $X^j=(X_{n,j})_n$ is located at a
vertex $j$ of a finite {\em weighted direct graph}, and it can be
interpreted as the sequence of ``actions'' adopted by an agent $j$ of
the network.  The interaction among the evolving dynamics of these
processes depends on the weighted adjacency matrix $W$ associated to
the underlying graph: indeed, the probability that an agent $j$
chooses a certain action depends on its personal ``inclination''
$Z_{n,j}$ and on the inclinations $Z_{n,h}$, with $h\neq j$, of the
other agents according to the elements of $W$.\\ \indent Asymptotic
results for the stochastic processes of the personal inclinations
$Z^j=(Z_{n,j})_n$ have been subject of studies in recent papers
(e.g.~\cite{ale-cri-ghi,cri-dai-lou-min});
%
%
while the asymptotic behavior of quantities based on the stochastic
processes $X^j$ of the actions has never been studied yet. In this
paper, we fill this gap by characterizing the asymptotic behavior of
the {\em empirical means} $N_{n,j}=\sum_{k=1}^n X_{k,j}/n$, proving
their almost sure synchronization and some central limit theorems in
the sense of stable convergence.  Moreover, we discuss some
statistical applications of these convergence results concerning
confidence intervals for the random limit toward which all the
processes of the system converge and tools to make inference on the
matrix $W$.
\end{abstract}

\\

\smallskip
\noindent {\em 2010 AMS classification:} 60F05, 60F15, 60K35;
62P35, 91D30.

\section{Framework, model and main ideas}
Real-world systems often consist of interacting agents
that may develop a collective behavior (e.g. \cite{alb-bar, bar-alb, new, hof}):
in neuroscience the brain is an active network
where billions of neurons interact in various ways in the cellular
circuits; many studies in biology focus on the interactions between
different sub-systems; social sciences and economics deal with
individuals that take decisions under the influence of other
individuals, and also in engineering and computer science ``consensus
problems'', understood as the ability of interacting dynamic agents to
reach a common asymptotic stable state, play a crucial role.
In all these frameworks, an usual phenomenon is the {\em synchronization},
that could be roughly defined as the tendency of different interacting
agents to adopt a common behavior. Taking into account various
features of these systems, several research works employed agent-based
models in order to analyze how macro-level collective behaviors arise
as products of the micro-level processes of interaction among the
agents of the system (we refer to \cite{are} for a detailed and well
structured survey on this topic, rich of examples and references). The
main goals of these researches are twofold: (i) to understand
whether and when a (complete or partial) synchronization in a
dynamical system of interacting agents can emerge and (ii) to
analyze the interplay between the network topology of the
interactions among the agents and the dynamics followed by the agents.
\\

\indent This work is placed in the stream of scientific literature
that studies systems of {\em interacting urn models} (e.g.~\cite{ale-ghi,
  ben, che-luc, cir, cri-dai-min, dai-lou-min, lau2, lau1, lima,
  mar-val, pag-sec, sah}) and their variants and generalizations
(e.g.~\cite{ale-cri-ghi, cri-dai-lou-min}). Specifically, our work deals with
the class of the so-called {\em interacting reinforced stochastic processes}
considered in~\cite{ale-cri-ghi, cri-dai-lou-min}.
Generally speaking, by reinforcement in a
stochastic dynamics we mean any mechanism for which the probability
that a given event occurs has an increasing dependence on the number
of times that events of the same type occurred in the past. This {\em
  ``self-reinforcing property''}, also known as {\em ``preferential
  attachment rule''}, is a key feature governing the dynamics of many
biological, economic and social systems (see, e.g. \cite{pem}).  The best
known example of reinforced stochastic process is the standard
P\`{o}lya's urn \cite{egg-pol,mah}, which has been widely studied and
generalized (some recent variants can be found in \cite{ale-ghi-pag,
  ale-ghi-vid, aoudia-perron, ber-cri-pra-rig-barriere,
  cal-che-cri-pam, chen-kuba, collevecchio, cri-ipergeom, ghi-pag14,
  ghi-vid-ros, laru-page}).  \\

\indent We consider a system of $N\geq1$ interacting reinforced
stochastic processes $\{X^j=(X_{n,j})_{n\geq 1}:\, 1 \leq j\leq N\}$
positioned at the vertices of a {\em weighted directed graph} $G=(V,\,
E,\, W)$, where $V:=\{1,...,N\}$ denotes the set of vertices,
$E\underline{\subset} V\times V$ the set of edges and
$W=[w_{h,j}]_{h,j\in V\times V}$ the weighted adjacency matrix with
$w_{h,j}\geq 0$ for each pair of vertices.  The presence of the edge
$(h,j)\in E$ indicates a ``direct influence'' that the vertex $h$ has
on the vertex $j$ and it corresponds to a strictly positive element
$w_{h,j}$ of $W$ that represents a weight quantifying this influence.
We assume the weights to be normalized so that $\sum_{h=1}^N
w_{h,j}=1$ for each $j\in V$.  For any $n\geq 1$, we assume the random
variables $\{X_{n,j}:\,j\in V\}$ to take values in $\{0,1\}$ and hence
they can be interpreted as ``two-modality actions'' that the agents of
the network can adopt at time $n$. Formally, the interaction between
the processes $\{X^j:\,j\in V\}$ is modeled as follows: for any $n\geq
0$, the random variables $\{X_{n+1,j}:\,j\in V\}$ are conditionally
independent given ${\mathcal F}_{n}$ with
\begin{equation}\label{interacting-1-intro}
P(X_{n+1,j}=1\, |\, {\mathcal F}_n)=\sum_{h=1}^N w_{h,j} Z_{n,h},
\end{equation}
and, for each $h\in V$,
\begin{equation}\label{interacting-2-intro}
Z_{n,h}=(1-r_{n-1})Z_{n-1,h}+r_{n-1}X_{n,h},
\end{equation}
where $Z_{0,h}$ are random variables with values in $[0,1]$,
${\mathcal F}_n:=\sigma(Z_{0,h}: h\in V)\vee \sigma(X_{k,j}:\, 1\leq
k\leq n,\,j\in V )$ and $0\leq r_n<1$ are real numbers such that
\begin{equation}\label{ass-r-intro}
\lim_n n^{\gamma} r_n=c>0\qquad\hbox{with } 1/2<\gamma\leq 1.
\end{equation}
(We refer to \cite{cri-dai-lou-min} for a discussion on the case
$0<\gamma \leq 1/2$, for which we have a different asymptotic
behavior of the model that is out of the scope of this research
work.) For example, if at each vertex $j\in V$ we have a standard
P\'olya's urn, with initial composition given by the pair $(a,b)$,
then we have $r_n=(a+b+n+1)^{-1}$ and so $\gamma=c=1$.  Each random
variable $Z_{n,h}$ takes values in $[0,1]$ and it can be interpreted
as the ``personal inclination'' of the agent $h$ of adopting ``action
1'', so that the probability that the agent $j$ adopts ``action 1'' at
time $(n+1)$ depends on its personal inclination $Z_{n,j}$ and on the
inclinations $Z_{n,h}$, with $h\neq j$, of the other agents at time
$n$ according to the ``influence-weights'' $w_{h,j}$.\\

\indent The previous quoted papers \cite{ale-cri-ghi, cri-dai-lou-min,
  cri-dai-min, dai-lou-min} are all focused on the asymptotic
behavior of the stochastic processes of the ``personal inclinations''
$\{Z^j=(Z_{n,j})_n:\, j\in V \}$ of the agents. On the contrary, in this
work we focus on the average of times in which the agents adopt
``action 1'', i.e. we study the stochastic processes of the {\em
  empirical means} $\{N^j=(N_{n,j})_{n}:\, j\in V\}$ defined,
for each $j\in V$, as $N^j_0:=0$ and, for any $n\geq 1$,
\begin{equation}\label{medie-empiriche-intro}
N_{n,j}:=\frac{1}{n}\sum_{k=1}^{n} X_{k,j}\,.
\end{equation}
Since $(1/n)\sum_{k=1}^{n-1}X_{k,j}=(1-1/n)N_{n-1,j}$, the dynamics of
each process $N^j$ can be written as follows:
\begin{equation}\label{interacting-N}
N_{n,j}=\left(1-\frac{1}{n}\right)N_{n-1,j}+\frac{1}{n}X_{n,j}.
\end{equation}
Furthermore, the above dynamics \eqref{interacting-1-intro},
\eqref{interacting-2-intro} and \eqref{interacting-N} can be expressed
in a compact form, using the random vectors
$\mathbf{X}_{n}:=(X_{n,1},\dots,X_{n,N})^{\top}$ for $n\geq 1$,
$\mathbf{N}_{n}:=(N_{n,1},\dots,N_{n,N})^{\top}$ and
$\mathbf{Z}_{n}:=(Z_{n,1},\dots,Z_{n,N})^{\top}$ for $n\geq 0$, as:
\begin{equation}\label{eq:dynamic-0}
E[\mathbf{X}_{n+1}|\mathcal{F}_{n}]=W^{\top}\,\mathbf{Z}_{n}\,,
\end{equation}
where $W^{\top}\mathbf{1}=\mathbf{1}$ by the normalization of the
weights, and
\begin{equation}\label{eq:dynamic}
\left\{\begin{aligned}
&\mathbf{Z}_{n}\ =\
\left(1-r_{n-1}\right)\mathbf{Z}_{n-1}\ +\ r_{n-1}\mathbf{X}_{n},\\
&\mathbf{N}_{n}\ =\
\left(1-\frac{1}{n}\right)\mathbf{N}_{n-1}\ +\ \frac{1}{n}\mathbf{X}_{n}.
\end{aligned}\right.
\end{equation}

\indent In the framework described above, under
  suitable assumptions, we prove that all the stochastic processes
  $N^j=(N_{n,j})_n$, with $j\in V$, converge almost surely to the same
  limit random variable (in other words, we prove their almost sure
  synchronization), which is also the common limit random variable of
  the stochastic processes $Z^j=(Z_{n,j})_n$, say $Z_{\infty}$ (see
  Theorem~\ref{th:sincro}).  From an applicative point of view, the
  almost sure synchronization of the stochastic processes $N^j$ means
  that, with probability $1$, the percentages of times that the agents
  of the system adopt the ``action 1'' tend to the same random value
  $Z_{\infty}$. Moreover, we provide some Central Limit Theorems
  (CLTs) in the sense of stable convergence, in which the asymptotic
  variances and covariances are expressed as functions of the
  eigen-structure of the weighted adjacency matrix $W$ and of the
  parameters $\gamma,\, c$ governing the asymptotic behavior of the
  sequence $(r_n)_n$ (see Theorem~\ref{thm:asymptotics_theta_gamma},
  Theorem~\ref{thm:N_1_gamma_1}, Theorem~\ref{thm:asymptotics_theta_1}
  and Theorem~\ref{thm:asymptotics_Z_1_star}). These convergence
  results are also discussed from the point of view of the
  statistical applications. In particular, they lead to the
  construction of asymptotic confidence intervals for the common limit
  random variable $Z_{\infty}$ based on the random variables $X_{n,j}$
  through the empirical means \eqref{medie-empiriche-intro}, that
  specifically require neither the knowledge of the initial random
  variables $\{Z_{0,j}:\, j\in V\}$ nor of the exact expression of
  the sequence $(r_n)_n$. For the case $\gamma=1$, that for instance
  includes the case of interacting standard P\'olya's urns, we also
  provide a statistical test, based on the random variables $X_{n,j}$
  through the empirical means \eqref{medie-empiriche-intro}, to make
  inference on the weighted adjacency matrix $W$ of the network. The
  fact that the confidence intervals and the inferential procedures
  presented in this work are based on $X_{n,j}$, instead of $Z_{n,j}$
  as done in~\cite{ale-cri-ghi}, represents a great improvement in any
  area of application, since the ``actions'' $X_{n,j}$ adopted by the
  agents of the network are much more likely to be observed than their
  ``personal inclinations'' $Z_{n,j}$ of adopting these actions.\\

\indent The proofs of the given CLTs are a substantial part of
this work and we believe that it is worth spending some words on the
main tools employed and technical issues faced.  The essential idea is
to decompose the stochastic process $(\mathbf{N}_{n})_n$ into the sum
of two terms, where the first one converges, at the rate
$n^{\gamma-1/2}$ for each $1/2<\gamma\leq 1$, stably in the strong
sense with respect to the filtration $(\mathcal{F}_n)_n$ toward a
certain Gaussian kernel, and the second term is an
$(\mathcal{F}_n)$-adapted stochastic process that converges stably to
a suitable Gaussian kernel, with the corresponding rate and argument
required for the proof different according to the value of $\gamma$.
Indeed, when $1/2<\gamma<1$, the second term converges stably at the
same rate as above, i.e.~$n^{\gamma-1/2}$, and in the proof we have a
certain remainder term that tends to zero in probability (see
Theorem~\ref{thm:asymptotics_theta_hat_gamma}). On the contrary, when
$\gamma=1$ and $N\geq 2$ (the case $\gamma=1, \, N=1$ is similar to
the previous case $1/2<\gamma<1$), we do not have the convergence to
zero of that remainder term (see
Remark~\ref{rem:different_approach_gamma_1}) and so we develop a
coupling technique based on the pair of random vectors
$(\mathbf{Z}_n,\mathbf{N}_n)$. So doing, we determine two different
rates for the convergence of the second term, depending on the second
highest real part ${\mathcal Re}(\lambda^*)$ of the eigenvalues of $W$
(see Theorem~\ref{thm:asymptotics_theta_hat_1} where the rate is
$\sqrt{n}$ and Theorem~\ref{thm:asymptotics_theta_1_star} where the
rate is $\sqrt{n/\ln(n)}$). The contributions of the two terms are in
particular reflected in the analytic expressions of the asymptotic
covariance matrix of $\mathbf{N}_{n}$ (see
Theorem~\ref{thm:asymptotics_theta_gamma},
Theorem~\ref{thm:asymptotics_theta_1} and
Theorem~\ref{thm:asymptotics_Z_1_star}), where there is a component
$\widetilde{\Sigma}_{\gamma}$ due to the first term (which is zero
when the rate for the second term is $\sqrt{n/\ln(n)}$, because the
contribution of the first term vanishes) and another component
due to the second term that is different in the various cases:
$\widehat{\Gamma}_{\gamma}$ when $1/2<\gamma<1$, and
$\widehat{\Sigma}_{\mathbf{N}\mathbf{N}}$ or
$\widehat{\Sigma}^{*}_{\mathbf{N}\mathbf{N}}$, according to the value
of ${\mathcal Re}(\lambda^*)$, when $\gamma=1$.\\

\indent Summing up, the main focus here concerns the asymptotic
behavior of the empirical means $({\mathbf N}_n)_n$, that has not been
subject of study yet. Furthermore, although we recover some results on
$(\mathbf{Z}_n)_n$ proved in \cite{ale-cri-ghi}, we point out that the
existence of joint central limit theorems for the pair
$(\mathbf{Z}_n,\mathbf{N}_n)$ is not obvious because the ``discount
factors'' in the dynamics of the increments
$(\mathbf{Z}_{n}-\mathbf{Z}_{n-1})_n$ and
$(\mathbf{N}_{n}-\mathbf{N}_{n-1})_n$ are generally different. Indeed,
as shown in~\eqref{eq:dynamic}, these two stochastic processes follow
the dynamics
\begin{equation}\label{eq-increments-intro}
\left\{
\begin{aligned}
&\mathbf{Z}_{n}-\mathbf{Z}_{n-1}\ =\
r_{n-1}\left(\mathbf{X}_{n}-\mathbf{Z}_{n-1}\right),\\
&\mathbf{N}_{n}-\mathbf{N}_{n-1}\ =\
\frac{1}{n}\left(\mathbf{X}_{n}-\mathbf{N}_{n-1}\right),
\end{aligned}
\right.
\end{equation}
and so, when we assume $1/2<\gamma<1$, it could be surprising that
there exists a common convergence rate. In addition, we will show
that, when $1/2<\gamma<1$, the stochastic processes $N^j=(N_{n,j})_n$
located at different vertices of the graph synchronize among each
other faster than how they converge to the common random limit
$Z_{\infty}$, i.e.  for any pair of vertices $(j,h)$ with $j\neq h$,
the velocity at which $(N_{n,j}-N_{n,h})_n$ converges almost surely to
zero is higher than the one at which $N^j=(N_{n,j})_n$ and
$N^h=(N_{h,n})_n$ converge almost surely to $Z_{\infty}$.  At the
contrary, when $\gamma=1$ the stochastic processes $N^j=(N_{n,j})_n$
synchronize and converge almost surely to $Z_{\infty}$ at the same
velocity.  The same asymptotic behaviors characterize the stochastic
processes $Z^j=(Z_{n,j})_n$, as proved also in~\cite{ale-cri-ghi,
  cri-dai-lou-min}.  However, while it is somehow guessable from
\eqref{eq-increments-intro} that the velocities of synchronization and
convergence for the processes $Z^j=(Z_{n,j})_n$ depend on the
parameter $\gamma$, it could be somehow unexpected that, although the
discount factor of the increments $(\mathbf{N}_{n}-\mathbf{N}_{n-1})$
is always $n^{-1}$, the corresponding velocities for the processes
$N^j=(N_{n,j})_n$ also depend on $\gamma$ and, in general, also these
processes do not synchronize and converge to $Z_{\infty}$ at the same
velocity.  As we will see, this fact is essentially due to their
dependence on the process $(\mathbf{Z}_{n})_n$, which is induced by
the process $(\mathbf{X}_{n})_n$. It is worthwhile to note that
dynamics similar to~\eqref{eq-increments-intro} have already been
considered in the Stochastic Approximation literature.  Specifically,
in~\cite{mok-pel} the authors established a CLT for a pair of
recursive procedures having two different step-sizes. However, this
result does not apply to our situation.  Indeed, the covariance
matrices $\Sigma_\mu$ and $\Sigma_\theta$ in their main result
(Theorem~1) are deterministic, while the asymptotic covariance
matrices in our CLTs are random (as said before, they depend on the
random variable $Z_\infty$).  This is why we do not use the simple
convergence in distribution, but we employ the notion of stable
convergence, which is, among other things, essential for the
considered statistical applications. Finally, in~\cite{mok-pel}, the
authors find two different convergence rates, depending on the two
different step-sizes, while, as already said, we find a common
convergence rate.\\

\indent The rest of the paper is organized as follows. In Section
\ref{section_model} we describe the notation and the assumptions used
along the paper.  In Section \ref{section_asymptotic_results} we
illustrate our main results and we discuss some possible statistical
applications. An interesting example of interacting system is also
provided in order to clarify the statement of the theorems and the
related comments. Section \ref{section_proofs} contains the proofs or
the main steps of the proofs (postponing some technical lemmas to
Appendix \ref{app-A}) of the presented results. For the reader's
convenience, Appendix \ref{app-B} supplies a brief review on the
notion of stable convergence and its variants.

\section{Notation and assumptions}\label{section_model}

Throughout all the paper, we will adopt the same notation used in
\cite{ale-cri-ghi}.  In particular, we denote by ${\mathcal Re}(z)$,
${\mathcal Im}(z)$, $\overline{z}$ and $|z|$ the real part, the
imaginary part, the conjugate and the modulus of a complex number $z$.
Then, for a matrix $A$ with complex elements, we let $\overline{A}$
and $A^{\top}$ be its conjugate and its transpose, while we indicate
by $|A|$ the sum of the modulus of its elements. The identity matrix
is denoted by $I$, independently of its dimension that will be clear
from the context.  The spectrum of $A$, i.e.~the set of all the
eigenvalues of $A$ repeated with their multiplicity, is denoted by
$Sp(A)$, while its sub-set containing the eigenvalues with maximum
real part is denoted by $\lambda_{\max}(A)$, i.e.  $\lambda^*\in
\lambda_{\max}(A)$ whenever ${\mathcal Re}(\lambda^*)=\max\{ {\mathcal
  Re}(\lambda):\, \lambda\in Sp(A) \}$.  Finally, we consider any
vector $\mathbf{v}$ as a matrix with only one column (so that all the
above notations apply to $\mathbf{v}$) and we indicate by
$\|\mathbf{v}\|$ its norm, i.e. $ \|\mathbf{v} \|^2 =
\overline{\mathbf{v}}^{\top}\mathbf{v}$.  The vectors whose elements
are all ones or zeros are denoted by $\mathbf{1}$ and $\mathbf{0}$,
respectively, independently of their dimension that will be clear from
the context.\\

\indent Throughout all the paper, we assume that the following
conditions hold:

\begin{ass}\label{ass:r_n}
There exist real constants $c>0$ and $1/2<\gamma\leq 1$ such that
condition \eqref{ass-r-intro} is satisfied, which can be rewritten as
\begin{equation}\label{ass:condition_r_n_gamma}
n^{\gamma} r_n\ =\ c\ +\ o(1).
\end{equation}
In some results for $\gamma=1$, we will require a slightly stricter
condition than~\eqref{ass:condition_r_n_gamma}, that is:
\begin{equation}\label{ass:condition_r_n_1}
nr_n\ =\ c\ +\ O\left(n^{-1}\right).
\end{equation}
We will explicitly mention this assumption in the statement of the
theorems when it is required.
\end{ass}

\begin{ass}
The weighted adjacency matrix $W$ is irreducible and diagonalizable.
\end{ass}

The irreducibility of $W$ reflects a situation in which all the
vertices are connected among each others and hence there are no
sub-systems with independent dynamics (see \cite{ale-cri-ghi, ale-ghi}
for further details). The diagonalizability of $W$ allows us to find a
non-singular matrix $\widetilde{U}$ such that
$\widetilde{U}^{\top}W(\widetilde{U}^{\top})^{-1}$ is diagonal with
complex elements $\lambda_j\in Sp(W)$.  Notice that each column
$\mathbf{u}_j$ of $\widetilde{U}$ is a left eigenvector of $W$
associated to some eigenvalue $\lambda_j$.  Without loss of
generality, we set $\|\mathbf{u}_j\|=1$.  Moreover, when the
multiplicity of some $\lambda_j$ is bigger than one, we set the
corresponding eigenvectors to be orthogonal.  Then, if we define
$\widetilde{V}=(\widetilde{U}^{\top})^{-1}$, we have that each column
$\mathbf{v}_j$ of $\widetilde{V}$ is a right eigenvector of $W$
associated to $\lambda_j$ such that
\begin{equation}\label{eq:relazioni-0}
\mathbf{u}_j^{\top}\,\mathbf{v}_j=1,\quad\mbox{ and }\qquad
\mathbf{u}_h^{\top}\,\mathbf{v}_j=0,\ \forall h\neq j.
\end{equation}

These constraints combined with the above assumptions on $W$
(precisely, $w_{h,j}\geq 0$, $W^{\top}\mathbf{1}=\mathbf{1}$ and
the irreducibility) imply, by
Frobenius-Perron Theorem, that $\lambda_1:=1$ is an eigenvalue of $W$
with multiplicity one, $\lambda_{\max}(W)=\{1\}$ and
\begin{equation}\label{eq:relazioni-1}
\mathbf{u}_1=N^{-1/2}\mathbf{1}, \qquad
N^{-1/2}{\mathbf 1}^{\top}{\mathbf v}_1=1\qquad
\mbox{and}\qquad
v_{1,j}:=[\mathbf{v}_1]_j>0\;\ \forall 1\leq j\leq N.
\end{equation}

We use $U$ and $V$ to indicate the sub-matrices of
$\widetilde{U}$ and $\widetilde{V}$, respectively, whose columns are
the left and the right eigenvectors of $W$ associated to
$Sp(W)\setminus \{1\}$, that is $\{\mathbf{u}_2,\dots,\mathbf{u}_N\}$
and $\{\mathbf{v}_2,\dots,\mathbf{v}_N\}$, respectively, and, finally, we denote
by $\lambda^*$ an eigenvalue belonging to $Sp(W)\setminus\{1\}$ such that
$$
{\mathcal Re}(\lambda^*)=\max\left\{ {\mathcal Re}(\lambda_j):\,
\lambda_j\in Sp(W)\setminus\{1\}\right\}.
$$
In other words, if we denote by $D$ the diagonal matrix whose
elements are $\lambda_j\in Sp(W)\setminus\{1\}$, we have
$\lambda^*\in \lambda_{\max}(D)$.

\section{Main results and discussion}
\label{section_asymptotic_results}

In this section, we present and discuss our main results concerning
the asymptotic behavior of the joint process
$(\mathbf{Z}_{n},\mathbf{N}_{n})_n$.  We recall the assumptions stated
in Section~\ref{section_model} and we refer to Appendix \ref{app-B}
for a brief review on the notion of stable convergence and its
variants.  \\

\indent We start by providing a first-order asymptotic result
concerning the almost sure convergence of the sequence of pairs
$(\mathbf{Z}_{n},\mathbf{N}_{n})_n$.
\begin{theo}\label{th:sincro}
For $N\geq 1$, we have
\begin{equation*}
\mathbf{N}_n\ \stackrel{a.s.}{\longrightarrow}\ Z_{\infty}\mathbf{1}\,,
\end{equation*}
where $Z_{\infty}$ is the random variable with values in $[0,1]$
defined as the common almost sure limit of the stochastic processes
$Z^j=(Z_{n,j})_n$. \\
\indent Moreover, the following statements hold true:
\begin{itemize}
\item[(i)] $P(Z_\infty=z)=0$ for any $z\in (0,1)$.
\item[(ii)] If we have
  $P(\bigcap_{j=1}^N\{Z_{0,j}=0\})+P(\bigcap_{j=1}^N\{Z_{0,j}=1\})<1$,
  then $P(0<Z_\infty<1)>0$.
\end{itemize}
\end{theo}
In particular, this result states that, when $N\geq 2$, all the
stochastic processes $N^j=(N_{n,j})_n$, located at the different
vertices $j\in V$ of the graph, synchronize almost surely, i.e. all of
them converge almost surely toward the same random variable
$Z_\infty$. Moreover, this random variable is the same limit toward
which all the stochastic processes $Z^j=(Z_{n,j})_n$ synchronize
almost surely (see Theorem 3.1 in~\cite{ale-cri-ghi}). In addition, it
is interesting to note that the synchronization holds true without any
assumption on the initial configuration $\mathbf{Z}_{0}$ and for any
choice of the weighted adjacency matrix $W$ with the required
assumptions. Finally, note that the synchronization is induced along
time independently of the fixed size $N$ of the network, and so it
does not require a large-scale limit (i.e. the limit for $N\to
+\infty$), which is usual in statistical mechanics for the study of
interacting particle systems.\\

\indent We now focus on the second-order asymptotic results.
Specifically, we present joint central limit theorems for the sequence
of pairs $(\mathbf{Z}_{n},\mathbf{N}_{n})_n$ in the sense of stable
convergence, that establish the rate of convergence to the limit
$Z_{\infty}\mathbf{1}$ given in Theorem~\ref{th:sincro} and the
relative asymptotic random covariance matrices.  First, we consider
the case $1/2<\gamma<1$:
\begin{theo}\label{thm:asymptotics_theta_gamma}
For $N\geq 1$ and $1/2<\gamma<1$, we have that
\begin{equation}\label{eq:CLT_theta_gamma}
n^{\gamma-\frac{1}{2}}
\begin{pmatrix}
\mathbf{Z}_n-Z_{\infty}\mathbf{1}\\
\mathbf{N}_n-Z_{\infty}\mathbf{1}
\end{pmatrix}
{\longrightarrow}\
\mathcal{N} \left(\ \mathbf{0}\ ,\ Z_{\infty}(1-Z_{\infty})
\begin{pmatrix}
\widetilde{\Sigma}_{\gamma} & \widetilde{\Sigma}_{\gamma}\\
\widetilde{\Sigma}_{\gamma} & \widetilde{\Sigma}_{\gamma}+\widehat{\Gamma}_{\gamma}
\end{pmatrix}\ \right)\ \ \ \ stably,
\end{equation}
where
\begin{equation}\label{def:Sigmatilde_gamma}
\widetilde{\Sigma}_{\gamma}:=
\widetilde{\sigma}_{\gamma}^2{\mathbf 1}{\mathbf 1}^{\top}
\qquad\mbox{ and }\qquad
\widetilde{\sigma}_{\gamma}^2:=
\frac{c^2\,\|\mathbf{v}_1\|^2}{N(2\gamma-1)}
>0,
\end{equation}
and
\begin{equation}\label{def:Sigmahat_NN_gamma}
\widehat{\Gamma}_{\gamma}:=
\widehat{\sigma}_{\gamma}^2{\mathbf 1}{\mathbf 1}^{\top}
\qquad\mbox{ and }\qquad
\widehat{\sigma}_{\gamma}^2:=
\frac{c^2\,\|\mathbf{v}_1\|^2}{N (3-2\gamma)}
>0.
\end{equation}
\end{theo}

\begin{rem}
\rm Some considerations can be drawn by looking at the analytic
expressions of $\widetilde{\sigma}_{\gamma}^2$ and
$\widehat{\sigma}_{\gamma}^2$ in~\eqref{def:Sigmatilde_gamma}
and~\eqref{def:Sigmahat_NN_gamma}, respectively.  First, they are both
decreasing in $N$, so that the asymptotic variances are small when the
number of vertices in the graph is large.  Second, they are both
increasing in $c$ and decreasing in $\gamma$, which, recalling that
$\lim_n n^{\gamma }r_n=c$, means that the faster is the convergence to
zero of the sequence $(r_n)_n$, the lower are the values of the
asymptotic variances $\widetilde{\sigma}_{\gamma}^2$ and
$\widehat{\sigma}_{\gamma}^2$. Third, when $\gamma$ is close to $1/2$,
$\widetilde{\sigma}_{\gamma}^2$ becomes very large, while
$\widehat{\sigma}_{\gamma}^2$ remains bounded, and hence the processes
$(\mathbf{Z}_n-Z_{\infty}\mathbf{1})$ and
$(\mathbf{N}_n-Z_{\infty}\mathbf{1})$ become highly correlated.
Finally, since we have
$$
1\leq 1+\|\mathbf{v}_1-\mathbf{u}_1\|^2=\|\mathbf{v}_1\|^2\leq N,
$$
we can obtain the following lower and upper bounds for
$\widetilde{\sigma}_{\gamma}^2$ and $\widehat{\sigma}_{\gamma}^2$ (not
depending on $W$):
$$
\frac{c^2}{N(2\gamma-1)}\leq
\widetilde{\sigma}_{\gamma}^2\leq\frac{c^2}{(2\gamma-1)}
\quad\hbox{and}\quad
\frac{c^2}{N(3-2\gamma)}\leq\widehat{\sigma}_{\gamma}^2\leq
\frac{c^2}{(3-2\gamma)}.
$$ Notice that the lower bound is achieved when
$\mathbf{v}_1=\mathbf{u}_1=N^{-1/2}\mathbf{1}$, i.e. when $W$ is
doubly stochastic.
\end{rem}

\begin{rem} \rm
Note that from~\eqref{eq:CLT_theta_gamma} of
Theorem~\ref{thm:asymptotics_theta_gamma}, we get in particular that,
for any pair of vertices $(j,h)$ with $j\neq h$,
$n^{\gamma-\frac{1}{2}}(N_{n,j}-N_{n,h})$ converges to zero in
probability.  Indeed, denoting by $\mathbf{e}_j$ the vector such that
$e_{j,j}=1$ and $e_{j,i}=0$ for all $i\neq j$, we have
$\mathbf{1}^{\top}(\mathbf{e}_j-\mathbf{e}_h)=0$ and hence
$(\mathbf{e}_j-\mathbf{e}_h)^{\top}\widetilde{\Sigma}_{\gamma}
(\mathbf{e}_j-\mathbf{e}_h)=
(\mathbf{e}_j-\mathbf{e}_h)^{\top}\widehat{\Gamma}_{\gamma}
(\mathbf{e}_j-\mathbf{e}_h)=0$.  Therefore,
Theorem~\ref{thm:asymptotics_theta_gamma} implies that the velocity at
which the stochastic processes $N^j=(N_{n,j})_n$, located at different
vertices $j\in V$, synchronize among each other is higher
than the one at which each of them converges almost surely to the
common random limit $Z_{\infty}$. The same asymptotic behavior is
shown also by the stochastic processes $Z^j=(Z_{n,j})_n$ as shown
in \cite{ale-cri-ghi, cri-dai-lou-min}.
\end{rem}

For $\gamma=1$ we need to distinguish the case $N=1$ and the case
$N\geq 2$.  Indeed, in the second case we can have different
convergence rates according to the value of ${\mathcal Re}(\lambda^*)$.
More precisely, we have the following results:
\begin{theo}\label{thm:N_1_gamma_1}
For $N=1$ and $\gamma=1$, we have that
\begin{equation*}
\sqrt{n}
\begin{pmatrix}
Z_n-Z_{\infty}\\
N_n-Z_{\infty}
\end{pmatrix}
{\longrightarrow}\
\mathcal{N}
\left(
\ \mathbf{0}\ ,\ Z_{\infty}(1-Z_{\infty})
\begin{pmatrix}
c^2 & c^2\\
c^2 & c^2+(c-1)^2
\end{pmatrix}
\ \right)\ \ \ \ stably.
\end{equation*}
\end{theo}

\begin{theo}\label{thm:asymptotics_theta_1}
For $N\geq 2$, $\gamma=1$ and ${\mathcal
  Re}(\lambda^{*})<1-(2c)^{-1}$, under
condition~\eqref{ass:condition_r_n_1}, we have that
\begin{equation}\label{eq:CLT_theta_1}
\sqrt{n}
\begin{pmatrix}
\mathbf{Z}_n-Z_{\infty}\mathbf{1}\\
\mathbf{N}_n-Z_{\infty}\mathbf{1}
\end{pmatrix}
{\longrightarrow}\
\mathcal{N} \left(\ \mathbf{0}\ ,\ Z_{\infty}(1-Z_{\infty})
\begin{pmatrix}
\widetilde{\Sigma}_1+\widehat{\Sigma}_{\mathbf{ZZ}}
& \widetilde{\Sigma}_1+\widehat{\Sigma}_{\mathbf{ZN}}\\
\widetilde{\Sigma}_1+\widehat{\Sigma}_{\mathbf{ZN}}^{\top}
& \widetilde{\Sigma}_1+\widehat{\Sigma}_{\mathbf{NN}}
\end{pmatrix}\ \right)\ \ \ \ stably,
\end{equation}
where $\widetilde{\Sigma}_1$ is defined as
in~\eqref{def:Sigmatilde_gamma} with $\gamma=1$, and
\begin{equation}\label{def:Sigmahat_ZZ_1}
\widehat{\Sigma}_{\mathbf{ZZ}}:=U\widehat{S}_{\mathbf{ZZ}}U^{\top},\qquad\mbox{with}
\end{equation}
\begin{equation}\label{def:Shat_ZZ}
[\widehat{S}_{\mathbf{ZZ}}]_{h,j}\ :=\
\frac{c^2}{c(2-\lambda_h-\lambda_j)-1}(\mathbf{v}_{h}^{\top}\mathbf{v}_{j}),\
2\leq h,j\leq N;
\end{equation}
\begin{equation}\label{def:Sigmahat_NN_1}
\widehat{\Sigma}_{\mathbf{NN}}:=
\widetilde{U}\widehat{S}_{\mathbf{NN}}\widetilde{U}^{\top},
\qquad\mbox{with}
\end{equation}
\begin{equation}\label{def:Shat_NN_1}
[\widehat{S}_{\mathbf{NN}}]_{1,1}:=(c-1)^2\|\mathbf{v}_{1}\|^2,\qquad
[\widehat{S}_{\mathbf{NN}}]_{1,j}=[\widehat{S}_{\mathbf{NN}}]_{j,1}:=
\left(\frac{1-c}{1-\lambda_j}\right)(\mathbf{v}_{1}^{\top}\mathbf{v}_{j}),\
2\leq j\leq N,
\end{equation}
\begin{equation}\label{def:Shat_NN_2}
[\widehat{S}_{\mathbf{NN}}]_{h,j}\ :=\
\frac{1+(c-1)[(1-\lambda_h)^{-1}+(1-\lambda_j)^{-1}]}{c(2-\lambda_h-\lambda_j)-1}
(\mathbf{v}_{h}^{\top}\mathbf{v}_{j}),\ 2\leq h,j\leq N;
\end{equation}
\begin{equation}\label{def:Sigmahat_ZN_1}
\widehat{\Sigma}_{\mathbf{ZN}}:=U\widehat{S}_{\mathbf{ZN}}\widetilde{U}^{\top},
\qquad\mbox{with}
\end{equation}
\begin{equation}\label{def:Shat_ZN_1}
[\widehat{S}_{\mathbf{ZN}}]_{h,1}:=\left(\frac{1-c}{1-\lambda_h}\right)
(\mathbf{v}_{h}^{\top}\mathbf{v}_{1}),\ 2\leq h\leq N,
\end{equation}
\begin{equation}\label{def:Shat_ZN_2}
[\widehat{S}_{\mathbf{ZN}}]_{h,j}\ :=\
\frac{c+(c-1)(1-\lambda_h)^{-1}}{c(2-\lambda_h-\lambda_j)- 1}
(\mathbf{v}_{h}^{\top}\mathbf{v}_{j}),\ 2\leq h,j\leq N.
\end{equation}
\end{theo}

The condition ${\mathcal Re}(\lambda^{*})<1-(2c)^{-1}$ in the above
Theorem~\ref{thm:asymptotics_theta_1} is the analogous of the one
typically required for the CLTs in the Stochastic Approximation
framework (e.g. \cite{konda, mok-pel, pel}). However, we deal with a
random limit $Z_\infty$ and random asymptotic covariances and our
proofs are not based on that results, but we employ different
arguments. Moreover, in the next theorem, we analyze also the case
${\mathcal Re}(\lambda^{*})=1-(2c)^{-1}$.

\begin{theo}\label{thm:asymptotics_Z_1_star}
For $N\geq 2$, $\gamma=1$ and ${\mathcal
  Re}(\lambda^{*})=1-(2c)^{-1}$, under
condition~\eqref{ass:condition_r_n_1}, we have that
\begin{equation}\label{eq:CLT_theta_1_star}
\sqrt{\frac{n}{\ln(n)}}
\begin{pmatrix}
\mathbf{Z}_n-Z_{\infty}\mathbf{1}\\
\mathbf{N}_n-Z_{\infty}\mathbf{1}
\end{pmatrix}
{\longrightarrow}\
\mathcal{N} \left(\ \mathbf{0}\ ,\ Z_{\infty}(1-Z_{\infty})
\begin{pmatrix}
\widehat{\Sigma}^{*}_{\mathbf{ZZ}} & \widehat{\Sigma}^{*}_{\mathbf{ZN}}\\
\widehat{\Sigma}_{\mathbf{ZN}}^{*\top} & \widehat{\Sigma}^{*}_{\mathbf{NN}}
\end{pmatrix}\
\right)\ \ \ \ stably,
\end{equation}
where
\begin{equation}\label{def:Sigmahat_ZZ_1_star}
\widehat{\Sigma}^{*}_{\mathbf{ZZ}}:=
U\widehat{S}^{*}_{\mathbf{ZZ}}U^{\top},\qquad\mbox{with}
\end{equation}
\begin{equation}\label{def:Shat_ZZ_star}
[\widehat{S}^{*}_{\mathbf{ZZ}}]_{h,j}\ :=\
c^2(\mathbf{v}_{h}^{\top}\mathbf{v}_{j})\ind_{\{c(2-\lambda_h-\lambda_j)=1\}},\
2\leq h,j\leq N;
\end{equation}
\begin{equation}\label{def:Sigmahat_NN_1_star}
\widehat{\Sigma}^{*}_{\mathbf{NN}}:=
U\widehat{S}^{*}_{\mathbf{NN}}U^{\top},\qquad\mbox{with}
\end{equation}
\begin{equation}\label{def:Shat_NN_star}
[\widehat{S}^{*}_{\mathbf{NN}}]_{h,j}\ :=\
\frac{\lambda_h\lambda_j}{(1-\lambda_h)(1-\lambda_j)}
(\mathbf{v}_{h}^{\top}\mathbf{v}_{j})\ind_{\{c(2-\lambda_h-\lambda_j)=1\}},\
2\leq h,j\leq N;
\end{equation}
\begin{equation}\label{def:Sigmahat_ZN_1_star}
\widehat{\Sigma}^{*}_{\mathbf{ZN}}:=
U\widehat{S}^{*}_{\mathbf{ZN}}U^{\top},\qquad\mbox{with}
\end{equation}
\begin{equation}\label{def:Shat_ZN_star}
[\widehat{S}^{*}_{\mathbf{ZN}}]_{h,j}\ :=\
\frac{c\lambda_j}{1-\lambda_h}(\mathbf{v}_{h}^{\top}\mathbf{v}_{j})
\ind_{\{c(2-\lambda_h-\lambda_j)=1\}},\ 2\leq h,j\leq N.
\end{equation}
\end{theo}

\begin{rem}
 \rm The central limit theorem only for the stochastic process
 $(\mathbf{Z}_n)_n$ can be established in the case ${\mathcal
   Re}(\lambda^*)<1-(2c)^{-1}$ replacing
 condition~\eqref{ass:condition_r_n_1} with the more general
 assumption~\eqref{ass:condition_r_n_gamma} (see Theorem 3.2 in~\cite{ale-cri-ghi}).
 However, condition~\eqref{ass:condition_r_n_1} is essential in our
 proof of the central limit theorem for the joint stochastic process
 $(\mathbf{Z}_n, \mathbf{N}_n)_n$ as stated in
 Theorem~\ref{thm:asymptotics_theta_1}.
\end{rem}

\begin{rem}
\rm From Theorem~\ref{thm:asymptotics_theta_1} and
Theorem~\ref{thm:asymptotics_Z_1_star} we get that, when $N\geq 2$ and
$\gamma=1$, for any pair of vertices $(j,h)$ with $j\neq h$, the
difference $(N_{n,j}-N_{n,h})$ converges almost surely to zero with
the same velocity at which each process $N^j=(N_{n,j})$ converges
almost surely to $Z_{\infty}$.  (The same asymptotic behavior is shown
also by the stochastic processes $Z^j=(Z_{n,j})_n$ as provided in
\cite{ale-cri-ghi, cri-dai-lou-min}.)  Indeed, although
$\widetilde{\Sigma}_{1}(\mathbf{e}_j-\mathbf{e}_h)=\mathbf{0}$ and
$\mathbf{u}_1^{\top}(\mathbf{e}_j-\mathbf{e}_h)=0$, we have
$U^{\top}(\mathbf{e}_j-\mathbf{e}_h)\neq\mathbf{0}$ and hence,
setting $\mathbf{u}_{j,h}:=U^{\top}(\mathbf{e}_j-\mathbf{e}_h)$ and
$\widetilde{\mathbf{u}}_{j,h}:=
\widetilde{U}^{\top}(\mathbf{e}_j-\mathbf{e}_h)=(0,\mathbf{u}_{j,h})^{\top}$,
for ${\mathcal Re}(\lambda^{*})<1-(2c)^{-1}$ by~\eqref{eq:CLT_theta_1}
we have
\begin{equation*}
\sqrt{n}
\begin{pmatrix}
Z_{n,j}-Z_{n,h}\\
N_{n,j}-N_{n,h}
\end{pmatrix}
{\longrightarrow}\
\mathcal{N} \left(\ \mathbf{0}\ ,\ Z_{\infty}(1-Z_{\infty})
\begin{pmatrix}
\mathbf{u}^{\top}\widehat{S}_{\mathbf{ZZ}}\mathbf{u}_{j,h}
& \mathbf{u}_{j,h}^{\top}\widehat{S}_{\mathbf{ZN}}\widetilde{\mathbf{u}}_{j,h}\\
\widetilde{\mathbf{u}}_{j,h}^{\top}\widehat{S}_{\mathbf{ZN}}^{\top}\mathbf{u}_{j,h}
& \widetilde{\mathbf{u}}_{j,h}^{\top}\widehat{S}_{\mathbf{NN}}
\widetilde{\mathbf{u}}_{j,h}
\end{pmatrix}\ \right)\ \ \ \ \hbox{stably};
\end{equation*}
while for ${\mathcal Re}(\lambda^{*})=1-(2c)^{-1}$
by~\eqref{eq:CLT_theta_1_star} we have
\begin{equation*}
\sqrt{\frac{n}{\ln(n)}}
\begin{pmatrix}
Z_{n,j}-Z_{n,h}\\
N_{n,j}-N_{n,h}
\end{pmatrix}
{\longrightarrow}\
\mathcal{N} \left(\ \mathbf{0}\ ,\ Z_{\infty}(1-Z_{\infty})
\begin{pmatrix}
\mathbf{u}_{j,h}^{\top}\widehat{S}^{*}_{\mathbf{ZZ}}\mathbf{u}_{j,h}
& \mathbf{u}_{j,h}^{\top}\widehat{S}^{*}_{\mathbf{ZN}}\mathbf{u}_{j,h}\\
\mathbf{u}_{j,h}^{\top}\widehat{S}_{\mathbf{ZN}}^{*\top}\mathbf{u}_{j,h}
& \mathbf{u}_{j,h}^{\top}\widehat{S}^{*}_{\mathbf{NN}}\mathbf{u}_{j,h}
\end{pmatrix}\
\right)\ \ \ \ \hbox{stably}.
\end{equation*}
Notice that the only elements $[\widehat{S}_{\mathbf{NN}}]_{h,j}$ that
count in the above limit relations are those with $2\leq h,j\leq N$.
Then, from~\eqref{def:Sigmahat_NN_1} we can see that these elements
remain bounded for any value of $c$, while
from~\eqref{def:Sigmahat_ZZ_1} we can see that the elements of
$\widehat{S}_{\mathbf{ZZ}}$ are increasing in $c$. (The same
considerations can be made for the elements of the matrices
$\widehat{S}^{*}_{\mathbf{NN}}$ and $\widehat{S}^{*}_{\mathbf{ZZ}}$,
but in this case the value of $c$ is uniquely determined by
${\mathcal Re}(\lambda^*)$).  As a consequence, for large values of
$c$, the asymptotic variance of $(N_{n,j}-N_{n,h})$ becomes negligible
with respect to the one of $(Z_{n,j}-Z_{n,h})$.  Therefore, when
$N\geq 2$ and $\gamma=1$, the synchronization between the empirical
means $N^j=(N_{n,j})_n$, located at different vertices $j\in V$,
is more accurate than the synchronization between the
stochastic processes $Z^j=(Z_{n,j})_n$.
\end{rem}

An interesting example of interacting system is provided by the {\em
  ``mean-field interaction''}, already considered
in~\cite{ale-cri-ghi, cri-dai-lou-min, cri-dai-min,
  dai-lou-min}. Naturally, all the weighted adjacency matrices
introduced and analyzed in \cite{ale-cri-ghi} can be considered as
well.

\begin{example}\label{ex:mean_field}\rm
The mean-field interaction can be expressed in terms of a
particular weighted adjacency matrix $W$ as follows: for any $1\leq h,
j\leq N$ (here we consider only the true ``interacting case'', that is
$N\geq 2$)
\begin{equation}\label{def:W_crimaldi}
w_{h,j}\ =\ \frac{\alpha}{N}\ +\ \delta_{h,j}(1-\alpha)
\qquad\mbox{with } \alpha\in [0,1],
\end{equation}
where $\delta_{h,j}$ is equal to $1$ when $h=j$ and to $0$
otherwise. Note that $W$ in~\eqref{def:W_crimaldi} is irreducible for
$\alpha>0$ and so we are going to consider this case. Since $W$ is
doubly stochastic, we have ${\mathbf v}_1={\mathbf
  u}_1=N^{-1/2}\mathbf{1}$.  Thus, for $1/2<\gamma<1$, we have
$$
\widetilde{\sigma}_{\gamma}^2=\frac{c^2}{N(2\gamma-1)}
\quad \hbox{and} \quad
\widehat{\sigma}_{\gamma}^2=\frac{c^2}{N(3-2\gamma)}.
$$
Furthermore, we have $\lambda_j=1-\alpha$ for all
$\lambda_j\in Sp(W)\setminus\{1\}$ and, consequently, the conditions
${\mathcal Re}(\lambda^*)<1-(2c)^{-1}$ or ${\mathcal
  Re}(\lambda^*)=1-(2c)^{-1}$ required in the previous results when
$\gamma=1$ correspond to the conditions $2c\alpha>1$ or $2c\alpha=1$.
Finally, since $W$ is also symmetric, we have $U=V$ and so
$U^{\top}U=V^{\top}V=I$ and
$UU^{\top}=VV^{\top}=(I-N^{-1}\mathbf{1}\mathbf{1}^{\top})$. Therefore,
for the case $\gamma=1$ and $2c\alpha>1$, we obtain:
\begin{itemize}
\item[(i)] $\widehat{S}_{\mathbf{ZZ}}=\frac{c^2}{2c\alpha-1}I$;
\item[(ii)] $[\widehat{S}_{\mathbf{NN}}]_{1,1}=(c-1)^2$ and
  $[\widehat{S}_{\mathbf{NN}}]_{j,j}=\frac{1+2(c-1)\alpha^{-1}}{2c\alpha-1}$
  for $2\leq j\leq N$, while $[\widehat{S}_{\mathbf{NN}}]_{h,j}=0$ for any
  $h\neq j$, $1\leq h,j\leq N$;
\item[(iii)]
  $[\widehat{S}_{\mathbf{ZN}}]_{j,j}=\frac{c+(c-1)\alpha^{-1}}{2c\alpha-1}$
  for $2\leq j\leq N$, while $[\widehat{S}_{\mathbf{ZN}}]_{h,j}=0$ for any
  $h\neq j$, $2\leq h\leq N$ and $1\leq j\leq N$.
\end{itemize}
Finally, when $\gamma=1$ and $2c\alpha=1$, we get:
\begin{itemize}
\item[(i)] $\widehat{S}^{*}_{\mathbf{ZZ}}=c^2I$;
\item[(ii)] $\widehat{S}^{*}_{\mathbf{NN}}=\frac{(1-\alpha)^2}{\alpha^2} I$;
\item[(iii)] $\widehat{S}^{*}_{\mathbf{ZN}}=\frac{c(1-\alpha)}{\alpha} I$. \qed
\end{itemize}
\end{example}

\subsection{Some comments on statistical applications}
\label{section_statistics}

The first statistical tool that can be derived from the previous
convergence results is the construction of asymptotic confidence
intervals for the limit random variable $Z_{\infty}$.  This issue has
been already considered in~\cite{ale-cri-ghi}, where from the central
limit theorem for
$\widetilde{Z}_{n}:=N^{-1/2}\,\mathbf{v}_1^{\top}\,\mathbf{Z}_{n}$
(recalled here in the following
Theorem~\ref{thm:asymptotics_Z_tilde}), a confidence interval with
approximate level $(1-\theta)$ is obtained for any $1/2<\gamma\leq 1$
as:
\begin{equation}\label{eq:confidence_interval_Z_inf_based_Z}
CI_{1-\theta}(Z_{\infty})\ =\
\widetilde{Z}_{n}\ \pm\ \frac{z_{\theta}}{n^{\gamma-1/2}}
\sqrt{\widetilde{Z}_n(1-\widetilde{Z}_n)\widetilde{\sigma}_{\gamma}^2},
\end{equation}
where $\widetilde{\sigma}_{\gamma}^2$ is defined as
in~\eqref{def:Sigmatilde_gamma} (also for $\gamma=1$) and $z_\theta$
is such that ${\mathcal N}(0,1)(z_\theta,+\infty)=\theta/2$. We note
that the construction of the above interval requires to know the
following quantities:
\begin{itemize}
\item[(i)] $N$: the number of vertices in the network;
\item[(ii)] $\mathbf{v}_1$: the right eigenvector of $W$ associated to
  $\lambda_1=1$ (note that it is not required to know the whole weighted
  adjacency matrix $W$, e.g.  we have
  $\mathbf{v}_1=\mathbf{u}_1=N^{-1/2}\mathbf{1}$ for any doubly
  stochastic matrix);
\item[(iii)] $\gamma$ and $c$: the parameters that describe the
  first-order asymptotic approximation of the sequence $(r_n)_n$
(see Assumption \ref{ass:r_n}).
\end{itemize}
In addition, the asymptotic confidence interval
in~\eqref{eq:confidence_interval_Z_inf_based_Z} requires the
observation of $\widetilde{Z}_{n}$, and so of $Z_{n,j}$ for any $j\in V$.
However, this requirement may not be feasible in practical
applications since the initial random variables $Z_{0,j}$ and the
exact expression of the sequence $(r_n)_n$ are typically unknown.
For instance, if at each vertex $j\in V$ we have a standard P\`olya's
urn with initial composition given by the pair $(a,b)$, then we have
$Z_{0,j}=a/(a+b)$ and $r_n=(a+b+n+1)^{-1}$ and hence, when the initial
composition is unknown, we have neither $Z_{0,j}$ nor the exact
value of $r_n$, but we can get $\gamma=c=1$. To face this problem,
here we propose asymptotic confidence intervals for $Z_\infty$ that do
not require the observation of $Z_{n,j}$, but are based on the
empirical means $N_{n,j}=\sum_{k=1}^{n}X_{k,j}/n$, where the random
variables $X_{k,j}$ are typically observable. To this aim, we consider
the convergence results presented in
Section~\ref{section_asymptotic_results} on the asymptotic behavior of
$\mathbf{N}_n$.  \\

\indent We first focus on the case $1/2<\gamma<1$ and we construct an
asymptotic confidence interval for $Z_{\infty}$ based on the empirical
means $N_{n,j}$, with $j\in V$, and the quantities in
(i)-(ii)-(iii). Indeed, setting
$\widetilde{N}_n:=N^{-1/2}\mathbf{v}_1^{\top}\mathbf{N}_n$ and using
the relation
$\mathbf{v}_1^{\top}\mathbf{u}_1=N^{-1/2}\mathbf{v}_1^{\top}\mathbf{1}=1$
(see~\eqref{eq:relazioni-1}),
from Theorem~\ref{thm:asymptotics_theta_gamma} we obtain that
$$
n^{\gamma-1/2}(\widetilde{N}_n-Z_{\infty}){\longrightarrow}\
\mathcal{N}\left(\ 0\ ,\
Z_{\infty}(1-Z_{\infty})(\widetilde{\sigma}_{\gamma}^2+\widehat{\sigma}_{\gamma}^2)\
\right)\ \ \ \ \ \ \ \hbox{stably},
$$
where $\widetilde{\sigma}_{\gamma}^2$ and $\widehat{\sigma}_{\gamma}^2$
are defined in~\eqref{def:Sigmatilde_gamma}
and~\eqref{def:Sigmahat_NN_gamma}, respectively.
Then, for $1/2<\gamma<1$, we have the following confidence interval
with approximate level $(1-\theta)$:
\begin{equation*}
CI_{1-\theta}(Z_{\infty})\ =\
\widetilde{N}_n\ \pm\ \frac{z_{\theta}}{n^{\gamma-1/2}}
\sqrt{\widetilde{N}_n(1-\widetilde{N}_n)
(\widetilde{\sigma}_{\gamma}^2+\widehat{\sigma}_{\gamma}^2)}.
\end{equation*}

\indent Analogously, for $\gamma=1$ and $N=1$, from Theorem
\ref{thm:N_1_gamma_1} we get
\begin{equation*}
CI_{1-\theta}(Z_{\infty})\ =\
N_n\ \pm\ \frac{z_{\theta}}{\sqrt{n}}
\sqrt{N_n(1-N_n)
(c^2+(c-1)^2)}.
\end{equation*}

\indent When $\gamma=1$ and $N\geq 2$, we have to distinguish two
cases according to the value of ${\mathcal Re}(\lambda^*)$. Thus, in this
case, the construction of suitable asymptotic confidence intervals for
$Z_\infty$ requires also the knowledge of ${\mathcal Re}(\lambda^*)$.
Specifically, when ${\mathcal Re}(\lambda^*)<1-(2c)^{-1}$, from
Theorem~\ref{thm:asymptotics_theta_1}, using the relations
$\mathbf{v}_1^{\top}\mathbf{u}_1=1$ and
$\mathbf{v}_1^{\top}U=\mathbf{0}$ (see~\eqref{eq:relazioni-0}), we
obtain that
$$
\sqrt{n}(\widetilde{N}_n-Z_{\infty}){\longrightarrow}\
\mathcal{N}\left(\ 0\ ,\ Z_{\infty}(1-Z_{\infty})
(\widetilde{\sigma}_{1}^2+N^{-1}[\widehat{S}_{\mathbf{NN}}]_{1,1})\
\right)\ \ \ \ \ \ \ \hbox{stably},
$$
where $\widetilde{\sigma}_{1}^2=c^2\|\mathbf{v}_1\|^2/N$ and
$[\widehat{S}_{\mathbf{NN}}]_{1,1}=(c-1)^2\|\mathbf{v}_1\|^2$.  Hence,
in this case we find:
\begin{equation*}
CI_{1-\theta}(Z_{\infty})\ =\
\widetilde{N}_n\ \pm\ \frac{z_{\theta}}{\sqrt{n}}
\!\sqrt{\widetilde{N}_n(1-\widetilde{N}_n)
\!\left(\textstyle{\frac{(c^2+(c-1)^2)\|\mathbf{v}_1\|^2}{N}}\right)}.
\end{equation*}
Note that analogous asymptotic confidence intervals for $Z_{\infty}$
can be constructed replacing $\widetilde{N}_n$ by another real
stochastic processes $(\mathbf{a}^{\top}\mathbf{N}_n)_n$, where
$\mathbf{a}\in\mathbb{R}^N$ and
$\mathbf{a}^{\top}\mathbf{1}=1$.  \\ \indent Finally, when ${\mathcal
  Re}(\lambda^*)=1-(2c)^{-1}$, we can not use $\widetilde{N}_n$ since,
by Theorem \ref{thm:asymptotics_Z_1_star} and the fact that
$\mathbf{v}_1^{\top}U=\mathbf{0}$, we have
$\sqrt{n/\ln(n)}(\widetilde{N}_n-Z_{\infty})\to 0$ in probability.
Therefore, in this case we need to replace the vector $\mathbf{v}_1$
by another vector $\mathbf{a}\in\mathbb{R}^N$ with
$\mathbf{a}^{\top}\mathbf{1}=1$ and $\mathbf{a}^{\top}U\neq
\mathbf{0}$.

\begin{example}\rm
In the case of a system with $N\geq 2$ and mean-field interaction (see
Example~\ref{ex:mean_field}), we get the following asymptotic
confidence intervals for $Z_{\infty}$ with approximate level
$(1-\theta)$:
\begin{itemize}
\item[(i)] when $1/2<\gamma<1$, setting
  $\widetilde{N}_n=N^{-1}\mathbf{1}^{\top}\mathbf{N}_n$, we have
\begin{equation*}
CI_{1-\theta}(Z_{\infty})\ =\
\widetilde{N}_n\ \pm\ \frac{z_{\theta}}{n^{\gamma-1/2}}
\sqrt{\widetilde{N}_n(1-\widetilde{N}_n)\frac{2c^2}{N(2\gamma-1)(3-2\gamma)}};
\end{equation*}
\item[(ii)] when $\gamma=1$ and $2c\alpha>1$, setting
  $\widetilde{N}_n=N^{-1}\mathbf{1}^{\top}\mathbf{N}_n$, we have
\begin{equation*}
CI_{1-\theta}(Z_{\infty})\ =\
\widetilde{N}_n \pm\ \frac{z_{\theta}}{\sqrt{n}}
\sqrt{\widetilde{N}_n(1-\widetilde{N}_n)\frac{c^2+(c-1)^2}{N}};
\end{equation*}
\item[(iii)] when $\gamma=1$ and $2c\alpha=1$, setting
  $\widetilde{N}^a_n:=\mathbf{a}^{\top}\mathbf{N}_n$ with
  $\mathbf{a}^{\top}\mathbf{1}=1$ and
  $\mathbf{a}\neq N^{-1}\mathbf{1}$, we have
\begin{equation*}
CI_{1-\theta}(Z_{\infty})\ =\
\widetilde{N}^a_n \pm\ z_{\theta} \sqrt{ \frac{\ln(n)}{n} }
\sqrt{\widetilde{N}^a_n(1-\widetilde{N}^a_n)}
\frac{(1-\alpha)}{\alpha}\|\mathbf{a}-N^{-1}\mathbf{1}\|,
\end{equation*}
where the last term follows by recalling that
$UU^{\top}=I-N^{-1}\mathbf{1}\mathbf{1}^{\top}$ and
noticing that
$$
\mathbf{a}^{\top}UU^{\top}\mathbf{a}=
\mathbf{a}^{\top}(I-N^{-1}\mathbf{1}\mathbf{1}^{\top})\mathbf{a}
=\|\mathbf{a}\|^2-N^{-1}=\|\mathbf{a}-N^{-1}\mathbf{1}\|^2
$$
(where for the last two equalities we used that
$\mathbf{a}^{\top}\mathbf{1}=1$).
\end{itemize}
\qed
\end{example}

\indent Another possible statistical application of the convergence
results of Section~\ref{section_asymptotic_results} concerns the
inference on the weighted adjacency matrix $W$ based on the empirical
means $N_{n,j}$, with $j\in V$, instead of the random variables
$Z_{n,j}$ as done in~\cite{ale-cri-ghi}.  Let us assume $N\geq 2$ (the
proper ``interacting'' case).  We propose to construct testing
procedures based on the multi-dimensional real stochastic process
$(UV^{\top}\mathbf{N}_n)_n$. Indeed, we note that it converges to
$\mathbf{0}$ almost surely because
$\mathbf{N}_n\stackrel{a.s.}{\longrightarrow}Z_{\infty}\mathbf{1}$ and
$V^{\top}\mathbf{1}=\mathbf{0}$ (since~\eqref{eq:relazioni-0} and
\eqref{eq:relazioni-1}). Moreover, when $\gamma=1$ and ${\mathcal
  Re}(\lambda^*)<1-(2c)^{-1}$, from
Theorem~\ref{thm:asymptotics_theta_1} we get that
$$
\sqrt{n}UV^{\top}\mathbf{N}_n{\longrightarrow}\
\mathcal{N} \left(\ \mathbf{0}\ ,\
Z_{\infty}(1-Z_{\infty})U[\widehat{S}_{\mathbf{NN}}]_{(-1)}U^{\top}\ \right)
\ \ \ \ \ \ \
\hbox{stably},
$$
where $[\widehat{S}_{\mathbf{NN}}]_{(-1)}$ denotes the square sub-matrix
obtained from $\widehat{S}_{\mathbf{NN}}$ removing its first row and its
first column.\\
Analogously, when $\gamma=1$ and
${\mathcal Re}(\lambda^*)=1-(2c)^{-1}$, from
Theorem~\ref{thm:asymptotics_Z_1_star} we get that
$$
\sqrt{\frac{n}{\ln(n)}} UV^{\top}\mathbf{N}_n{\longrightarrow}\
\mathcal{N} \left(\ \mathbf{0}\ ,\
Z_{\infty}(1-Z_{\infty})\widehat{\Sigma}^{*}_{\mathbf{NN}}\ \right)
\ \ \ \ \ \ \
\hbox{stably}.
$$
Remember that the case $\gamma=1$ includes, for instance, systems
of interacting P\'olya's urns.

\begin{example}\rm
In the case of $N\geq 2$ and mean-field interaction (see
Example~\ref{ex:mean_field}), recalling that $U=V$,
$UU^{\top}=(I-N^{-1}\mathbf{1}\mathbf{1}^{\top})$,
$[\widehat{S}_{\mathbf{NN}}]_{(-1)}=\frac{1+2(c-1)\alpha^{-1}}{2c\alpha-1}I$
and
$\widehat{\Sigma}^{*}_{\mathbf{NN}}=\frac{(1-\alpha)^2}{\alpha^2}UU^{\top}$,
we obtain that:
\begin{itemize}
\item[(i)] when $\gamma=1$ and $2c\alpha>1$,
$$
\sqrt{n}(I-N^{-1}\mathbf{1}\mathbf{1}^{\top})\mathbf{N}_n{\longrightarrow}\
\mathcal{N} \left(\ \mathbf{0}\ ,\
Z_{\infty}(1-Z_{\infty})
\frac{1+2(c-1)\alpha^{-1}}{2c\alpha-1}
(I-N^{-1}\mathbf{1}\mathbf{1}^{\top})\
\right)\ \ \ \ \ \ \
\hbox{stably};
$$
\item[(ii)] when $\gamma=1$ and $2c\alpha=1$,
$$
\sqrt{\frac{n}{\ln(n)}}(I-N^{-1}\mathbf{1}\mathbf{1}^{\top})
\mathbf{N}_n{\longrightarrow}\
\mathcal{N} \left(\ \mathbf{0}\ ,\
Z_{\infty}(1-Z_{\infty})
\frac{(1-\alpha)^2}{\alpha^2}
(I-N^{-1}\mathbf{1}\mathbf{1}^{\top})\ \right)
\ \ \ \ \ \ \
\hbox{stably}.
$$
\end{itemize}
In this framework, it may be of interest to test whether the unknown
parameter $\alpha$ can be assumed to be equal to a specific value
$\alpha_0\in (0,1]$, i.e. we may be interested in a statistical test
  of the type:
$$
H_0:\, W=W_{\alpha_0}\qquad\mbox{ vs} \qquad
H_1:\, W=W_{\alpha}\ \mbox{for some }\alpha\in (0,1]\setminus\{\alpha_0\}.
$$
To this purpose, assuming $2c\alpha_0\geq 1$ and setting
$\widetilde{N}_n:=N^{-1}\mathbf{1}^{\top}\mathbf{N}_n$, we note that:
\begin{itemize}
\item[(i)] for $\gamma=1$ and $2c\alpha_0>1$, under $H_0$ we have that
$$n\left[\widetilde{N}_n(1-\widetilde{N}_n)\right]^{-1}
\frac{2c\alpha_0-1}{1+2(c-1)\alpha_0^{-1}}\,
\mathbf{N}_n^{\top}(I-N^{-1}\mathbf{1}\mathbf{1}^{\top})\mathbf{N}_n
\stackrel{d}\sim\chi^2_{N-1};
$$
\item[(ii)] for $\gamma=1$ and $2c\alpha_0=1$, under $H_0$ we have that
  $$\frac{n}{\ln(n)}\left[\widetilde{N}_n(1-\widetilde{N}_n)\right]^{-1}
  \frac{\alpha_0^2}{(1-\alpha_0)^2}\,
\mathbf{N}_n^{\top}(I-N^{-1}\mathbf{1}\mathbf{1}^{\top})\mathbf{N}_n
\stackrel{d}\sim\chi^2_{N-1}.$$
\end{itemize}
Concerning the distribution of the above quantities for $\alpha\neq \alpha_0$,
since the eigenvectors of $W$ do not depend on $\alpha$, we have that,
for any fixed $\alpha\in (0,1]\setminus\{\alpha_0\}$, under the hypothesis
  $\{W=W_{\alpha}\}\subset H_1$, we have that:
\begin{itemize}
\item[(i)] for $\gamma=1$, $2c\alpha_0>1$ and
for any $\alpha\neq\alpha_0$ such that $2c\alpha>1$,
$$\frac{n}{\widetilde{N}_n(1-\widetilde{N}_n)}
\frac{2c\alpha_0-1}{1+2(c-1)\alpha_0^{-1}}
\mathbf{N}_n^{\top}(I-N^{-1}\mathbf{1}\mathbf{1}^{\top})\mathbf{N}_n
\stackrel{d}\sim
\left(\frac{2c\alpha_0-1}{2c\alpha-1}\right)
\left(\frac{1+2(c-1)\alpha^{-1}}{1+2(c-1)\alpha_0^{-1}}\right)\chi^2_{N-1};$$
while, if $2c\alpha=1$, the above quantity
converges in probability to infinity;
\item[(ii)] for $\gamma=1$, $2c\alpha_0=1$ and for any $\alpha$ such that
$2c\alpha>1$ (which obviously implies $\alpha\neq\alpha_0$), we have
$$\frac{n}{\ln(n)}\left[\widetilde{N}_n(1-\widetilde{N}_n)\right]^{-1}
  \frac{\alpha_0^2}{(1-\alpha_0)^2}\,
\mathbf{N}_n^{\top}(I-N^{-1}\mathbf{1}\mathbf{1}^{\top})\mathbf{N}_n
\stackrel{P}{\longrightarrow}0.$$
\end{itemize}
\qed
\end{example}

The case $1/2<\gamma<1$ requires further future investigation. Indeed,
since $V^{\top}\mathbf{1}=\mathbf{0}$ (by~\eqref{eq:relazioni-0} and
\eqref{eq:relazioni-1}), from
Theorem~\ref{thm:asymptotics_theta_gamma} we obtain that
$n^{\gamma-\frac{1}{2}} UV^{\top}\mathbf{N}_n\to \mathbf{0}$ in
probability.  Then, a central limit theorem for
$UV^{\top}\mathbf{N}_n$ with the exact convergence rate (if exists) is
needful. In this paper, as we will see more ahead in
Remark~\ref{rem:O_gamma_2}, by the computations done in the proofs of
Section~\ref{section_proofs} we can only affirm that $n^{e}
UV^{\top}\mathbf{N}_n\to \mathbf{0}$ in probability for all
$e<\gamma/2$ and, when $e=\gamma/2$, the random vector $n^{e}
UV^{\top}\mathbf{N}_n$ is the sum of a term converging to zero in
probability and a term bounded in $L^1$.  Therefore, further analysis
on the asymptotic behavior of $n^{\gamma/2}UV^{\top}\mathbf{N}_n$
results to be interesting for future developments.

\section{Proofs}\label{section_proofs}

This section contains all the proofs of the results presented in the
previous Section~\ref{section_asymptotic_results}.

\subsection{Preliminary relations and results}

We start by recalling that, given the eigen-structure of $W$ described
in Section~\ref{section_model}, the matrix ${\mathbf
  u}_1{\mathbf v}_1^{\top}$ has real elements and
the following relations hold:
\begin{equation}\label{eq:relazioni-2}
V^{\top}\,\mathbf{u}_1=U^{\top}\,\mathbf{v}_1=\mathbf{0},\quad
V^{\top}\,U=U^{\top}\,V=I\quad\mbox{and}\quad
I={\mathbf u}_1{\mathbf v}_1^{\top} + UV^{\top},
\end{equation}
which implies that the matrix $UV^{\top}$ has real elements. Moreover,
using the matrix $D$ defined in Section~\ref{section_model}, we can
decompose the matrix $W^{\top}$ as follows:
\begin{equation}\label{eq:decomposition-matrix}
W^{\top}\ =\ {\mathbf u}_1{\mathbf v}_1^{\top}\ +\ UDV^{\top}.
\end{equation}

\indent Now, in order to understand the asymptotic behavior of the
stochastic processes $(\mathbf{Z}_{n})_n$ and $(\mathbf{N}_{n})_n$,
let us express the dynamics~\eqref{eq:dynamic} as follows:
\begin{equation}\label{eq:dynamic_SA}
\left\{\begin{aligned}
&\mathbf{Z}_{n+1}-\mathbf{Z}_{n}\ =\ -r_n\left(I-W^{\top}\right)\mathbf{Z}_{n}\
+\ r_n\Delta\mathbf{M}_{n+1},\\
&\mathbf{N}_{n+1}-\mathbf{N}_{n}\ =\
-\frac{1}{n+1}\left(\mathbf{N}_{n}-W^{\top}\mathbf{Z}_{n}\right)\
+\ \frac{1}{n+1}\Delta\mathbf{M}_{n+1},
\end{aligned}\right.
\end{equation}
where
$\Delta\mathbf{M}_{n+1}=(\mathbf{X}_{n+1}-W^{\top}\mathbf{Z}_{n})$ is
a martingale increment with respect to the filtration~${\mathcal
  F}:=({\mathcal F}_{n})_n$. Furthermore, we decompose the stochastic
process $(\mathbf{Z}_{n})_n$ as
\begin{equation}\label{eq:decomposition_Z}
\mathbf{Z}_{n}=\widetilde{Z}_n\mathbf{1} + \widehat{\mathbf{Z}}_{n}
=\sqrt{N}\widetilde{Z}_n \mathbf{u}_1 + \widehat{\mathbf{Z}}_{n},
\;\mbox{where}\;
\left\{
\begin{aligned}
&\widetilde{Z}_{n}:=N^{-1/2}\,\mathbf{v}_1^{\top}\,\mathbf{Z}_{n}, \\
&\widehat{\mathbf{Z}}_{n}:={\mathbf Z}_n-{\mathbf 1}{\widetilde Z}_n
=(I-\mathbf{u}_1\mathbf{v}_1^{\top})\mathbf{Z}_n
=U\,V^{\top}\,\mathbf{Z}_{n};
\end{aligned}
\right.
\end{equation}
while we decompose the stochastic process $(\mathbf{N}_{n})_n$ as
\begin{equation}\label{eq:decomposition_N}
\mathbf{N}_{n}=\widetilde{Z}_n\mathbf{1} + \widehat{\mathbf{N}}_{n}
=\sqrt{N}\widetilde{Z}_n \mathbf{u}_1 + \widehat{\mathbf{N}}_{n},
\qquad\mbox{where}\qquad
\widehat{\mathbf{N}}_{n}:={\mathbf N}_n-{\widetilde Z}_n{\mathbf 1}.
\end{equation}
Then, the asymptotic behavior of the joint stochastic process
$(\mathbf{Z}_{n},\mathbf{N}_{n})_n$ is obtained by establishing the asymptotic
behavior of $(\widetilde{Z}_{n})_n$ and of
$(\widehat{\mathbf{Z}}_{n},\widehat{\mathbf{N}}_{n})_n$.

\begin{rem}
\rm In the particular case when $W$ is doubly stochastic, we have
$\mathbf{v}_1=\mathbf{u}_1=N^{-1/2}\mathbf{1}$.
As a consequence, we have
\begin{equation*}
\widetilde{Z}_{n}=N^{-1}\mathbf{1}^{\top}\mathbf{Z}_{n}
=N^{-1}\sum_{j=1}^N Z_{n,j},
\end{equation*}
which represents the average of the stochastic processes $Z_{n,j}$,
with $j\in V$, in the network, and
$$
\widehat{\mathbf{Z}}_{n}=
\left(I-N^{-1}\mathbf{1}\mathbf{1}^{\top}\right)\mathbf{Z}_n
\quad\hbox{and}\quad
\widehat{\mathbf{N}}_n=\mathbf{N}_n-N^{-1}\mathbf{1}\mathbf{1}^{\top}\mathbf{Z}_n.
$$ Notice that the assumed normalization
$W^{\top}\mathbf{1}=\mathbf{1}$ implies that symmetric matrices $W$
are also doubly stochastic. Therefore, the above equalities hold for
any undirected graph for which $W$ is symmetric by definition.
\end{rem}

Concerning the real-valued stochastic process $(\widetilde{Z}_{n})_n$,
from~\cite[Section 4.2]{ale-cri-ghi} we have that it is an $\mathcal F$-martingale
with values in $[0,1]$ and its dynamics can be expressed as follows:
\begin{equation}\label{eq:dynamic_SA_tilde}
\widetilde{Z}_{n+1}-\widetilde{Z}_{n}\ =\
N^{-1/2} r_n\left(\mathbf{v}_1^{\top}\Delta\mathbf{M}_{n+1}\right).
\end{equation}
In particular, we have that ${\widetilde
  Z}_n\stackrel{a.s.}\longrightarrow Z_{\infty}$ and
in~\cite{ale-cri-ghi} the following central limit theorem for
$(\widetilde{Z}_{n})_n$ is esta\-bli\-shed:

\begin{theo}\cite[Theorem 4.2]{ale-cri-ghi}\label{thm:asymptotics_Z_tilde}
For $N\geq 1$ and $1/2<\gamma\leq 1$, we have
\begin{equation*}
n^{\gamma-\frac{1}{2}}
\left(\widetilde{Z}_{n}-Z_{\infty}\right)\
{\longrightarrow}\
\mathcal{N}
\left(\ 0\ ,\ \widetilde{\sigma}_{\gamma}^2\,Z_{\infty}(1-Z_{\infty})\ \right)\
\ \ \ stably,
\end{equation*}
where $\widetilde{\sigma}_{\gamma}^2$ is defined as
in~\eqref{def:Sigmatilde_gamma} (also for $\gamma=1$).  The above
convergence is also in the sense of the almost sure conditional
convergence w.r.t.~${\mathcal F}$.
\end{theo}

Concerning the multi-dimensional real stochastic process
$(\widehat{\mathbf{Z}}_n)_n$, we
firstly recall the relation
\begin{equation}\label{rel-tecnica-i}
W^{\top}\widehat{\mathbf{Z}}_{n}\ =\ UDV^{\top}\widehat{\mathbf{Z}}_{n},
\end{equation}
which is due to \eqref{eq:relazioni-2} and
\eqref{eq:decomposition-matrix}, and, moreover, we recall that
from~\cite[Section 4.2]{ale-cri-ghi} we have the dynamics
\begin{equation}\label{eq:dynamic_SA_hat}
\widehat{\mathbf{Z}}_{n+1}-\widehat{\mathbf{Z}}_{n}=
-r_nU(I-D)V^{\top}\widehat{\mathbf{Z}}_{n} + r_n UV^{\top}\,\Delta\mathbf{M}_{n+1}
\end{equation}
and $\widehat{\mathbf{Z}}_n\stackrel{a.s.}\longrightarrow \mathbf{0}$.
\\

\indent Finally, concerning the multi-dimensional real stochastic process
$(\widehat{\mathbf{N}}_n)_n$, using~\eqref{eq:dynamic_SA},
\eqref{eq:decomposition_Z}, \eqref{eq:decomposition_N} and the
assumption $W^{\top}\mathbf{1}=\mathbf{1}$ (which implies
$W^{\top}\mathbf{Z}_n=\widetilde{Z}_n\mathbf{1}+W^{\top}\widehat{\mathbf{Z}}_n$),
we obtain the dynamics:
\begin{equation}\label{eq:dynamic_SA_N_hat}
\widehat{\mathbf{N}}_{n+1}-\widehat{\mathbf{N}}_{n}=
-\frac{1}{n+1}(\widehat{\mathbf{N}}_{n}-W^{\top}\widehat{\mathbf{Z}}_{n}) +
\frac{1}{n+1}
\Delta\mathbf{M}_{n+1}-(\widetilde{Z}_{n+1}-\widetilde{Z}_n)\mathbf{1}.
\end{equation}

\subsection{Proof of Theorem \ref{th:sincro}
(Almost sure synchronization of the empirical means)}

We recall that in \cite[Theorem 3.1]{ale-cri-ghi}, by
decomposition~\eqref{eq:decomposition_Z}, i.e. $\mathbf{Z}_{n}
=\widetilde{Z}_n\mathbf{1}+\widehat{\mathbf{Z}}_{n}$, and combining
${\widetilde Z}_n\stackrel{a.s.}\longrightarrow Z_{\infty}$ and
$\widehat{\mathbf{Z}}_n\stackrel{a.s.}\longrightarrow \mathbf{0}$,
it is proved that
$\mathbf{Z}_{n}\stackrel{a.s.}\longrightarrow Z_{\infty}\mathbf{1}$.
As a consequence, using
$W^{\top}\mathbf{1}=\mathbf{1}$ and~\eqref{eq:dynamic-0}, we obtain
$E[\mathbf{X}_{n}|{\mathcal F}_{n-1}]\stackrel{a.s.}\longrightarrow
Z_{\infty}\mathbf{1}$ and, applying Lemma~\ref{lemma-serie-rv} (with
$c_k=k$, $v_{n,k}=k/n$ and $\eta=1$), we get that $\mathbf{N}_n
\stackrel{a.s.}\longrightarrow\ Z_{\infty}\mathbf{1}$.  This concludes
the proof of the first part of the theorem, concerning the
synchronization result.  For the second part, that is the results on
the limit random variable $Z_\infty$, we refer to
\cite[Theorem 3.5 and Theorem 3.6]{ale-cri-ghi}. \qed \\

\indent Note that, by the synchronization result for $(Z_n)$, we can state that
\begin{equation}\label{eq:multidim-limite-conditional-M}
E[(\Delta{\mathbf M}_{n+1})(\Delta{\mathbf M}_{n+1})^{\top}
\,|\,{\mathcal F}_n]
\stackrel{a.s.}\longrightarrow Z_{\infty}(1-Z_{\infty})I.
\end{equation}
Indeed, since $\{X_{n+1,j}:\, j=1,\dots, N\}$ are conditionally independent
given $\mathcal{F}_n$, we have
\begin{equation}\label{eq:covar-conditional-M}
E[\Delta{M}_{n+1,h}\Delta{M}_{n+1,j}\,|\,{\mathcal F}_n]=0\quad
\mbox{ for } h\neq j;
\end{equation}
while, for each $j$, using the normalization $W^{\top}{\mathbf
  1}={\mathbf 1}$, we have
\begin{equation}\label{eq:var-conditional-M}
E[(\Delta{M}_{n+1,j})^2\,|\,{\mathcal F}_n]=
\left(\sum_{h=1}^N w_{h,j}Z_{n,h}\right)
\left(1-\sum_{h=1}^N w_{h,j}Z_{n,h}\right)
\stackrel{a.s.}\longrightarrow Z_{\infty}(1-Z_{\infty}).
\end{equation}

\subsection{Proof of Theorem \ref{thm:asymptotics_theta_gamma}
(CLT for $({\mathbf Z}_n,{\mathbf N}_n)_n$ in the case $1/2<\gamma<1$)}

In order to prove Theorem \ref{thm:asymptotics_theta_gamma}, we need
to provide the asymptotic behavior of the stochastic processes
$(\widehat{\mathbf{Z}}_{n})_n$ and
$(\widehat{\mathbf{N}}_{n})_n$. First of all, we recall that
$\widehat{\mathbf{Z}}_{n}=0$ for each $n$ when $N=1$ and, for $N\geq
2$ and $1/2<\gamma<1$, we have from \cite[Theorem 4.3]{ale-cri-ghi} that
\begin{equation}\label{eq:CLT_Z_hat_gamma}
n^{\frac{\gamma}{2}}\,\widehat{\mathbf{Z}}_{n}\longrightarrow\
\mathcal{N}\left( \mathbf{0} ,
Z_{\infty}(1-Z_{\infty})\widehat{\Sigma}_{\gamma}\right)
\ \ \ \ \hbox{stably},
\end{equation}
where
\begin{equation*}
\widehat{\Sigma}_{\gamma}:=U\widehat{S}_{\gamma}U^{\top}
\qquad\mbox{ and }\qquad
[\widehat{S}_{\gamma}]_{h,j}:=
\frac{c}{2-(\lambda_h+\lambda_j)} (\mathbf{v}_{h}^{\top}\mathbf{v}_{j})\
\mbox{ with }\ 2\leq h,j\leq N.
\end{equation*}
Moreover, looking at the proof of \eqref{eq:CLT_Z_hat_gamma} in
\cite{ale-cri-ghi}, it is easy to realize that for $N\geq 2$ and
$1/2<\gamma<1$ we have $\lim_n
n^{\gamma}E\left[\,\|\widehat{\mathbf{Z}}_n\|^2\,\right]=C$, where $C$
is a suitable constant in $(0,+\infty)$, and so, recalling that
$\widehat{\mathbf{Z}}_n=0$ for each $n$ when $N=1$, we can affirm
that, for every $N\geq 1$ and $1/2<\gamma<1$, we have that
\begin{equation}\label{eq:L2_Z_hat}
E\left[\,\|\widehat{\mathbf{Z}}_n\|^2\,\right]=O( n^{-\gamma} ).
\end{equation}
Regarding the stochastic process $(\widehat{\mathbf{N}}_{n})_n$, we
are going to prove the following convergence result:

\begin{theo}\label{thm:asymptotics_theta_hat_gamma}
For $N\geq 1$ and $1/2<\gamma<1$, we have that
\begin{equation}\label{eq:CLT_N_hat_gamma}
n^{\gamma -\frac{1}{2}}\,\widehat{\mathbf{N}}_{n}\ \stackrel{d}\longrightarrow\
\mathcal{N}\left(\mathbf{0},
Z_{\infty}(1-Z_\infty)\widehat{\Gamma}_{\gamma}\right)
\ \ \ \ stably,
\end{equation}
where $\widehat{\Gamma}_{\gamma}$ is the matrix defined in
\eqref{def:Sigmahat_NN_gamma}.
\end{theo}
\begin{proof}
We observe that by means of~\eqref{eq:dynamic_SA_N_hat}  we can write
$$n(\widehat{\mathbf{N}}_{n}-\widehat{\mathbf{N}}_{n-1})=
-\widehat{\mathbf{N}}_{n-1}+W^{\top}\widehat{\mathbf{Z}}_{n-1} +
\Delta\mathbf{M}_{n}+n(\widetilde{Z}_{n-1}-\widetilde{Z}_n)\mathbf{1}.$$
Then, using the relation
$$
n(\widehat{\mathbf{N}}_{n}-\widehat{\mathbf{N}}_{n-1})+
\widehat{\mathbf{N}}_{n-1}
\ =\ n\widehat{\mathbf{N}}_{n}-(n-1)\widehat{\mathbf{N}}_{n-1},
$$
we obtain that
$$n\widehat{\mathbf{N}}_{n}=\sum_{k=1}^n
\left[k\widehat{\mathbf{N}}_{k}-(k-1)
  \widehat{\mathbf{N}}_{k-1}\right] =
W^{\top}\sum_{k=1}^n\widehat{\mathbf{Z}}_{k-1}
+
\sum_{k=1}^n\left[\Delta\mathbf{M}_{k}+k(\widetilde{Z}_{k-1}-\widetilde{Z}_k)
  \mathbf{1}\right].
$$
Now, we set $e:=\gamma-1/2>0$ for each
$1/2<\gamma<1$ and hence from the above expression we get
$n^e\widehat{\mathbf{N}}_n = t_n\sum_{k=1}^n {\mathbf T}_{k}+ W^{\top}{\mathbf
  Q}_{n}$, where $t_n:=1/n^{(1-e)}$, ${\mathbf
  Q}_{n}:=t_n\sum_{k=1}^n \widehat{\mathbf{Z}}_{k-1}$
and
$${\mathbf T}_{k}:=
\Delta\mathbf{M}_k+
k\left(\widetilde{Z}_{k-1}-\widetilde{Z}_k\right){\mathbf 1}
=\Delta\mathbf{M}_k
-N^{-1/2}kr_k\left(\mathbf{v}_1^{\top}\Delta\mathbf{M}_k\right)\mathbf{1}.
$$
The idea of the proof is to study separately the two terms
$$
t_n\sum_{k=1}^n {\mathbf T}_{k}\qquad\mbox{and}\qquad {\mathbf Q}_{n}\,.
$$
More precisely, we are going to prove that the first term converges
stably to the desired Gaussian kernel, while the second term converges
in probability to zero.
\\

\noindent {\em First step: the convergence result for $t_n\sum_{k=1}^n {\mathbf
    T}_{k}$.}\\
\noindent We note that $(\mathbf{T}_{k})_{1\leq k\leq n}$ is a
martingale difference array with respect to ${\mathcal F}$. Therefore,
we want to apply Theorem~\ref{thm:triangular} (with $k_n=n$,
$\mathbf{T}_{n,k}=\mathbf{T}_k$ and ${\mathcal G}_{n,k}={\mathcal
  F}_k$). To this purpose, we observe that condition~(c1) is obviously
satisfied and so we have to prove only conditions~(c2) and~(c3).
\\

\indent Regarding condition~(c2), we note that

\begin{multline*}
\sum_{k=1}^n {\mathbf T}_k{\mathbf T}_k^{\top}
=
\sum_{k=1}^n \Delta\mathbf{M}_k(\Delta\mathbf{M}_k)^{\top}
+
N^{-1}\sum_{k=1}^n k^2r_k^2\left(\mathbf{v}_1^{\top}\Delta\mathbf{M}_k\right)^2
\mathbf{1}\mathbf{1}^{\top}\\
-N^{-1/2}\sum_{k=1}^n kr_k \left(\mathbf{v}_1^{\top}\Delta\mathbf{M}_k\right)
\Delta\mathbf{M}_k\mathbf{1}^{\top}
-N^{-1/2}\sum_{k=1}^n kr_k \left(\mathbf{v}_1^{\top}\Delta\mathbf{M}_k\right)
\mathbf{1}(\Delta\mathbf{M}_k)^{\top}.
\end{multline*}
The convergence rate of each of the four terms will be determined in the
following.
\\

\indent By~\eqref{eq:multidim-limite-conditional-M}
and~Lemma~\ref{lemma-serie-rv}
(with $c_k=k$, $v_{n,k}=k/n$ and $\eta=1$),
for the first term, we obtain that
\begin{equation*}
n^{-1}\sum_{k=1}^n \Delta\mathbf{M}_k(\Delta\mathbf{M}_k)^{\top}
\stackrel{a.s.}\longrightarrow  Z_{\infty}(1-Z_{\infty})I\,.
\end{equation*}
Moreover, regarding the second term, by \eqref{relazione-nota} we
have that
\begin{equation*}
\lim_n n^{-2(1-e)} \sum_{k=1}^n k^2 r_k^2=
c^2 \lim_n n^{-2(1-e)} \sum_{k=1}^n \frac{1}{k^{1-2(1-e)}}=
\frac{c^2}{2(1-e)}
\end{equation*}
and, since by \eqref{eq:covar-conditional-M} and
\eqref{eq:var-conditional-M} we have that
\begin{equation*}
E\left[\left(\mathbf{v}_1^{\top}\Delta\mathbf{M}_k\right)^2\,|\,
{\mathcal F}_{k-1}\right]
=\sum_{j=1}^N v_{1,j}^2 E\left[(\Delta M_{k,j})^2|{\mathcal F}_{k-1}\right]
\stackrel{a.s.}\longrightarrow
\|\mathbf{v}_1\|^2 Z_{\infty}(1-Z_{\infty})\,,
\end{equation*}
by~Lemma~\ref{lemma-serie-rv} again (with $c_k=k$,
$v_{n,k}=k^3r_k^2/n^{2(1-e)}$ and $\eta=\frac{c^2}{2(1-e)}$), we
obtain that
\begin{equation*}
n^{-2(1-e)}\,N^{-1}
\sum_{k=1}^n k^2r_k^2\left(\mathbf{v}_1^{\top}\Delta\mathbf{M}_k\right)^2
\mathbf{1}\mathbf{1}^{\top}
\stackrel{a.s.}\longrightarrow
\frac{c^2}{2(1-e)N}\|\mathbf{v}_1\|^2 Z_{\infty}(1-Z_{\infty})
\mathbf{1}\mathbf{1}^{\top}\,.
\end{equation*}
Furthermore, concerning the third term, by \eqref{relazione-nota} we
have that
\begin{equation*}
\lim_n n^{-(1+\frac{1}{2}-e)}\sum_{k=1}^n k r_k=
c\lim_n n^{-(1+\frac{1}{2}-e)}\sum_{k=1}^n k^{\frac{1}{2}-e}=
\frac{c}{1+\frac{1}{2}-e}\,.
\end{equation*}
On the other hand, by means of~\eqref{eq:covar-conditional-M} and
\eqref{eq:var-conditional-M}, we have that
\begin{equation*}
E\Big[\Big(\mathbf{v}_1^{\top}\Delta\mathbf{M}_k\Big)
\Delta\mathbf{M}_k\mathbf{1}^{\top}\,|\,
{\mathcal F}_{k-1}\Big]
=E\Big[\Big(\sum_{j=1}^N v_{1,j}\Delta M_{k,j}\Big)
\Delta\mathbf{M}_k\mathbf{1}^{\top}\,|\,
{\mathcal F}_{k-1}\Big]
\stackrel{a.s.}\longrightarrow
\mathbf{v}_1\mathbf{1}^{\top} Z_{\infty}(1-Z_{\infty}),
\end{equation*}
and so, by~Lemma~\ref{lemma-serie-rv} again (with $c_k=k$,
$v_{n,k}=k r_k/n^{1+\frac{1}{2}-e}$ and $\eta=\frac{c}{(1+1/2-e)}$),
it follows
\begin{equation*}
 n^{-(1+\frac{1}{2}-e)}\,
N^{-1/2}\sum_{k=1}^n kr_k \left(\mathbf{v}_1^{\top}\Delta\mathbf{M}_k\right)
\Delta\mathbf{M}_k\mathbf{1}^{\top}
\stackrel{a.s.}\longrightarrow
\frac{c}{(1+1/2-e)\sqrt{N}}Z_{\infty}(1-Z_{\infty})\mathbf{v}_1\mathbf{1}^{\top}\,.
\end{equation*}
Finally, for the convergence of the fourth term, we can argue as we
have just done for the third one. Indeed, observing that,
by~\eqref{eq:covar-conditional-M} and \eqref{eq:var-conditional-M}, we
have that
\begin{equation*}
E\Big[\Big(\mathbf{v}_1^{\top}\Delta\mathbf{M}_k\Big)
\mathbf{1}(\Delta\mathbf{M}_k)^{\top}\,|\,
{\mathcal F}_{k-1}\Big]
=E\Big[\mathbf{1}\Big(\sum_{j=1}^N v_{1,j}\Delta M_{k,j}\Big)
(\Delta\mathbf{M}_k)^{\top}\,|\,
{\mathcal F}_{k-1}\Big]
\stackrel{a.s.}\longrightarrow
\mathbf{1}\mathbf{v}_1^{\top} Z_{\infty}(1-Z_{\infty})\,,
\end{equation*}
we get
\begin{equation*}
 n^{-(1+\frac{1}{2}-e)}\,N^{-1/2}
\sum_{k=1}^n kr_k \left(\mathbf{v}_1^{\top}\Delta\mathbf{M}_k\right)
\mathbf{1}(\Delta\mathbf{M}_k)^{\top}
\stackrel{a.s.}\longrightarrow
\frac{c}{(1+1/2-e)\sqrt{N}} Z_{\infty}(1-Z_{\infty})\mathbf{1}\mathbf{v}_1^{\top}\,.
\end{equation*}
Summing up, since for $1/2<\gamma<1$ we have $2(1-e)>1$ and
$2(1-e)>1+1/2-e$, we obtain that
\begin{equation}\label{var-asin-gamma}
\begin{split}
t_n^2\sum_{k=1}^n {\mathbf T}_k{\mathbf T}_k^{\top}
=
\frac{1}{n^{2(1-e)}}\sum_{k=1}^n {\mathbf T}_k{\mathbf T}_k^{\top}
&\stackrel{a.s.}\longrightarrow
0+\frac{c^2}{N}\|\mathbf{v}_1\|^2 \frac{1}{2(1-e)} Z_{\infty}(1-Z_{\infty})
\mathbf{1}\mathbf{1}^{\top}
-0-0
\\
&= Z_{\infty}(1-Z_\infty)\widehat{\Gamma}_{\gamma}\,.
\end{split}
\end{equation}
\indent Regarding condition~(c3), we note that
\begin{equation*}
t_n \sup_{1\leq k\leq n} |\mathbf{T}_k|=
\frac{1}{n^{1-e}} \sup_{1\leq k\leq n} O(k^{1-\gamma})
= O(1/n^{\gamma-e})=O(1/\sqrt{n})\longrightarrow 0\,.
\end{equation*}
Therefore also this condition is satisfied and we can conclude that
$t_n\sum_{k=1}^n {\mathbf T}_{k}$ converges stably to the Gaussian
kernel with mean zero and random covariance matrix given by
\eqref{var-asin-gamma}.
\\

\noindent{\em Second step: the convergence result for $\mathbf{Q}_n$.}\\
\noindent We aim at proving that $\mathbf{Q}_n$ converges in
probability to zero, that is each component $Q_{n,j}$ converges in
probability to zero. To this purpose, we note that
\begin{equation*}
E\left[\,|Q_{n,j}|\,\right]
\leq
t_n\sum_{k=1}^n
E\left\{\, | \widehat{Z}_{k-1,j} |\, \right\}
\leq
t_n\sum_{k=1}^n
\sqrt{
E\left[\, (\widehat{Z}_{k-1,j})^2 \,\right]
}
\leq t_n\sum_{k=1}^n
\sqrt{
E\left[\, \| \widehat{\mathbf{Z}}_{k-1} \|^2 \,\right]
}
\,.
\end{equation*}

Therefore, recalling that, for $1/2<\gamma<1$, we have $E\left[
  \|\widehat{\mathbf{Z}}_n\|^2\right]=O(n^{-\gamma})$ (see
\eqref{eq:L2_Z_hat}), we can conclude by \eqref{relazione-nota} that
$$
E[\,|Q_{n,j}|\,]=O\Big(t_n\sum_{k=1}^n k^{-\gamma/2}\Big)
\!=\!
O\Big(n^{-(1-e)}\sum_{k=1}^n \frac{1}{k^{1-(1-\gamma/2)}}\Big)
\!=\!
O\Big(n^{-1+e+1-\gamma/2}\Big)
\!=\!
O\Big(\frac{1}{n^{(1-\gamma)/2}}\Big)
\!\to\! 0\,.
$$
\end{proof}

Now, the proof of Theorem \ref{thm:asymptotics_theta_gamma} follows
from the previous result, together with
Theorem~\ref{thm:asymptotics_Z_tilde} and Theorem~\ref{blocco}.
\\

{\it Proof of Theorem \ref{thm:asymptotics_theta_gamma}.} By
Theorem~\ref{thm:asymptotics_Z_tilde}, we have that
\begin{equation*}
n^{\gamma -\frac{1}{2}}(\widetilde{Z}_n-Z_\infty)\mathbf{1}
\longrightarrow {\mathcal N}\big(\mathbf{0},
Z_{\infty}(1-Z_{\infty})\widetilde{\Sigma}_\gamma\big)
\ \ \ \hbox{stably in the strong sense}.
\end{equation*}
Thus, from Theorem \ref{thm:asymptotics_theta_hat_gamma}, applying
Theorem~\ref{blocco}, we obtain that
$$
n^{\gamma -\frac{1}{2}}
\left({\mathbf N}_n-\widetilde{Z}_n{\mathbf 1},
(\widetilde{Z}_n-Z_\infty)\mathbf{1}
\right)
\longrightarrow
\mathcal{N}\left(\mathbf{0},
Z_\infty(1-Z_\infty)\widehat{\Gamma}_\gamma\right)
\otimes
{\mathcal N}\left(\mathbf{0},
Z_{\infty}(1-Z_{\infty})\widetilde{\Sigma}_\gamma\right)
\ \ \ \hbox{stably}.
$$
In order to conclude, it is enough to observe that
$$
n^{\gamma -\frac{1}{2}}
\left(
\mathbf{Z}_n-Z_{\infty}\mathbf{1},
{\mathbf N}_n-Z_{\infty}\mathbf{1}
\right)
=\Phi\left(
n^{\gamma -\frac{1}{2}}
(\mathbf{N}_n-\widetilde{Z}_n\mathbf{1}),
n^{\gamma -\frac{1}{2}}
(\widetilde{Z}_n-Z_{\infty})\mathbf{1}
\right)
+\frac{1}{ n^{(1-\gamma)/2} }
\left(
n^{ \frac{\gamma}{2} } \widehat{\mathbf Z}_n, \mathbf{0}
\right)\,,
$$
where $\Phi(x,y)=(y, x+y)$ and the last term converges in
probability to zero (since $\widehat{\mathbf Z}_n=0$ for each $n$ when
$N=1$ and since \eqref{eq:CLT_Z_hat_gamma} when $N\geq 2$).
\qed

\begin{rem}\label{rem:O_gamma_2}
\rm
With reference to the statistical applications discussed in
Subsection \ref{section_statistics}, we recall that, since
$V^{\top}\mathbf{1}=\mathbf{0}$ (by \eqref{eq:relazioni-2}), we have
$UV^{\top}\mathbf{N}_n=UV^{\top}\widehat{\mathbf{N}}_n$ and
$V^{\top}\widehat{\Gamma}_{\gamma}V$ is the null matrix, and so
from~\eqref{eq:CLT_N_hat_gamma} we can get that $n^{\gamma
  -\frac{1}{2}}UV^{\top}\mathbf{N}_{n}\stackrel{P}\to \mathbf{0}$ for
$1/2<\gamma<1$.  More precisely, following the arguments in the proof
of Theorem~\ref{thm:asymptotics_theta_hat_gamma}, it is possible to
show that, when $1/2<\gamma<1$, we have
$n^{e}UV^{\top}\mathbf{N}_{n}\stackrel{P}\to \mathbf{0}$ for each $e<\gamma/2$.
Indeed, from \eqref{eq:dynamic_SA_N_hat}, together with
\eqref{rel-tecnica-i} and again the relation
$V^{\top}\mathbf{1}=\mathbf{0}$, we obtain
$$ n(UV^{\top}\mathbf{N}_{n}-UV^{\top}\mathbf{N}_{n-1})=
-UV^{\top}\mathbf{N}_{n-1}+W^{\top}\widehat{\mathbf{Z}}_{n-1} +
UV^{\top}\Delta\mathbf{M}_{n}$$ and hence, setting $t_n:=1/n^{1-e}$,
$\mathbf{T}_k:=UV^{\top}\Delta\mathbf{M}_k$ and
$\mathbf{Q}_n:=t_n\sum_{k=1}^n\widehat{\mathbf{Z}}_{k-1}$, we get
$$
n^e UV^{\top}\mathbf{N}_n=t_n\sum_{k=1}^n \mathbf{T}_k+ W^{\top}\mathbf{Q}_n
=\frac{1}{ n^{\frac{1}{2}-e} }
\frac{1}{\sqrt{n}}\sum_{k=1}^n T_k+ W^{\top}\mathbf{Q}_n,
$$ where $\frac{1}{\sqrt{n}}\sum_{k=1}^n \mathbf{T}_k$ converges
stably to the Gaussian kernel $\mathcal{N}(\mathbf{0},
Z_{\infty}(1-Z_\infty)UV^{\top}VU^{\top})$ and
$E[\,|\mathbf{Q}_n|\,]=O( t_nn^{1-\frac{\gamma}{2}} )=O( n^{-(
  \frac{\gamma}{2}-e)} ).$ From these relations, we can also conclude
that for $1/2<\gamma<1$ and $e=\gamma/2$, we have that $n^e
UV^{\top}\mathbf{N}_n$ is the sum of a term converging to zero in
probability and a term bounded in $L^1$. Therefore the asymptotic
behavior of $n^{\gamma/2}UV^{\top}\mathbf{N}_n$ needs further
investigation.
\end{rem}

\subsection{Proof of Theorem \ref{thm:N_1_gamma_1}
(CLT for $({\mathbf Z}_n,{\mathbf N}_n)_n$ in the case $N=1$ and $\gamma=1$)}

The proof in the case $N=1$ and $\gamma=1$ is similar to the one for
$1/2<\gamma<1$. Indeed, using the same arguments as in the proof of
Theorem \ref{thm:asymptotics_theta_hat_gamma}, together with the facts
that $\widetilde{Z}_n=Z_n$, $\widehat{\mathbf Z}_n=0$ for each $n$,
$\mathbf{v}_1=v_{1,1}=1$ and $2(1-e)=1+1/2-e=1$, we obtain that
\begin{equation*}
\sqrt{n}(N_n-Z_n)=\sqrt{n}\,\widehat{N}_{n}
\ \longrightarrow\
\mathcal{N}\left(0,  Z_{\infty}(1-Z_\infty)(c-1)^2\right)
\ \ \ \ \hbox{stably}.
\end{equation*}
On the other hand, by Theorem~\ref{thm:asymptotics_Z_tilde}, we have that
\begin{equation*}
\sqrt{n}(Z_n-Z_\infty)=\sqrt{n}(\widetilde{Z}_n-Z_\infty)
\longrightarrow {\mathcal N}\big(0,
Z_{\infty}(1-Z_{\infty})c^2\big)
\ \ \ \hbox{stably in the strong sense}.
\end{equation*}
Thus, applying Theorem~\ref{blocco}, we obtain
$$
\sqrt{n}
\left(N_n-Z_n, Z_n-Z_\infty\right)
\longrightarrow
\mathcal{N}\left(0,
Z_\infty(1-Z_\infty) (c-1)^2\right)
\otimes
{\mathcal N}\left(0,
Z_{\infty}(1-Z_{\infty})c^2\right)
\ \ \ \hbox{stably}.
$$
In order to conclude, it is enough to observe that
$$
\sqrt{n}
\left(Z_n-Z_{\infty}, N_n-Z_{\infty}\right)
=\Phi\left(
\sqrt{n}(N_n-Z_n), \sqrt{n}(Z_n-Z_{\infty})
\right)\,,
$$
where $\Phi(x,y)=(y, x+y)$.
\qed

\begin{rem}\label{rem:different_approach_gamma_1}
\rm Looking at the arguments of the proof of
Theorem~\ref{thm:asymptotics_theta_hat_gamma} with $N\geq 2$ and
$\gamma=1$, we find
$E\left[\,|\mathbf{Q}_n|\,\right]=O\left(\frac{1}{n^{(1-\gamma)/2}}\right)=O(1)$
and so, from this relation, we can not conclude that $\mathbf{Q}_n$
converges to zero in probability. Therefore part of the proof of
Theorem~\ref{thm:asymptotics_theta_hat_gamma} does not work when
$N\geq 2$ and $\gamma=1$. Moreover, since $\mathbf{Q}_n=\sum_{k=1}^n
\widehat{\mathbf{Z}}_{k-1}/\sqrt{n}$ and, from \cite[Theorem
  4.3]{ale-cri-ghi}, we know that, when $N\geq 2$ and $\gamma=1$, the
rate of convergence of $\widehat{\mathbf{Z}}_n$ is $\sqrt{n}$ or
$\sqrt{n/\ln(n)}$ according to the value of ${\mathcal
  Re}(\lambda^*)$, we may conjecture that, for $N\geq 2$ and
$\gamma=1$, $\mathbf{Q}_{n}$ generally does not converge in
probability to zero.  This fact leads us to a complete different
approach to the proofs of Theorem~\ref{thm:asymptotics_theta_1} and
Theorem~\ref{thm:asymptotics_Z_1_star} concerning the case $N\geq 2$
and $\gamma=1$, that will be developed in the next sections.
\end{rem}

\subsection{Proof of Theorem \ref{thm:asymptotics_theta_1}
(CLT for $({\mathbf Z}_n,{\mathbf N}_n)_n$ in the case $N\geq 2$, $\gamma=1$
and ${\mathcal Re}(\lambda^{*})<1-(2c)^{-1}$)}

In order to prove Theorem \ref{thm:asymptotics_theta_1}, we need the
following convergence result on $(\widehat{\mathbf{Z}}_n,
\widehat{\mathbf{N}}_n)_n$:

\begin{theo}\label{thm:asymptotics_theta_hat_1}
Let $N\geq 2$, $\gamma=1$ and ${\mathcal R}e(\lambda^{*})<1-(2c)^{-1}$.
Then, under condition~\eqref{ass:condition_r_n_1}, we have that
\begin{equation*}
\sqrt{n}
\begin{pmatrix}
\widehat{\mathbf{Z}}_n\\
\widehat{\mathbf{N}}_n
\end{pmatrix}
{\longrightarrow}\
\mathcal{N} \left(\ \mathbf{0}\ ,\ Z_{\infty}(1-Z_{\infty})
\begin{pmatrix}
\widehat{\Sigma}_{\mathbf{ZZ}} & \widehat{\Sigma}_{\mathbf{ZN}}\\
\widehat{\Sigma}_{\mathbf{ZN}}^{\top} & \widehat{\Sigma}_{\mathbf{NN}}
\end{pmatrix}\ \right)\ \ \ \ stably,
\end{equation*}
where $\widehat{\Sigma}_{\mathbf{ZZ}}$, $\widehat{\Sigma}_{\mathbf{NN}}$ and
$\widehat{\Sigma}_{\mathbf{ZN}}$ are the matrices defined
in~\eqref{def:Sigmahat_ZZ_1},~\eqref{def:Sigmahat_NN_1}
and~\eqref{def:Sigmahat_ZN_1}, respectively.
\end{theo}

\begin{proof}
First we use \eqref{rel-tecnica-i} in~\eqref{eq:dynamic_SA_N_hat} and
we replace the term $(\widetilde{Z}_n-\widetilde{Z}_{n-1})$
in~\eqref{eq:dynamic_SA_N_hat} as shown
in~\eqref{eq:dynamic_SA_tilde}, so that we obtain
$$\widehat{\mathbf{N}}_n-\widehat{\mathbf{N}}_{n-1}=
\frac{1}{n}(-\widehat{\mathbf{N}}_{n-1}
+UDV^{\top}\widehat{\mathbf{Z}}_{n-1}+\Delta\mathbf{M}_n)\
-r_{n-1}N^{-1/2}\mathbf{v}_1^{\top}\Delta\mathbf{M}_n\mathbf{1}.
$$
Then, if we define the remainder term as
\begin{equation}\label{def:R_n}
\mathbf{R}_n\ :=\
\Big(\frac{1}{nr_{n-1}}-\frac{1}{c}\Big)
(-\widehat{\mathbf{N}}_{n-1}
+UDV^{\top}\widehat{\mathbf{Z}}_{n-1}+\Delta\mathbf{M}_n),
\end{equation}
we can rewrite the above dynamics of $\widehat{\mathbf{N}}_n$ as follows:
\begin{equation}\label{eq:dynamic_SA_N_hat_proof}
\widehat{\mathbf{N}}_n=(1-r_{n-1}c^{-1})\widehat{\mathbf{N}}_{n-1}
+r_{n-1} c^{-1}UDV^{\top}\widehat{\mathbf{Z}}_{n-1}
+r_{n-1}[c^{-1}I-N^{-1/2}\mathbf{1}\mathbf{v}_1^{\top}]\Delta\mathbf{M}_n
+r_{n-1}\mathbf{R}_n.
\end{equation}
Then, setting
$\boldsymbol{\theta}_n:=
(\widehat{\mathbf{Z}}_{n},\widehat{\mathbf{N}}_{n})^{\top}$,
$\Delta\mathbf{M}_{\theta,n}:=(\Delta\mathbf{M}_n,\Delta\mathbf{M}_n)^{\top}$
and $\mathbf{R}_{\theta,n}:=(\mathbf{0},\mathbf{R}_n)^{\top}$, which
are vectors of dimension $2N$, and combining~\eqref{eq:dynamic_SA_hat}
and~\eqref{eq:dynamic_SA_N_hat_proof}, we can write
$$
\boldsymbol{\theta}_{n+1}=
(I-r_nQ)\boldsymbol{\theta}_{n}+
r_n(R\Delta\mathbf{M}_{\theta,n+1}+\mathbf{R}_{\theta,n+1}),
$$
where
\begin{equation*}
Q:=\begin{pmatrix}
U(I-D)V^{\top} & 0 \\
-c^{-1}UDV^{\top} & c^{-1}I
\end{pmatrix},
\end{equation*}
and (recalling~that $\mathbf{u}_1=N^{-1/2}\mathbf{1}$ and
$I=\mathbf{u}_1\mathbf{v}_1^{\top}+UV^{\top}$ by
\eqref{eq:relazioni-1} and~\eqref{eq:relazioni-2})
\begin{equation}\label{def:R-irene}
R:=\begin{pmatrix}
UV^{\top} & 0\\
0 & (c^{-1}-1)\mathbf{u}_1\mathbf{v}_1^{\top}+c^{-1}UV^{\top}
\end{pmatrix}.
\end{equation}
Now, we will prove that $\sqrt{n}\boldsymbol{\theta}_{n}$ converges stably
to the desired Gaussian kernel.  To this end, the first step is to define
the $(2N)\times(2N-1)$ matrices
\[U_{\theta}:=\begin{pmatrix}
U & 0 \\
0 & \widetilde{U}
\end{pmatrix}=\begin{pmatrix}
U & \mathbf{0} & 0 \\
0 & \mathbf{u}_1 & U
\end{pmatrix}
\qquad\mbox{and}\qquad
V_{\theta}:=\begin{pmatrix}
V & 0  \\
0 & \widetilde{V}
\end{pmatrix}=\begin{pmatrix}
V & \mathbf{0} & 0 \\
0 & \mathbf{v}_1 & V
\end{pmatrix},\]
and observe that from~\eqref{eq:relazioni-2} we have
$V_{\theta}^{\top}U_{\theta}=I$ and
\begin{equation*}
U_{\theta}V_{\theta}^{\top}=
\begin{pmatrix}
UV^{\top} & 0  \\
0 & I
\end{pmatrix}.
\end{equation*}
Then, defining the $(2N)\times(2N-1)$ matrices
\begin{equation}\label{def:S_Q_S_R}
S_Q:=\begin{pmatrix}
(I-D) & \mathbf{0} & 0\\
\mathbf{0}^{\top} & c^{-1} & \mathbf{0}^{\top}\\
-c^{-1}D & \mathbf{0} & Ic^{-1}
\end{pmatrix}
\qquad\hbox{and}\qquad
S_R:=\begin{pmatrix}
I & \mathbf{0} & 0\\
\mathbf{0}^{\top} & c^{-1}-1 & \mathbf{0}^{\top}\\
0 & \mathbf{0} & c^{-1}I
\end{pmatrix},
\end{equation}
we have that $Q=U_{\theta}S_QV_{\theta}^{\top}$
and $R=U_{\theta}S_R V_{\theta}^{\top}$.
From the above relations on $U_{\theta}$ and $V_{\theta}$,
we get that
$U_{\theta}V_{\theta}^{\top}\boldsymbol{\theta}_n=\boldsymbol{\theta}_n$ and hence
we can write
$$
\boldsymbol{\theta}_{n+1}=
U_{\theta}\left[I-r_nS_Q\right]V_{\theta}^{\top}\boldsymbol{\theta}_n
+r_nR\Delta\mathbf{M}_{\theta, n+1}+r_n\mathbf{R}_{\theta,n+1}.
$$ Let us now set $\alpha_j:=1-\lambda_j\in\mathbb{C}$ with
$\lambda_j\in Sp(W)\setminus \{1\}=Sp(D)$ and recall that ${\mathcal
  Re}(\alpha_j)>0$ for each $j$ since ${\mathcal Re}(\lambda_j)<1$ for
each $j$.  Then, if we take an integer $m_0\geq 2$ large enough such
that ${\mathcal Re}(\alpha_j) r_n<1$ for all $j$ and $n\geq m_0$, we
can write
\begin{equation}\label{eq:theta-irene}
\boldsymbol{\theta}_{n+1}=
C_{m_0,n}\boldsymbol{\theta}_{m_0}+
\sum_{k=m_0}^nC_{k+1,n}r_kR\Delta\mathbf{M}_{\theta,k+1}+
\sum_{k=m_0}^nC_{k+1,n}r_k\mathbf{R}_{\theta,k+1}
\quad\hbox{for } n\geq m_0,
\end{equation}
where
\begin{equation}\label{def:C_k_n_A_k_n}
\begin{split}
&C_{k+1,n}:=U_{\theta}A_{k+1,n}V_{\theta}^{\top}
\quad\hbox{for } m_0-1\leq k\leq n,
\\
&A_{k+1,n}:=
\begin{cases}
\prod_{m=k+1}^n\left[I-r_mS_Q\right]=\begin{pmatrix}
A^{11}_{k+1,n} & \mathbf{0} & 0\\
\mathbf{0}^{\top} & a^{22}_{k+1,n} & \mathbf{0}^{\top}\\
A^{31}_{k+1,n} & \mathbf{0} & A^{33}_{k+1,n}
\end{pmatrix}
\;&\hbox{for } m_0\!-\!1\!\leq\! k\!\leq\! n\!-\!1
\\
I\;&\hbox{for } k\!=\!n.
\end{cases}
\end{split}
\end{equation}
Notice that the blocks $A^{11}_{k+1,n}$, $A^{31}_{k+1,n}$ and
$A^{33}_{k+1,n}$ are all diagonal $(N-1)\times(N-1)$ matrices.  In
particular, setting for any $x\in\mathbb{C}$, $p_{m_0-1}(x):=1$ and
$p_{k}(x):=\prod_{m=m_0}^{k}(1-r_mx)$ for $k\geq m_0$ and
$F_{k+1,n}(x):=p_{n}(x)/p_{k}(x)$ for $m_0-1\leq k\leq n-1$, from
Lemma~\ref{lem:terms_in_A_k_n} we get
\begin{equation}\label{def:terms_in_A_k_n}\begin{aligned}
&[A^{11}_{k+1,n}]_{jj}\ =\ F_{k+1,n}(\alpha_j),\\
&[A^{33}_{k+1,n}]_{jj}\ =\ a^{22}_{k+1,n}\ =\ F_{k+1,n}(c^{-1}),\\
&[A^{31}_{k+1,n}]_{jj}\ =\ \left\{\begin{aligned}
&\left(\frac{1-\alpha_j}{c\alpha_j-1}\right)
(F_{k+1,n}(c^{-1})-F_{k+1,n}(\alpha_j)),\ &\hbox{for } c\alpha_j\neq 1,\\
&(1-c^{-1})F_{k+1,n}(c^{-1})\ln\left(\frac{n}{k}\right)+O(n^{-1}),\
&\hbox{for } c\alpha_j=1.
\end{aligned}\right.
\end{aligned}\end{equation}

Finally, we rewrite \eqref{eq:theta-irene} as
\begin{equation}\label{eq:dynamic_SA_theta_proof}
\boldsymbol{\theta}_{n+1}=C_{m_0,n}\boldsymbol{\theta}_{m_0}+
\sum_{k=m_0}^n {\mathbf T}_{n,k}+\boldsymbol{\rho}_{n},\qquad\mbox{where}
\qquad \left\{
\begin{aligned}
&{\mathbf T}_{n,k}:=r_k C_{k+1,n} R\Delta{\mathbf M}_{\theta,k+1},\\
&\boldsymbol{\rho}_{n}:=\sum_{k=m_0}^nr_k C_{k+1,n} {\mathbf R}_{\theta,k+1}.
\end{aligned}\right.\end{equation}
and, in the sequel of the proof, we will establish the asymptotic
behavior of $\boldsymbol{\theta}_{n}$ by studying separately the terms
$C_{m_0,n}\boldsymbol{\theta}_{m_0}$, $\sum_{k=m_0}^n {\mathbf
  T}_{n,k}$ and $\boldsymbol{\rho}_{n}$.\\

\indent Let us use the symbol $^*$ for the quantities
$a_{\alpha_j}:={\mathcal Re}(\alpha_j)$ and $p_{n}(\alpha_j)$
corresponding to $\alpha^*=\alpha_j=1-\lambda_j$ with
$\lambda_j=\lambda^*\in \lambda_{\max}(D)$. Now, we note that, as a
consequence of \eqref{def:C_k_n_A_k_n}, \eqref{def:terms_in_A_k_n}
and Lemma~\ref{lemma-tecnico_1}, we have
\begin{equation}\label{eq-corretta}
\begin{split}
|A_{k+1,n}|&=O\left(\frac{|p_n^*|}{|p_k^*|}\right)
+O\left(\frac{k}{n}\right)+O\left(\frac{k}{n}\ln\left(\frac{n}{k}\right)\right)
+O\left(\frac{1}{n}\right)\\
&\hbox{(where the last two terms are present when
there exists $j$ such that $c\alpha_j=1$)}
\\
&=O\left(\frac{|p_n^*|}{|p_k^*|}\right)
+O\left(\frac{k}{n}\ln(n)\right)
=
O\left( \Big(\frac{k}{n}\Big)^{ca^*} \right)
+O\left(\frac{k}{n}\ln(n)\right)
\quad\hbox{for } m_0-1\leq k\leq n-1.
\end{split}
\end{equation}
Therefore, we get that $\sqrt{n}
|C_{m_0,n}\boldsymbol{\theta}_{m_0}|\rightarrow 0$ almost surely 
%
%
 because $ca^*>1/2$ by assumption.\\

\indent Concerning the term $\boldsymbol{\rho}_{n}$, notice that
by~\eqref{ass:condition_r_n_1} and~\eqref{def:R_n} we have that
$|\mathbf{R}_k|=O(k^{-1})$ and, by \eqref{eq-corretta}, we have that
\begin{equation*}
|C_{k+1,n}|\ =\  O\Big(\ \Big(\frac{k}{n}\Big)^{ca^*}\ \Big)
+O\Big(\ \frac{k}{n}\ln(n)\ \Big)
\quad\hbox{for } m_0\leq k\leq n-1.
\end{equation*}
Therefore, since $\boldsymbol{\rho}_{n}=\sum_{k=m_0}^nr_k C_{k+1,n}
{\mathbf R}_{\theta,k+1}=\sum_{k=m_0}^{n-1} r_k C_{k+1,n}{\mathbf
  R}_{\theta,k+1} +r_nC_{n+1,n}{\mathbf R}_{\theta,n+1}$, it follows (by
\eqref{relazione-nota}) that
\begin{equation*}
\sqrt{n}|\boldsymbol{\rho}_{n}|\ =\
O\Big( n^{1/2-c a^*}
\sum_{k=m_0}^{n-1} k^{-(2-c a^*)}\Big)
+
O\big( n^{-1/2} \ln(n)^2 \big)
+
O\big(n^{-3/2}\big)
\longrightarrow 0\quad \hbox{a.s.}\,,
\end{equation*}
because $ca^*>1/2$.
\\

\indent We now focus on the asymptotic behavior of the second term.
Specifically, we aim at pro\-ving that $\sqrt{n}\sum_{k=m_0}^n{\mathbf
  T}_{n,k}$ converges stably to the suitable Gaussian kernel.  For this
purpose, we set ${\mathcal G}_{n,k}={\mathcal F}_{k+1}$, and consider
Theorem~\ref{thm:triangular} (recall that ${\mathbf T}_{n,k}$ are real
random vectors).  Given the fact that condition $(c1)$ of
Theorem~\ref{thm:triangular} is obviously satisfied, we will check
only conditions $(c2)$ and $(c3)$.  \\

Regarding condition $(c2)$, since the relation
$V_{\theta}^{\top}U_{\theta}=I$ implies $V_{\theta}^{\top}R=S_R
V_{\theta}^{\top}$, we have that
\begin{align*}
\sum_{k=m_0}^n (\sqrt{n}\mathbf{T}_{n,k})(\sqrt{n}\mathbf{T}_{n,k})^{\top}
& = n\sum_{k=m_0}^n r_k^2
C_{k+1,n}R
(\Delta{\mathbf M}_{\theta,k+1})(\Delta{\mathbf M}_{\theta,k+1})^{\top}RC_{k+1,n}^{\top}
\\
&= U_{\theta} \left(n\sum_{k=m_0}^n r_k^2
A_{k+1,n}\,V_{\theta}^{\top}\,R
(\Delta{\mathbf M}_{\theta,k+1})(\Delta{\mathbf M}_{\theta,k+1})^{\top}R\,
V_{\theta}\,A^{\top}_{k+1,n}
\right) U_{\theta}^{\top}\\
&= U_{\theta} \left(n\sum_{k=m_0}^n r_k^2
A_{k+1,n}S_R\,V_{\theta}^{\top}\,(\Delta{\mathbf M}_{\theta,k+1})
(\Delta{\mathbf M}_{\theta,k+1})^{\top}\,
V_{\theta}\, S_RA^{\top}_{k+1,n}
\right) U_{\theta}^{\top}.
\end{align*}
Therefore, it is enough to study the convergence of
\begin{equation*}
n\sum_{k=m_0}^n r_k^2
A_{k+1,n}S_R\,V_{\theta}^{\top}\,
(\Delta{\mathbf M}_{\theta,k+1})(\Delta{\mathbf M}_{\theta,k+1})^{\top}\,
V_{\theta}\, S_RA^{\top}_{k+1,n}.
\end{equation*}
Moreover, since $O(nr_n^2)=O(n^{-1})\to 0$ the last term in the above sum is negligible as $n$ increase to infinity,
and hence it is enough to study the convergence of
\begin{equation}\label{step_proof-irene}
n\sum_{k=m_0}^{n-1} r_k^2
A_{k+1,n}S_R\,V_{\theta}^{\top}\,
(\Delta{\mathbf M}_{\theta,k+1})(\Delta{\mathbf M}_{\theta,k+1})^{\top}\,
V_{\theta}\, S_RA^{\top}_{k+1,n}.
\end{equation}
To this purpose, setting
$B_{\theta,k+1}:=V_{\theta}^{\top}\,(\Delta{\mathbf
    M}_{\theta,k+1})(\Delta{\mathbf
    M}_{\theta,k+1})^{\top}\,V_{\theta}$,
$B_{k+1}:=V^{\top}\,(\Delta{\mathbf M}_{k+1})(\Delta{\mathbf
    M}_{k+1})^{\top}\,V$,
$\mathbf{b}_{k+1}:=V^{\top}\,(\Delta{\mathbf M}_{k+1})(\Delta{\mathbf
    M}_{k+1})^{\top}\,\mathbf{v}_1$ and
$b_{k+1}:=\mathbf{v}_1^{\top}\,(\Delta{\mathbf
    M}_{k+1})(\Delta{\mathbf M}_{k+1})^{\top}\,\mathbf{v}_1$, we
observe that
\begin{equation}\label{def:B_matrix}
B_{\theta,k+1}=\begin{pmatrix}
B_{k+1} & \mathbf{b}_{k+1} & B_{k+1}\\
\mathbf{b}_{k+1}^{\top} & b_{k+1} & \mathbf{b}_{k+1}^{\top}\\
B_{k+1} & \mathbf{b}_{k+1} & B_{k+1}
\end{pmatrix}.
\end{equation}
Since in $B_{\theta,k+1}$ the first and the third row and column of
blocks are the same, in~\eqref{step_proof-irene} the
$(2N-1)\times(2N-1)$ matrix $(A_{k+1,n}S_R)$ can be rewritten as a
diagonal matrix with the following diagonal blocks:
$A^{1}_{k+1,n}:=A^{11}_{k+1,n}$,
$A^{3}_{k+1,n}:=(A^{31}_{k+1,n}+c^{-1}A^{33}_{k+1,n})$ and
$a^{2}_{k+1,n}:=(c^{-1}-1)a^{22}_{k+1,n}$.  Hence, the expression
in~\eqref{step_proof-irene} can be rewritten as
\begin{equation}\label{eq:matrix}
n\sum_{k=m_0}^{n-1} r_k^2
\begin{pmatrix}
A^{1}_{k+1,n}B_{k+1}A^{1}_{k+1,n} & a^{2}_{k+1,n}A^{1}_{k+1,n}\mathbf{b}_{k+1}
& A^{1}_{k+1,n}B_{k+1}A^{3}_{k+1,n}\\
a^{2}_{k+1,n}\mathbf{b}_{k+1}^{\top}A^{1}_{k+1,n} & (a^{2}_{k+1,n})^2b_{k+1}
& a^{2}_{k+1,n}\mathbf{b}_{k+1}^{\top}A^{3}_{k+1,n}\\
A^{3}_{k+1,n}B_{k+1}A^{1}_{k+1,n} & a^{2}_{k+1,n}A^{3}_{k+1,n}\mathbf{b}_{k+1}
& A^{3}_{k+1,n}B_{k+1}A^{3}_{k+1,n}
\end{pmatrix}.\end{equation}
The elements of $A^{1}_{k+1,n}$, $a^{2}_{k+1,n}$ and $A^{3}_{k+1,n}$ in the above matrix
can be rewritten in terms of $F_{k+1,n}(\cdot)$, by~\eqref{def:terms_in_A_k_n}, in the following way:
\begin{equation}\label{def:terms_in_A_k_n_B_k}
\begin{aligned}
&[A^{1}_{k+1,n}]_{jj}=F_{k+1,n}(\alpha_j),\\
&a^{2}_{k+1,n}=(c^{-1}-1)F_{k+1,n}(c^{-1}),\\
&[A^{3}_{k+1,n}]_{jj}=
\left\{\begin{aligned}
&\frac{1}{c\alpha_j-1}
\left[(1-c^{-1})F_{k+1,n}(c^{-1})-(1-\alpha_j)F_{k+1,n}(\alpha_j)\right],\
&\hbox{for } c\alpha_j\neq 1,\\
&\left[(1-c^{-1})\ln\left(\frac{n}{k}\right)+c^{-1}\right]F_{k+1,n}(c^{-1})
+O(n^{-1}),\
&\hbox{for } c\alpha_j=1.
\end{aligned}
\right.
\end{aligned}
\end{equation}
Hence, the almost sure convergences
of all the elements in~\eqref{eq:matrix} can be obtained by combining the
results of the following limits:
\begin{equation}\label{eq:results_limits}
\begin{aligned}
&n\sum_{k=m_0}^{n-1} r_k^2 \beta_{k+1}
F_{k+1,n}(x)F_{k+1,n}(y)\
\stackrel{a.s}\longrightarrow\
\beta\frac{c^2}{c(x+y)-1},
\\
&n\sum_{k=m_0}^{n-1} r_k^2 \beta_{k+1}
\ln\left(\frac{n}{k}\right)F_{k+1,n}(x)F_{k+1,n}(y)\
\stackrel{a.s}\longrightarrow\
\beta\frac{c^2}{(c(x+y)-1)^2},
\\
&n\sum_{k=m_0}^{n-1} r_k^2 \beta_{k+1}
\ln^2\left(\frac{n}{k}\right)F_{k+1,n}(x)F_{k+1,n}(y)\
\stackrel{a.s}\longrightarrow\
\beta\frac{2c^2}{(c(x+y)-1)^3},
\end{aligned}
\end{equation}
for certain complex numbers $x,y\in\{\alpha_j,\, 2\leq j\leq N\}$
(remember that, by the assumption ${\mathcal Re}(\lambda^*)<1
-(2c)^{-1}$, we have $c(a_x+a_y)>1$ with $a_x:={\mathcal Re}(x)$ and
$a_y:={\mathcal Re}(y)$), a suitable sequence of random variables
$\beta_{k}\in\{[B_{k}]_{h,j},[\mathbf{b}_{k}]_{j},b_{k};\,
2\leq h,j\leq N\}$ and some random variable $\beta$.  Indeed,
using~Lemma~\ref{lemma-tecnico_1} and relation \eqref{relazione-nota},
we have
\begin{itemize}
\item[(1)] $n\sum_{k=m_0}^{n-1}r_k^2|\beta_{k+1}|O(n^{-2})\ =\
O(n^{-1})\sum_{k=m_0}^{n-1}O(k^{-2})\ \rightarrow 0 $;
\item[(2)] $n\ln(n)\sum_{k=m_0}^{n-1}r_k^2|\beta_{k+1}|O(n^{-1})|F_{k+1,n}(c^{-1})|\ =\
O(n^{-1}\ln(n))\sum_{k=m_0}^{n-1}O(k^{-1})\ \rightarrow 0$;
\item[(3)] $n\sum_{k=m_0}^{n-1}r_k^2|\beta_{k+1}|O(n^{-1})|F_{k+1,n}(y)|\ =\
O(n^{-ca_y})\sum_{k=m_0}^{n-1}O(k^{-(2-ca_y)})\ \rightarrow 0$.
\\
\end{itemize}

\indent In order to prove the convergences
in~\eqref{eq:results_limits}, we will apply Lemma \ref{lemma-serie-rv}
to each of the three limits.  Indeed, each quantity
in~\eqref{eq:results_limits} can be written as $\sum_{k=m_0}^{n-1}
v^{(e)}_{n,k}Y_k/c_k$, where
$$
Y_k=\beta_{k+1},
\quad
c_k=\frac{1}{kr_k^2}
\quad\mbox{and }\quad
v^{(e)}_{n,k}=\left(\frac{n}{k}\right)\ln^e
\left(\frac{n}{k}\right)F_{k+1,n}(x)F_{k+1,n}(y),
\quad\mbox{ for }e\in\{0,1,2\},
$$
satisfy the assumptions of Lemma \ref{lemma-serie-rv}.
More precisely, setting ${\mathcal H}_n={\mathcal F}_{n+1}$ we have
\begin{equation*}
\begin{split}
E[Y_n\,|\,{\mathcal H }_{n-1}] =
E[\beta_{n+1}\, |\, {\mathcal F}_n]
\ \stackrel{a.s}\longrightarrow\ \beta,
\end{split}
\end{equation*}
because, by \eqref{eq:multidim-limite-conditional-M}, we get that
\begin{equation*}
\begin{split}
&E[B_{n+1}\,|\,{\mathcal F }_{n}]\
=\
V^{\top}E[(\Delta{\mathbf M}_{n+1})
(\Delta{\mathbf M}_{n+1})^{\top}\,|\,{\mathcal F}_n] V\
\stackrel{a.s}\longrightarrow\ (V^{\top}V)Z_{\infty}(1-Z_{\infty}),
\\
&E[\mathbf{b}_{n+1}\,|\,{\mathcal F }_{n}]\
=\ V^{\top}E[(\Delta{\mathbf M}_{n+1})
(\Delta{\mathbf M}_{n+1})^{\top}\,|\,{\mathcal F}_n] \mathbf{v}_1\
\stackrel{a.s}\longrightarrow\ (V^{\top}\mathbf{v}_1)Z_{\infty}(1-Z_{\infty}),
\\
&E[b_{n+1}\,|\,{\mathcal F }_{n}]\
=\ \mathbf{v}_1^{\top}E[(\Delta{\mathbf M}_{n+1})
(\Delta{\mathbf M}_{n+1})^{\top}\,|\,{\mathcal F}_n]\mathbf{v}_1
\stackrel{a.s}\longrightarrow\ \|\mathbf{v}_1\|^2Z_{\infty}(1-Z_{\infty}).
\end{split}
\end{equation*}
Moreover, we have
$$
\sum_{k}\frac{ E[\,|Y_k|^2] }{ c_k^2 }=\sum_k E[\,|Y_k|^2]r_k^4 k^2=
\sum_k r_k^4 O(k^{2})=\sum_k O(1/k^{2})<+\infty.
$$
In addition, since
$|v^{(e)}_{n,k}|/c_k= n r_k^2\ln^e(n/k)|F_{k+1,n}(x)F_{k+1,n}(y)|$,
from~\eqref{affermazione3} in Lemma~\ref{lemma-tecnico_2} (with $u=1$)
it follows that $\sum_{k=m_0}^{n-1}\frac{|v^{(e)}_{n,k}|}{c_k}=O(1)$.
Analogously, using again Lemma~\ref{lemma-tecnico_2}, we can prove that
$\sum_{k=m_0}^{n-1}|v^{(e)}_{n,k}-v^{(e)}_{n,k-1}|=O(1)$ since by
Remark~\ref{rem:relations_diff_v} we have
\[
\left\{\begin{aligned}
&|v^{(e)}_{n,k}-v^{(e)}_{n,k-1}|\ =\
O\Big(nr_k^2\frac{|p_n(x)||p_n(y)|}{|p_k(x)||p_k(y)|}\Big),\
&\hbox{for } e=0,\\
&|v^{(e)}_{n,k}-v^{(e)}_{n,k-1}|\ =\
O\Big(nr_k^2(\ln(n/k)+1)\frac{|p_n(x)||p_n(y)|}{|p_k(x)||p_k(y)|}\Big),
\ &\hbox{for } e=1,\\
&|v^{(e)}_{n,k}-v^{(e)}_{n,k-1}|\ =\
O\Big(
nr_k^2(\ln^2(n/k)+\ln(n/k))\frac{|p_n(x)||p_n(y)|}{|p_k(x)||p_k(y)|}
\Big),
\ &\hbox{for } e=2.
\end{aligned}\right.
\]
Hence, condition~\eqref{cond-serie-rv-2} in Lemma \ref{lemma-serie-rv}
is satisfied and so, in order to apply this lemma, it only remains to
prove condition~\eqref{cond-serie-rv-1}.  To this end, we get the
values of $\lim_n\sum_{k=m_0}^n v^{(e)}_{n,k}/c_k$ by
\eqref{affermazione2} in Lemma \ref{lemma-tecnico_2}, and we observe
that $\lim_n v^{(e)}_{n,n}=s\in\{0,1\}$ and, for a fixed $k$,
$\lim_n|v^{(e)}_{n,k}|=0$ since by~Lemma \ref{lemma-tecnico_1} we have
$|p_n(x)p_n(y)|=O(n^{-c(a_x+a_y)})=o((n\ln^e(n))^{-1})$.\\

Now that we have proved the convergences in~\eqref{eq:results_limits},
we can use the relations in~\eqref{def:terms_in_A_k_n_B_k} to compute
the almost sure limits of all the elements in~\eqref{eq:matrix}.  The
results are listed below, while the technical computations are
reported in
Appendix~\ref{subsubsection_appendix_technical_computation_1}.
\begin{itemize}
\item $n\sum_{k=m_0}^{n-1} r_k^2 [A^{1}_{k+1,n}B_{k+1}A^{1}_{k+1,n}]_{h,j}\stackrel{a.s}\longrightarrow
\frac{c^2}{c(\alpha_h+\alpha_j)-1}(\mathbf{v}_h^{\top}\mathbf{v}_j)Z_{\infty}(1-Z_{\infty});$
\item $n\sum_{k=m_0}^{n-1} r_k^2 [A^{3}_{k+1,n}B_{k+1}A^{3}_{k+1,n}]_{h,j}\stackrel{a.s}\longrightarrow
\frac{1+(c-1)(\alpha_h^{-1}+\alpha_j^{-1})}{c(\alpha_h+\alpha_j)-1}(\mathbf{v}_h^{\top}\mathbf{v}_j)Z_{\infty}(1-Z_{\infty})$;
\item $n\sum_{k=m_0}^{n-1} r_k^2 (a^{2}_{k+1,n})^2b_{k+1}\stackrel{a.s}\longrightarrow (c-1)^2\|\mathbf{v}_1\|^2Z_{\infty}(1-Z_{\infty})$;
\item $n\sum_{k=m_0}^{n-1} r_k^2 [A^{1}_{k+1,n}B_{k+1}A^{3}_{k+1,n}]_{h,j}\stackrel{a.s}\longrightarrow
\frac{\alpha_h^{-1}(c-1)+c}{c(\alpha_h+\alpha_j)-1}(\mathbf{v}_h^{\top}\mathbf{v}_j)Z_{\infty}(1-Z_{\infty})$;
\item $n\sum_{k=m_0}^{n-1} r_k^2 a^{2}_{k+1,n}[\mathbf{b}_{k+1}^{\top}A^{1}_{k+1,n}]_{j}\stackrel{a.s}\longrightarrow
\frac{1-c}{\alpha_j}(\mathbf{v}_1^{\top}\mathbf{v}_j)Z_{\infty}(1-Z_{\infty})$;
\item $n\sum_{k=m_0}^{n-1} r_k^2 a^{2}_{k+1,n}[\mathbf{b}_{k+1}^{\top}A^{3}_{k+1,n}]_{j}\stackrel{a.s}\longrightarrow
\frac{1-c}{\alpha_j}(\mathbf{v}_1^{\top}\mathbf{v}_j)Z_{\infty}(1-Z_{\infty})$.
\end{itemize}

Hence, recalling the definitions of the matrices
$\widehat{S}_{\mathbf{ZZ}}, \widehat{S}_{\mathbf{NN}}$ and
$\widehat{S}_{\mathbf{ZN}}$ given in~\eqref{def:Shat_ZZ},
\eqref{def:Shat_NN_1}, \eqref{def:Shat_NN_2}, \eqref{def:Shat_ZN_1}
and \eqref{def:Shat_ZN_2}, we obtain 
\begin{equation*}\begin{aligned}
&Z_\infty(1-Z_\infty)\widehat{S}_{\mathbf{ZZ}}\ =\
a.s.-\lim_{n\rightarrow\infty}n\sum_{k=m_0}^{n-1} r_k^2A^{1}_{k+1,n}B_{k+1}A^{1}_{k+1,n},\\
&Z_\infty(1-Z_\infty)\widehat{S}_{\mathbf{NN}}\ =\
a.s.-\lim_{n\rightarrow\infty}n\sum_{k=m_0}^{n-1} r_k^2
\begin{pmatrix}
(a^{2}_{k+1,n})^2b_{k+1} & a^{2}_{k+1,n}\mathbf{b}_{k+1}^{\top}A^{3}_{k+1,n}\\
a^{2}_{k+1,n}A^{3}_{k+1,n}\mathbf{b}_{k+1} & A^{3}_{k+1,n}B_{k+1}A^{3}_{k+1,n}
\end{pmatrix},\\
&Z_\infty(1-Z_\infty)\widehat{S}_{\mathbf{ZN}}\ =\
a.s.-\lim_{n\rightarrow\infty}n\sum_{k=m_0}^{n-1} r_k^2
\begin{pmatrix}
a^{2}_{k+1,n}A^{1}_{k+1,n}\mathbf{b}_{k+1} & A^{1}_{k+1,n}B_{k+1}A^{3}_{k+1,n}
\end{pmatrix}.
\end{aligned}
\end{equation*}
Therefore, using~\eqref{eq:matrix}, we can state that
$$\sum_{k=m_0}^n (\sqrt{n}\mathbf{T}_{n,k})(\sqrt{n}\mathbf{T}_{n,k})^{\top}\
\stackrel{a.s.}\longrightarrow\ 
Z_\infty(1-Z_\infty)
U_{\theta}\begin{pmatrix}
\widehat{S}_{\mathbf{ZZ}} & \widehat{S}_{\mathbf{ZN}}\\
\widehat{S}_{\mathbf{ZN}}^{\top} & \widehat{S}_{\mathbf{NN}}
\end{pmatrix}U^{\top}_{\theta}\ =\ 
Z_\infty(1-Z_\infty)
\begin{pmatrix}
U\widehat{S}_{\mathbf{ZZ}}U^{\top} & U\widehat{S}_{\mathbf{ZN}}\widetilde{U}^{\top}\\
\widetilde{U}\widehat{S}_{\mathbf{ZN}}^{\top}U^{\top}
& \widetilde{U}\widehat{S}_{\mathbf{NN}}\widetilde{U}^{\top}
\end{pmatrix},$$
where the last matrix coincides with the one in the statement of the
theorem because of~\eqref{def:Sigmahat_ZZ_1},
\eqref{def:Sigmahat_NN_1} and \eqref{def:Sigmahat_ZN_1}.  
\\

\indent Regarding condition $(c_3)$, we observe that, using the
inequalities
$$
|{\mathbf T}_{n,k}|=r_k |C_{k+1,n}R\Delta{\mathbf M}_{\theta,k+1}|\leq
r_k |U||A_{k+1,n}||V^{\top}||R||\Delta{\mathbf M}_{\theta,k+1}| \leq K r_k |A_{k+1,n}|,
$$
with a suitable constant $K$, we find for any $u>1$
\begin{equation*}
\begin{split}
\big(\sup_{m_0\leq k\leq n}|\sqrt{n}\mathbf{T}_{n,k}|\big)^{2u}
&\leq n^{u} \!\sum_{k=m_0}^{n-1} |\mathbf{T}_{n,k}|^{2u} 
+ n^{u} |\mathbf{T}_{n,n}|^{2u}\\
&=
n^{u} O\left(|p_{n}^*|^{2u}\sum_{k=m_0}^{n-1} \frac{r_k^{2u}}{|p_{k}^*|^{2u}}\right)
+ O\Big( \frac{\ln(n)^{2u}}{n^{u-1}} \Big)
+ n^{u} O(r_n^{2u}),
\end{split}
\end{equation*}
where, for the last equality, we have used \eqref{eq-corretta}.  Now,
since $2ca^*>1$, by \eqref{affermazione3} in Lemma
\ref{lemma-tecnico_2} (with $x=y=\alpha^*=1-\lambda^*$, $e=0$ and
$u>1$), we have
\begin{equation*}
|p_{n}^*|^{2u}\sum_{k=m_0}^{n-1} \frac{r_k^{2u}}{|p_{k}^*|^{2u}}\ =
\begin{cases}
O(n^{-2uca^*})\quad &\hbox{for } 2uca^*<2u-1,\\
O(n^{-(2u-1)}\ln(n))\quad &\hbox{for } 2uca^*=2u-1,\\
O(n^{-(2u-1)})\quad &\hbox{for } 2uca^*>2u-1,
\end{cases}
\end{equation*}
\noindent which, in particular, implies $(\sup_{m_0\leq k\leq
  n}|\sqrt{n}\mathbf{T}_{n,k}|)^{2u}\stackrel{L^1}\longrightarrow 0$ for any
$u>1$.  As a consequence of the above convergence to zero, condition
(c3) of Theorem~\ref{thm:triangular} holds true.  \\

Summing up, all the conditions required by
Theorem~\ref{thm:triangular} are satisfied and so we can apply this
theorem and obtain the stable convergence of $\sqrt{n}\sum_{k=m_0}^n
\mathbf{T}_{n,k}$ to the Gaussian kernel with random covariance matrix defined
in Theorem~\ref{thm:asymptotics_theta_hat_1}.
\end{proof}

Now, we are ready to prove Theorem \ref{thm:asymptotics_theta_1}.
\\

{\it Proof of Theorem \ref{thm:asymptotics_theta_1}.} By
Theorem~\ref{thm:asymptotics_Z_tilde}, we have that
\begin{equation*}
\sqrt{n}(\widetilde{Z}_n-Z_\infty)\mathbf{1}
\longrightarrow {\mathcal N}\big(\mathbf{0},
Z_{\infty}(1-Z_{\infty})\widetilde{\Sigma}_\gamma\big)
\ \ \ \hbox{stably in the strong sense}.
\end{equation*}
Thus, from Theorem \ref{thm:asymptotics_theta_hat_1}, applying
Theorem~\ref{blocco}, we obtain that
$$
\sqrt{n}
\left(
\begin{pmatrix}
{\mathbf Z}_n - \widetilde{Z}_n{\mathbf 1}\\
{\mathbf N}_n - \widetilde{Z}_n{\mathbf 1}
\end{pmatrix},
(\widetilde{Z}_n-Z_\infty)\mathbf{1}
\right)
\longrightarrow
\mathcal{N}\left(\mathbf{0},
Z_\infty(1-Z_\infty)
\begin{pmatrix}
\widehat{\Sigma}_{\mathbf{ZZ}} & \widehat{\Sigma}_{\mathbf{ZN}}\\
\widehat{\Sigma}_{\mathbf{ZN}}^{\top} & \widehat{\Sigma}_{\mathbf{NN}}
\end{pmatrix}
\right)
\otimes
{\mathcal N}\left(\mathbf{0},
Z_{\infty}(1-Z_{\infty})\widetilde{\Sigma}_\gamma\right)
$$
stably.  In order to conclude, it is enough to observe that
$$
\sqrt{n }
\begin{pmatrix}
\mathbf{Z}_n-Z_{\infty}\mathbf{1}\\
{\mathbf N}_n-Z_{\infty}\mathbf{1}
\end{pmatrix}
=\Phi
\left(
{\mathbf Z}_n - \widetilde{Z}_n{\mathbf 1},
{\mathbf N}_n - \widetilde{Z}_n{\mathbf 1},
(\widetilde{Z}_n-Z_\infty)\mathbf{1}
\right)\,,
$$
where $\Phi(x,y,z)=(x+z, y+z)^{\top}$.
\qed

\subsection{Proof of Theorem \ref{thm:asymptotics_Z_1_star}
(CLT for $({\mathbf Z}_n,{\mathbf N}_n)_n$ in the case $N\geq 2$, $\gamma=1$
and ${\mathcal Re}(\lambda^{*})=1-(2c)^{-1}$)}

As above, in order to prove Theorem \ref{thm:asymptotics_Z_1_star}, we
need the following convergence result on
$(\widehat{\mathbf{Z}}_n, \widehat{\mathbf{N}}_n)_n$:

\begin{theo}\label{thm:asymptotics_theta_1_star}
Let $N\geq 2$, $\gamma=1$ and ${\mathcal Re}(\lambda^{*})=1-(2c)^{-1}$.
Then, under condition~\eqref{ass:condition_r_n_1}, we have that
\begin{equation*}
\sqrt{\frac{n}{\ln(n)}}
\begin{pmatrix}
\widehat{\mathbf{Z}}_n\\
\widehat{\mathbf{N}}_n
\end{pmatrix}
{\longrightarrow}\
\mathcal{N} \left(\ \mathbf{0}\ ,\ Z_{\infty}(1-Z_{\infty})
\begin{pmatrix}
\widehat{\Sigma}^{*}_{\mathbf{ZZ}} & \widehat{\Sigma}^{*}_{\mathbf{ZN}}\\
\widehat{\Sigma}_{\mathbf{ZN}}^{*\top} & \widehat{\Sigma}^{*}_{\mathbf{NN}}
\end{pmatrix}\
\right)\ \ \ \ stably,
\end{equation*}
where $\widehat{\Sigma}^{*}_{\mathbf{ZZ}}$, $\widehat{\Sigma}^{*}_{\mathbf{NN}}$ and
$\widehat{\Sigma}^{*}_{\mathbf{ZN}}$ are the matrices defined
in~\eqref{def:Sigmahat_ZZ_1_star},~\eqref{def:Sigmahat_NN_1_star}
and~\eqref{def:Sigmahat_ZN_1_star}, respectively.
\end{theo}

\begin{proof} The proof of Theorem~\ref{thm:asymptotics_theta_1_star} follows
analogous arguments to those used in
Theorem~\ref{thm:asymptotics_theta_hat_1}.  In particular, consider
the joint dynamics of $\boldsymbol{\theta}_n:=
(\widehat{\mathbf{Z}}_{n},\widehat{\mathbf{N}}_{n})^{\top}$ defined
in~\eqref{eq:dynamic_SA_theta_proof} as follows:
$$\boldsymbol{\theta}_{n+1}=
C_{m_0,n}\boldsymbol{\theta}_{m_0}+
\sum_{k=m_0}^n {\mathbf T}_{n,k}+\boldsymbol{\rho}_{n},\qquad\mbox{where}
\qquad \left\{
\begin{aligned}
&{\mathbf T}_{n,k}=r_k C_{k+1,n} R\Delta{\mathbf M}_{\theta,k+1},\\
&\boldsymbol{\rho}_{n}=\sum_{k=m_0}^nr_k C_{k+1,n} {\mathbf R}_{\theta,k+1},
\end{aligned}\right.$$
where $C_{k+1,n}$ is defined
in~\eqref{def:C_k_n_A_k_n}, $R$ is defined in
\eqref{def:R-irene},
$\Delta\mathbf{M}_{\theta,n}=(\Delta\mathbf{M}_n,\Delta\mathbf{M}_n)^{\top}$ and
$\mathbf{R}_{\theta,n}=(\mathbf{0},\mathbf{R}_n)^{\top}$ with
$\mathbf{R}_n$ defined in~\eqref{def:R_n}.  Then, we are going to prove that
$\sqrt{n/\ln(n)}\sum_{k=m_0}^n {\mathbf T}_{n,k}$ converges stably to the desired
Gaussian kernel, while $\sqrt{n/\ln(n)}|C_{m_0,n}\boldsymbol{\theta}_{m_0}|$ and
$\sqrt{n/\ln(n)}|\boldsymbol{\rho}_{n}|$ converge almost surely to zero.
\\

First, note that by \eqref{eq-corretta}, we have that
\begin{equation}\label{eq-corretta-2}
\begin{split}
|A_{k+1,n}|&=O\left(\frac{|p_n^*|}{|p_k^*|}\right)
+O\left(\frac{k}{n}\ln(n)\right)
=O\left( \Big(\frac{k}{n}\Big)^{ca^*} \right)
+O\left(\frac{k}{n}\ln(n)\right)
\\
&=O\left( \Big(\frac{k}{n}\Big)^{1/2} \right)
+O\left(\frac{k}{n}\ln(n)\right)
\quad\hbox{for } m_0-1\leq k\leq n-1,
\end{split}
\end{equation}
where, as before, the symbol $^*$ refers to the quantities
$a_{\alpha_j}:={\mathcal Re}(\alpha_j)$ and $p_n(\alpha_j)$
corresponding to $\alpha^*=\alpha_j=1-\lambda_j$ with
$\lambda_j=\lambda^*\in \lambda_{\max}(D)$, and hence the last passage
follows since $ca^*=1/2$ by assumption. As a consequence, we obtain
\begin{equation*}
|C_{m_0,n}\boldsymbol{\theta}_{m_0}|\ =\
O\left(n^{-1/2}\right) +O\left(\frac{\ln(n)}{n}\right)
\end{equation*}
and so  $\sqrt{n/\ln(n)} |C_{m_0,n}\boldsymbol{\theta}_{m_0}|\rightarrow0$
almost surely.\\

\indent Concerning the term $\boldsymbol{\rho}_{n}$, notice that
by~\eqref{ass:condition_r_n_1} and~\eqref{def:R_n} we have that
$|\mathbf{R}_k|=O(k^{-1})$ and, by \eqref{eq-corretta-2}, we have that
\begin{equation*}
|C_{k+1,n}|\ =\ O\left(\ \left(\frac{k}{n}\right)^{1/2}\ \right)
+O\left(\frac{k}{n}\ln(n)\right)
\quad\hbox{for } m_0\leq k\leq n-1.
\end{equation*}
Therefore, since $\boldsymbol{\rho}_{n}=\sum_{k=m_0}^nr_k C_{k+1,n}
{\mathbf R}_{\theta,k+1}=\sum_{k=m_0}^{n-1} r_k C_{k+1,n} {\mathbf
  R}_{\theta,k+1} + r_nC_{n+1,n}{\mathbf R}_{\theta,n+1}$, it follows
(by \eqref{relazione-nota}) that
\begin{equation*}
\sqrt{n/\ln(n)}|\boldsymbol{\rho}_{n}|\ =\
O\left( 1/\sqrt{\ln(n)}
\sum_{k=m_0}^{n-1} k^{-3/2}
\right)
+
O\left( n^{-1/2} \ln(n)^{3/2} \right)
+
O\left(n^{-3/2} \ln(n)^{-1/2}\right)
\longrightarrow 0\quad \hbox{a.s.}\,.
\end{equation*}

\indent We now focus on the proof of the fact that
$\sqrt{n/\ln(n)}\sum_{k=m_0}^n{\mathbf T}_{n,k}$ converges stably to the 
suitable Gaussian kernel.  For this purpose, we set ${\mathcal
  G}_{n,k}={\mathcal F}_{k+1}$, and consider
Theorem~\ref{thm:triangular}.  Given the fact that condition $(c1)$ of
Theorem~\ref{thm:triangular} is obviously satisfied, we will check
only conditions $(c2)$ and $(c3)$.\\

Regarding condition $(c2)$, from the computations seen in the proof of
Theorem~\ref{thm:asymptotics_theta_hat_1} and using the fact that
$O(nr_n^2/\ln(n))=O(n^{-1}/\ln(n))\to 0$, we have
\[\begin{aligned}
&a.s.-\lim_n\sum_{k=m_0}^n \left(\sqrt{\frac{n}{\ln(n)}}\mathbf{T}_{n,k}\right)
\left(\sqrt{\frac{n}{\ln(n)}}\mathbf{T}_{n,k}\right)^{\top}\
=\\
&U_{\theta}\left(a.s.-\lim_n\frac{n}{\ln(n)}\sum_{k=m_0}^{n-1} r_k^2
A_{k+1,n}S_R\,V_{\theta}^{\top}
\,(\Delta{\mathbf M}_{\theta,k+1})(\Delta{\mathbf M}_{\theta,k+1})^{\top}\,
V_{\theta}\, S_RA^{\top}_{k+1,n}\right)U_{\theta}^{\top}.
\end{aligned}
\]
Then, setting $B_{\theta,k+1}$ as in~\eqref{def:B_matrix},
the limit of the above expression can be obtain by studying the
convergence of the following matrix:
\begin{equation}\label{eq:matrix_star}
\frac{n}{\ln(n)}\sum_{k=m_0}^{n-1} r_k^2
\begin{pmatrix}
A^{1}_{k+1,n}B_{k+1}A^{1}_{k+1,n} & a^{2}_{k+1,n}A^{1}_{k+1,n}\mathbf{b}_{k+1}
& A^{1}_{k+1,n}B_{k+1}A^{3}_{k+1,n}\\
a^{2}_{k+1,n}\mathbf{b}_{k+1}^{\top}A^{1}_{k+1,n} & (a^{2}_{k+1,n})^2b_{k+1}
& a^{2}_{k+1,n}\mathbf{b}_{k+1}^{\top}A^{3}_{k+1,n}\\
A^{3}_{k+1,n}B_{k+1}A^{1}_{k+1,n} & a^{2}_{k+1,n}A^{3}_{k+1,n}\mathbf{b}_{k+1}
& A^{3}_{k+1,n}B_{k+1}A^{3}_{k+1,n}
\end{pmatrix},\end{equation}
where $A^{1}_{k+1,n},a^{2}_{k+1,n},A^{3}_{k+1,n}$ are defined
in~\eqref{def:terms_in_A_k_n_B_k}.  Notice that the almost sure convergences
of all the elements in~\eqref{eq:matrix_star} can be obtained by
combining the results of the following limits:
\begin{equation}\label{eq:results_limit_star}
\begin{split}
&\frac{n}{\ln(n)}
\sum_{k=m_0}^{n-1} r_k^2 \beta_{k+1}
\ln^e\left(\frac{n}{k}\right)F_{k+1,n}(x)F_{k+1,n}(y)\
\stackrel{a.s}\longrightarrow\ 0,
\quad\hbox{with } c(a_x+a_y)>1\,\hbox{and}\,
e=0,1,2,\\
&\frac{n}{\ln(n)}
\sum_{k=m_0}^{n-1} r_k^2 \beta_{k+1}
F_{k+1,n}(x)F_{k+1,n}(y)\
\stackrel{a.s}\longrightarrow\
\left\{\begin{aligned}
&c^2\beta\ &\mbox{ if }c(a_x+a_y)=1\mbox{ and }b_x+b_y=0,\\
&0\ &\mbox{ if }c(a_x+a_y)=1\mbox{ and }b_x+b_y\neq0,
\end{aligned}\right.
\end{split}
\end{equation}
for certain complex numbers $x,y\in\{\alpha_j,\,2\leq j\leq N\}$ with
$a_x:={\mathcal Re}(x)$, $b_x:={\mathcal Im}(x)$, $a_y:={\mathcal Re}(y)$
and $b_y:={\mathcal Im}(y)$ (remember that, by the assumption on
${\mathcal Re}(\lambda^*)$, we can have both cases $c(a_x+a_y)>1$ and
$c(a_x+a_y)=1$), a suitable sequence of random variables
$\beta_{k}\in\{[B_{k}]_{h,j},[\mathbf{b}_{k}]_{j},b_{k};\, 2\leq
h,j\leq N\}$ and some random variable $\beta$.  \\

In order to prove the convergence in \eqref{eq:results_limit_star} for the
case $c(a_x+a_y)>1$, we can use the convergences in~\eqref{eq:results_limits}
established in the proof of Theorem~\ref{thm:asymptotics_theta_hat_1};
while for the case $c(a_x+a_y)=1$ we can apply Lemma
\ref{lemma-serie-rv} since each quantity
in~\eqref{eq:results_limit_star} can be written as $\sum_{k=m_0}^{n-1}
v_{n,k}Y_k/c_k$, where
$$
Y_k=\beta_{k+1},
\quad
c_k=\frac{1}{kr_k^2}
\quad\mbox{and }\quad
v_{n,k}=\frac{1}{\ln(n)}\left(\frac{n}{k}\right)F_{k+1,n}(x)F_{k+1,n}(y)
$$
satisfy the assumptions of Lemma \ref{lemma-serie-rv}.  Indeed,
similarly as in the proof of Theorem~\ref{thm:asymptotics_theta_hat_1},
we have
$$
\sum_{k}\frac{ E[\,|Y_k|^2] }{ c_k^2 }<+\infty,\;
E[[B_{k+1}]_{h,j}|{\mathcal F}_n] \stackrel{a.s}\rightarrow
(\mathbf{v}_h^{\top}\mathbf{v}_j),\;
E[[\mathbf{b}_{k+1}]_{j}|{\mathcal F}_n]
\stackrel{a.s}\rightarrow (\mathbf{v}_j^{\top}\mathbf{v}_1)
\;\hbox{and }
E[b_{k+1}|{\mathcal F}_n] \stackrel{a.s}\rightarrow
\|\mathbf{v}_1\|^2.$$
In addition, since
$|v_{n,k}|/c_k=(n/\ln(n))r_k^2|F_{k+1,n}(x)F_{k+1,n}(y)|$,
from~\eqref{affermazione3_log} in Lemma~\ref{lemma-tecnico_2} (with
$u=1$) it follows that
$\sum_{k=m_0}^{n-1}\frac{|v_{n,k}|}{c_k}=O(1)$.  Moreover, we
have that $\sum_{k=m_0}^{n-1}|v_{n,k}-v_{n,k-1}|=O(1)$ since by
Remark~\ref{rem:relations_diff_v} we have
\[|v_{n,k}-v_{n,k-1}|\ =\ \left\{\begin{aligned}
&O(k^{-1}/\ln(n))\ &\mbox{ if }b_x+b_y \neq 0,\\
&O(k^{-2}/\ln(n))\ &\mbox{ if }b_x+b_y = 0.
\end{aligned}\right.\]
Hence, condition~\eqref{cond-serie-rv-2} of Lemma \ref{lemma-serie-rv}
is satisfied and so, in order to apply this lemma, it only remains to
prove condition~\eqref{cond-serie-rv-1}.  To this end, we get the
value of $\lim_n\sum_{k=m_0}^{n-1} v_{n,k}/c_k$ from
\eqref{affermazione2_log} in Lemma \ref{lemma-tecnico_2}, and we
observe that $\lim_n v_{n,n}=0$ and, for a fixed $k$,
$\lim_n|v_{n,k}|=0$ since by Lemma~\ref{lemma-tecnico_1} we have
$|p_n(x)p_n(y)|=O(n^{-1})$.\\

Now that we have proved the convergences
in~\eqref{eq:results_limit_star}, we can use the relations
in~\eqref{def:terms_in_A_k_n_B_k} to compute the almost sure limits of
all the elements in~\eqref{eq:matrix_star}.  The results are listed
below, while the technical computations are reported in
Appendix~\ref{subsubsection_appendix_technical_computation_1_star}.
\begin{itemize}
\item $\frac{n}{\ln(n)}\sum_{k=m_0}^{n-1} r_k^2[A^{1}_{k+1,n}B_{k+1}A^{1}_{k+1,n}]_{h,j}\stackrel{a.s}\longrightarrow
(\mathbf{v}_h^{\top}\mathbf{v}_j)Z_{\infty}(1-Z_{\infty})c^2\ind_{\{b_{\alpha_h}+b_{\alpha_j}=0\}}$;
\item $\frac{n}{\ln(n)}\sum_{k=m_0}^{n-1} r_k^2 [A^{3}_{k+1,n}B_{k+1}A^{3}_{k+1,n}]_{h,j}\stackrel{a.s}\longrightarrow
(\mathbf{v}_h^{\top}\mathbf{v}_j)Z_{\infty}(1-Z_{\infty})\frac{(\alpha_h-1)(\alpha_j-1)}{\alpha_h\alpha_j}\ind_{\{ b_{\alpha_h}+b_{\alpha_j}=0\}}$;
\item $\frac{n}{\ln(n)}\sum_{k=m_0}^{n-1} r_k^2b_{k+1}(a^{2}_{k+1,n})^2\stackrel{a.s}\longrightarrow0$;
\item $\frac{n}{\ln(n)}\sum_{k=m_0}^{n-1} r_k^2[A^{1}_{k+1,n}B_{k+1}A^{3}_{k+1,n}]_{h,j}\stackrel{a.s}\longrightarrow
(\mathbf{v}_h^{\top}\mathbf{v}_j)Z_{\infty}(1-Z_{\infty})\frac{c(1-\alpha_j)}{\alpha_h}\ind_{\{b_{\alpha_h}+b_{\alpha_j}=0\}}$;
\item $\frac{n}{\ln(n)}\sum_{k=m_0}^{n-1} r_k^2 a^{2}_{k+1,n}[\mathbf{b}_{k+1}^{\top}A^{1}_{k+1,n}]_{j}\stackrel{a.s}\longrightarrow0$;
\item $\frac{n}{\ln(n)}\sum_{k=m_0}^{n-1} r_k^2 a^{2}_{k+1,n}[\mathbf{b}_{k+1}^{\top}A^{3}_{k+1,n}]_{j}\stackrel{a.s}\longrightarrow0$.
\end{itemize}

Hence, recalling the definitions of the matrices
$\widehat{S}^{*}_{\mathbf{ZZ}}, \widehat{S}^{*}_{\mathbf{NN}}$ and
$\widehat{S}^{*}_{\mathbf{ZN}}$ given in~\eqref{def:Shat_ZZ_star},
\eqref{def:Shat_NN_star} and \eqref{def:Shat_ZN_star}, we obtain 
\begin{equation*}\begin{aligned}
&Z_\infty(1-Z_\infty)\widehat{S}^{*}_{\mathbf{ZZ}}\ =\
a.s.-\lim_{n\rightarrow\infty}n\sum_{k=m_0}^{n-1} r_k^2A^{1}_{k+1,n}B_{k+1}A^{1}_{k+1,n},\\
&Z_\infty(1-Z_\infty)\widehat{S}^{*}_{\mathbf{NN}}\ =\
a.s.-\lim_{n\rightarrow\infty}n\sum_{k=m_0}^{n-1} r_k^2A^{3}_{k+1,n}B_{k+1}A^{3}_{k+1,n},\\
&Z_\infty(1-Z_\infty)\widehat{S}^{*}_{\mathbf{ZN}}\ =\
a.s.-\lim_{n\rightarrow\infty}n\sum_{k=m_0}^{n-1} r_k^2A^{1}_{k+1,n}B_{k+1}A^{3}_{k+1,n}.
\end{aligned}
\end{equation*}
Therefore, using~\eqref{eq:matrix_star}, we can state that
\begin{equation*}
\begin{split}
\sum_{k=m_0}^n
\left(\sqrt{\frac{n}{\ln(n)}}\mathbf{T}_{n,k}\right)
\left(\sqrt{\frac{n}{\ln(n)}}\mathbf{T}_{n,k}\right)^{\top}
&\ \stackrel{a.s.}\longrightarrow\
Z_\infty(1-Z_\infty)
U_{\theta}
\begin{pmatrix}
\widehat{S}^{*}_{\mathbf{ZZ}} & \mathbf{0} & \widehat{S}^{*}_{\mathbf{ZN}}\\
\mathbf{0}^{\top} & 0 &  \mathbf{0}^{\top}\\
\widehat{S}_{\mathbf{ZN}}^{*\top} & \mathbf{0} & \widehat{S}^{*}_{\mathbf{NN}}
\end{pmatrix}
U^{\top}_{\theta}
\\
&\ =\
Z_\infty(1-Z_\infty)
\begin{pmatrix}
U\widehat{S}^{*}_{\mathbf{ZZ}}U^{\top} & U\widehat{S}^{*}_{\mathbf{ZN}}U^{\top}\\
U\widehat{S}_{\mathbf{ZN}}^{*\top}U^{\top} & U\widehat{S}^{*}_{\mathbf{NN}}U^{\top}
\end{pmatrix},
\end{split}
\end{equation*}
where the last matrix coincides with the one in the statement of
the theorem because of~\eqref{def:Sigmahat_ZZ_1_star},
\eqref{def:Sigmahat_NN_1_star} and \eqref{def:Sigmahat_ZN_1_star}. 
\\

\indent Regarding condition $(c_3)$, we observe that, using the
inequalities
$$
|{\mathbf T}_{n,k}|=r_k |C_{k+1,n}R\Delta{\mathbf M}_{\theta,k+1}|\leq
r_k |U||A_{k+1,n}||V^{\top}||R||\Delta{\mathbf M}_{\theta,k+1}| \leq
K r_k |A_{k+1,n}|,
$$
with a suitable constant $K$, we find for any $u>1$
\begin{align*}
\left(\sup_{m_0\leq k\leq
    n}\Big|\sqrt{\frac{n}{\ln(n)}}\mathbf{T}_{n,k}\Big|\right)^{2u} &\leq
  \left(\frac{n}{\ln(n)}\right)^{u} \sum_{k=m_0}^{n-1}
  |\mathbf{T}_{n,k}|^{2u} + \left(\frac{n}{\ln(n)}\right)^{u}
  |\mathbf{T}_{n,n}|^{2u}
  \\ &= \left(\frac{n}{\ln(n)}\right)^{u}
  O\left(|p_{n}^*|^{2u} \sum_{k=m_0}^{n-1}
  \frac{r_k^{2u}}{|p_{k}^*|^{2u}}\right)
+
O\left(\frac{\ln(n)^{u}}{n^{u-1}}\right)
+
 \left(\frac{n}{\ln(n)}\right)^{u} O(r_n^{2u}),
\end{align*}
where, for the last equality, we have used \eqref{eq-corretta-2}. Now,
since $2ca^*=1$, by \eqref{affermazione3_log} in Lemma
\ref{lemma-tecnico_2} (with $x=y=\alpha^*=1-\lambda^*$ and $u>1$), we
have
\begin{equation*}
|p_{n}^*|^{2u}\sum_{k=m_0}^{n-1}
\frac{r_k^{2u}}{|p_{k}^*|^{2u}}\ =\ O(n^{-u}),
\end{equation*}
which, in particular, implies $(\sup_{m_0\leq k\leq
  n}|\sqrt{(n/\ln(n))}\mathbf{T}_{n,k}|)^{2u}\stackrel{L^1}\longrightarrow0$
for any $u>1$.  As a consequence of the above convergence to zero,
condition (c3) of Theorem~\ref{thm:triangular} holds true.  \\

Summing up, all the conditions required by
Theorem~\ref{thm:triangular} are satisfied and so we can apply this
theorem and obtain the stable convergence of
$\sqrt{n/\ln(n)}\sum_{k=m_0}^n \mathbf{T}_{n,k}$ to the Gaussian
kernel with random covariance matrix defined in
Theorem~\ref{thm:asymptotics_theta_1_star}.
\end{proof}

Now, we are ready to prove Theorem \ref{thm:asymptotics_Z_1_star}.
\\

{\it Proof of Theorem \ref{thm:asymptotics_Z_1_star}.} By
Theorem~\ref{thm:asymptotics_Z_tilde}, we have that
\begin{equation*}
\sqrt{n}(\widetilde{Z}_n-Z_\infty)
\longrightarrow {\mathcal N}\big(0,
Z_{\infty}(1-Z_{\infty})\widetilde{\sigma}_\gamma^2\big)
\ \ \ \hbox{stably}.
\end{equation*}
Moreover, from Theorem \ref{thm:asymptotics_theta_1_star}, we have that
$$
\sqrt{\frac{n}{\ln(n)}}
\begin{pmatrix}
{\mathbf{Z}}_n-\widetilde{Z}_n\mathbf{1}\\
{\mathbf{N}}_n-\widetilde{Z}_n\mathbf{1}
\end{pmatrix}
{\longrightarrow}\
\mathcal{N} \left(\ \mathbf{0}\ ,\ Z_{\infty}(1-Z_{\infty})
\begin{pmatrix}
\widehat{\Sigma}^{*}_{\mathbf{ZZ}} & \widehat{\Sigma}^{*}_{\mathbf{ZN}}\\
\widehat{\Sigma}_{\mathbf{ZN}}^{*\top} & \widehat{\Sigma}^{*}_{\mathbf{NN}}
\end{pmatrix}\
\right)\ \ \ \ \hbox{stably}.
$$
In order to conclude, it is enough to observe that
\begin{equation*}
\sqrt{\frac{n}{\ln(n)}}
\begin{pmatrix}
\mathbf{Z}_n-Z_{\infty}\mathbf{1}\\
\mathbf{N}_n-Z_{\infty}\mathbf{1}
\end{pmatrix}
=
\sqrt{\frac{n}{\ln(n)}}
\begin{pmatrix}
{\mathbf{Z}}_n-\widetilde{Z}_n\mathbf{1}\\
{\mathbf{N}}_n-\widetilde{Z}_n\mathbf{1}
\end{pmatrix}
+
\sqrt{\frac{1}{\ln(n)}}\sqrt{n}
(\widetilde{Z}_n-Z_{\infty})
\begin{pmatrix}
\mathbf{1}\\
\mathbf{1}
\end{pmatrix},
\end{equation*}
where the last term converges in probability to zero.
\qed

\begin{center}
{\bf Acknowledgments}
\end{center}

\noindent Irene Crimaldi and Andrea Ghiglietti are members of the
Italian group ``Gruppo Nazionale per l'Analisi Matematica, la
Probabilit\`a e le loro Applicazioni (GNAMPA)'' of the Italian
Institute ``Istituto Nazionale di Alta Matematica (INdAM)''.\\ Giacomo
Aletti is a member of the Italian group ``Gruppo Nazionale per il
Calcolo Scientifico (GNCS)'' of the Italian Institute ``Istituto
Nazionale di Alta Matematica (INdAM)''.

\newpage
\appendix

\begin{center}
\huge{{\bf Appendix}}
\end{center}

\section{Some technical results}\label{app-A}

In all the sequel, given $(a_n)_n, (b_n)_n$ two sequences of real
numbers with $b_n\geq 0$, the notation $a_n=O(b_n)$ means $|a_n|\leq C
b_n$ for a suitable constant $C>0$ and $n$ large enough.  Therefore,
if we also have $a_n^{-1}=O(b_n^{-1})$, then $C'b_n\leq |a_n|\leq C
b_n$ for suitable constants $C,C'>0$ and $n$ large enough. Moreover,
given $(z_n)_n, (z'_n)_n$ two sequences of complex numbers, with
$z'_n\neq 0$, the notation $z_n=o(z'_n)$ means $\lim_n z_n/z'_n=0$.

\subsection{Asymptotic results for sums of complex numbers}

We start recalling an extension of the Toeplitz lemma (see
\cite{lin-ros}) to complex numbers provided in \cite{ale-cri-ghi},
from which we get useful technical results employed in our proofs.

\begin{lem}{\cite[Lemma A.2]{ale-cri-ghi}}\label{toeplitz-lemma-complex} 
(Generalized Toeplitz lemma)\\
Let $\{z_{n,k}:\, 1\leq k\leq k_n\}$ be a
triangular array of complex numbers such that
\begin{itemize}
\item[i)] $\lim_n z_{n,k}=0$ for each fixed $k$;
\item[ii)] $\lim_n \sum_{k=1}^{k_n} z_{n,k}=s\in\{0,1\}$;
\item[iii)] $\sum_{k=1}^{k_n} |z_{n,k}|=O(1)$.
\end{itemize}
Let $(w_n)_n$ be a sequence of complex numbers with $\lim_n
w_n=w\in{\mathbb C}$. Then, we have $\lim_n \sum_{k=1}^{k_n}
z_{n,k}w_k=s w$.
\end{lem}

From this lemma we can easily get the following corollary, which
slightly extends the generalized version of the Kronecker lemma
provided in \cite[Corollary A.3]{ale-cri-ghi}:

\begin{cor}\label{kronecker-complex} (Generalized Kronecker lemma)\\
Let $\{v_{n,k}:1\leq k\leq n\}$ and $(z_n)_n$ be respectively a
triangular array and a sequence of complex numbers such that
 $v_{n,k}\neq 0$ and
$$
\lim_n v_{n,k}=0,\quad\lim_n v_{n,n}\;\hbox{exists finite},\quad
\sum_{k=1}^n \left|v_{n,k}-v_{n,k-1}\right|=O(1)
$$
and $\sum_n z_n$ is convergent. Then
$$
\lim_n \sum_{k=1}^n v_{n,k}z_k=0.
$$
\end{cor}
\begin{proof}
Without loss of generality, we can suppose $\lim_n
v_{n,n}=s\in\{0,1\}$.  Set $w_n=\sum_{k=n}^{+\infty} z_k$ and observe
that, since $\sum_n z_n$ is convergent, we have $\lim_n w_n=w=0$ and,
moreover, we can write
\begin{equation*}
\sum_{k=1}^n v_{n,k}z_k\ =\ \sum_{k=1}^n v_{n,k}(w_k-w_{k+1})\
=\ \sum_{k=2}^n (v_{n,k}-v_{n,k-1})w_k +v_{n,1}w_1-v_{n,n}w_{n+1}.
\end{equation*}
The second and the third term obviously converge to zero. In order to
prove that the first term converges to zero, it is enough to apply
Lemma \ref{toeplitz-lemma-complex} with $z_{n,k}=v_{n,k}-v_{n,k-1}$.
\end{proof}

The above corollary is useful to get the following result for complex
random variables, which again slightly extends the version provided in
\cite[Lemma A.3]{ale-cri-ghi}:

\begin{lem}\label{lemma-serie-rv}
Let ${\mathcal H}=({\mathcal H}_n)_n$ be a filtration and
$(Y_n)_n$ a $\mathcal H$-adapted sequence of complex random variables
such that $E[Y_{n}|{\mathcal H}_{n-1}]\to Y$ almost surely. Moreover,
let $(c_n)_n$ be a sequence of strictly positive real numbers such that
$\sum_n E\left[|Y_n|^2\right]/c_n^2<+\infty$ and let $\{v_{n,k},1\leq
k\leq n\}$ be a triangular array of complex numbers such that
$v_{n,k}\neq 0$ and
\begin{equation}\label{cond-serie-rv-1}
\lim_n v_{n,k}=0,\quad
\lim_n v_{n,n}\;\hbox{exists finite},\quad
\lim_n\sum_{k=1}^n \frac{v_{n,k}}{c_k}= \eta\in{\mathbb C},
\end{equation}
\begin{equation}\label{cond-serie-rv-2}
\sum_{k=1}^n\frac{|v_{n,k}|}{c_k}=O(1),
\quad
\sum_{k=1}^n \left|v_{n,k}-v_{n,k-1}\right|=O(1).
\end{equation}
Then $\sum_{k=1}^n v_{n,k}Y_k/c_k\stackrel{a.s.}\longrightarrow \eta Y$.
\end{lem}
\begin{proof} Let $A$ be an event such that $P(A)=1$ and $\lim_n
E[Y_n|{\mathcal H}_{n-1}](\omega)=Y(\omega)$ for each $\omega\in
A$. Fix $\omega \in A$ and set $w_n=E[Y_n|{\mathcal H}_{n-1}](\omega)$
and $w=Y(\omega)$. If $\eta\neq 0$, applying Lemma
\ref{toeplitz-lemma-complex} to $z_{n,k}=v_{n,k}/(c_k\eta)$, $s=1$ and $w_n$,
we obtain
$$
\lim_n \sum_{k=1}^n v_{n,k}\frac{E[Y_k|{\mathcal H}_{k-1}](\omega)}{c_k\eta}=
Y(\omega).
$$
If $\eta=0$,  applying Lemma
\ref{toeplitz-lemma-complex} to $z_{n,k}=v_{n,k}/c_k$, $s=0$ and $w_n$,
we obtain
$$
\lim_n \sum_{k=1}^n v_{n,k}\frac{E[Y_k|{\mathcal H}_{k-1}](\omega)}{c_k}=
0.
$$
Therefore, for both cases, we have
$$
\sum_{k=1}^n v_{n,k}\frac{E[Y_k|{\mathcal H}_{k-1}]}{c_k}
\stackrel{a.s.}\longrightarrow \eta Y.
$$

Now, consider the martingale $(M_n)_n$ defined by
$$
M_n=\sum_{k=1}^n \frac{Y_k-E[Y_k|{\mathcal H}_{k-1}]}{c_k}.
$$
It is bounded in $L^2$ since $\sum_{k=1}^n
\frac{E[|Y_k|^2]}{c_k^2}<+\infty$ by assumption and so it is almost surely
convergent, that means $$ \sum_{k}
\frac{Y_k(\omega)-E[Y_k|{\mathcal H}_{k-1}](\omega)}{c_k}<+\infty
$$ for $\omega\in B$ with $P(B)=1$. Therefore, fixing $\omega\in B$
and setting $z_k=\frac{Y_k(\omega)-E[Y_k|{\mathcal
      H}_{k-1}](\omega)}{c_k}$, by Corollary \ref{kronecker-complex},
we get
$$
\lim_n \sum_{k=1}^n v_{n,k}
\frac{Y_k(\omega)-E[Y_k|{\mathcal H}_{k-1}](\omega)}{c_k}
=0
$$
and so
$$
\sum_{k=1}^n v_{n,k}
\frac{Y_k-E[Y_k|{\mathcal H}_{k-1}]}{c_k}
\stackrel{a.s.}\longrightarrow 0.
$$
In order to conclude, it is enough to observe that
$$
\sum_{k=1}^n v_{n,k}\frac{Y_k}{c_k}=
\sum_{k=1}^n v_{n,k}\frac{Y_k-E[Y_k|{\mathcal H}_{k-1}]}{c_k} +
\sum_{k=1}^n v_{n,k}\frac{E[Y_k|{\mathcal H}_{k-1}]}{c_k}.
$$
\end{proof}

We conclude this subsection recalling the following well-known
relations for $a\in{\mathbb R}$:
\begin{equation}\label{relazione-nota}
\sum_{k=1}^n\frac{1}{k^{1-a}}=
\begin{cases}
& O(1) \quad\mbox{for } a<0,\\
& \ln(n)+O(1) \quad\mbox{for } a=0,\\
&a^{-1}\, n^a + O(1) \quad\mbox{for } 0<a\leq 1,\\
&a^{-1}\, n^{a} + O(n^{a-1}) \quad\mbox{for } a>1.
\end{cases}
\end{equation}
More precisely, in the case $a=0$, we have
\begin{equation}\label{eulero_mascheroni}
d_n=\sum_{k=1}^n\frac{1}{k}-\ln(n)=
d+O(n^{-1})
\end{equation}
where $d$ denotes the Euler-Mascheroni constant.

\subsection{Asymptotic results for products of complex numbers}

Fix $\gamma=1$ and $c>0$, and consider a sequence $(r_n)_n$ of
real numbers such that $0\leq r_n<1$ for each $n$ and
\begin{equation}\label{ass:condition_r_n_1_appendix}
nr_n-c\ =\  O\left(n^{-1}\right).
\end{equation}
Obviously, we have $r_n>0$ for $n$ large enough and so in the sequel,
without loss of generality, we will assume $0<r_n<1$ for all $n$.
\\

\indent Let $x=a_x+i\,b_x\in{\mathbb C}$ and
$y=a_y+i\,b_y\in{\mathbb C}$ with $a_x,a_y>0$ and $c(a_x+a_y)\geq1$.
Denote by $m_0\geq 2$ an integer such that $\max\{a_x,\,a_y\}r_m<1$
for all $m\geq m_0$ and set:
$$
p_{m_0-1}(x):=1,\quad
p_{n}(x):=\prod_{m=m_0}^n (1-x r_m)\; \hbox{for } n\geq m_0
\quad\mbox{and}\quad
F_{k+1,n}(x):=\frac{p_{n}(x)}{p_{k}(x)}\; \hbox{for } m_0-1\leq k\leq n-1.
$$

We recall the following result, which has been proved in \cite{ale-cri-ghi}:
\begin{lem}\cite[Lemma A.4]{ale-cri-ghi}\label{lemma-tecnico_1}
We have that
\begin{equation*}
|p_n(x)|\ =\ O\left(n^{- ca_x}\right)\qquad\mbox{and}\qquad
|p^{-1}_n(x)|\ =\ O\left(n^{ca_x}\right).
\end{equation*}
\end{lem}

Inspired by the computation done in
\cite{ale-cri-ghi,cri-dai-lou-min}, we can prove the following
other technical result:

\begin{lem}\label{lemma-tecnico_2}
(i) When $c(a_x+a_y)=1$, we have
\begin{equation}\label{affermazione2_log}
\lim_n \frac{n}{\ln(n)}\sum_{k=m_0}^{n-1} r_k^2 F_{k+1,n}(x)F_{k+1,n}(y)\
=\
\left\{\begin{aligned}
&c^2\ &\mbox{ if }b_x+b_y=0,\\
&0\ &\mbox{ if } b_x+b_y\neq 0;
\end{aligned}\right.\end{equation}
while when $c(a_x+a_y)>1$, we have
\begin{equation}\label{affermazione2}\begin{aligned}
&\lim_n n\sum_{k=m_0}^{n-1} r_k^2 F_{k+1,n}(x)F_{k+1,n}(y)
\ =\
\frac{c^2}{c(x+y)-1},\\
&\lim_n n\sum_{k=m_0}^{n-1} r_k^2 \ln\left(\frac{n}{k}\right)F_{k+1,n}(x)F_{k+1,n}(y)
\ =\
\frac{c^2}{(c(x+y)-1)^2},\\
&\lim_n n\sum_{k=m_0}^{n-1}r_k^2\ln^2\left(\frac{n}{k}\right)F_{k+1,n}(x)F_{k+1,n}(y)
\ = \
\frac{2c^2}{(c(x+y)-1)^3}.
\end{aligned}\end{equation}
(ii) Moreover, for any $u\geq 1$, we have:\\
\noindent when $c(a_x+a_y)=1$
\begin{equation}\label{affermazione3_log}
\sum_{k=m_0}^{n-1} r_k^{2u}\frac{|p_{n}(x)|^u|p_{n}(y)|^u}{|p_{k}(x)|^u|p_{k}(y)|^u}
=\left\{\begin{aligned}
&O(\ln(n)/n)\quad&\mbox{for } u=1,\\
&O\left(n^{-u}\right)\quad&\mbox{for } u>1;
\end{aligned}\right.
\end{equation}
while when $c(a_x+a_y)>1$ and $e\in\{0,1,2\}$
\begin{equation}\label{affermazione3}
\sum_{k=m_0}^{n-1} r_k^{2u}\ln^{e u}
\left(\frac{n}{k}\right) \frac{|p_{n}(x)|^u|p_{n}(y)|^u}{|p_{k}(x)|^u|p_{k}(y)|^u}
=\begin{cases}
O(n^{-uc(a_x+a_y)}\ln^{eu}(n))\quad &\hbox{for } uc(a_x+a_y)<2u-1,\\
O(n^{-(2u-1)}\ln^{eu+1}(n))\quad &\hbox{for } uc(a_x+a_y)=2u-1,\\
O(n^{-(2u-1)})\quad &\hbox{for } uc(a_x+a_y)>2u-1
\end{cases}
\end{equation}
(note that for $u=1$ only the third case is possible).
\end{lem}

\begin{proof} (i) First of all, let us notice that the limit
\eqref{affermazione2_log} and the first of the limits
\eqref{affermazione2} have already been proved
in~\cite[Eq. (A.11),(A.18)]{ale-cri-ghi}. Therefore, we can focus on
the second and the third limits in \eqref{affermazione2}. To this end,
let us set
$$
\mathcal{S}_{1,n}:=\sum_{k=m_0}^{n-1} \frac{r_k^2}{p_k(x)p_k(y)},\qquad
\mathcal{S}_{2,n}:=\sum_{k=m_0}^{n-1} \frac{r_k^2\ln(k)}{p_k(x)p_k(y)},\qquad
\mathcal{S}_{3,n}:=\sum_{k=m_0}^{n-1} \frac{r_k^2\ln^2(k)}{p_k(x)p_k(y)},
$$
so that, recalling the equality $F_{k+1,n}(x)=p_n(x)/p_k(x)$, we can write:

\begin{equation*}
\begin{aligned}
&n\sum_{k=m_0}^{n-1} r_k^2 F_{k+1,n}(x)F_{k+1,n}(y)=
np_n(x)p_n(y)\mathcal{S}_{1,n},\\
&n\sum_{k=m_0}^{n-1} r_k^2 \ln\left(\frac{n}{k}\right)F_{k+1,n}(x)F_{k+1,n}(y)=
np_n(x)p_n(y)\left(\ln(n)\mathcal{S}_{1,n}-\mathcal{S}_{2,n}\right),\\
&n\sum_{k=m_0}^{n-1} r_k^2 \ln^2\left(\frac{n}{k}\right)F_{k+1,n}(x)F_{k+1,n}(y)=
np_n(x)p_n(y)
\left(\ln^2(n)\mathcal{S}_{1,n}-2\ln(n)\mathcal{S}_{2,n}+\mathcal{S}_{3,n}\right).
\end{aligned}
\end{equation*}
Now, set $G_{1,k}:=c^2/[kp_{k}(x)p_{k}(y)]$ and recall that, as seen
in~\cite[Proof of Lemma A.5]{ale-cri-ghi}, when $c(a_x+a_y)>1$ we have
\begin{equation}\label{irene1}
\Delta G_{1,k}
=(c(x+y)-1)\Delta\mathcal{S}_{1,k}
+O\left(k^{-1} |\Delta \mathcal{S}_{1,k}| \right).
\end{equation}
Using analogous arguments, we can set
 $G_{2,k}:=c^2\ln(k)/[kp_{k}(x)p_{k}(y)]$ and observe that we have:
\begin{equation*}
\begin{split}
\Delta G_{2,k}&=\frac{c^2}{p_k(x)p_k(y)}
\left[
\left(\frac{\ln(k)}{k}-\frac{\ln(k-1)}{k-1}\right)
\left(1-(x+y)r_k+r_k^2xy\right) +
\frac{\ln(k)}{k}\left((x+y)r_k-r_k^2xy\right)\right]
\\
&=\frac{c^2}{p_k(x)p_k(y)}
\left[
\left(-\frac{\ln(k)}{k^2}+\frac{1}{k^2}+o(k^{-2})\right)
\left(1-(x+y)r_k+r_k^2xy\right) +
\frac{\ln(k)}{k}\left((x+y)r_k-r_k^2xy\right)
\right]
\\
&=(c(x+y)-1)\Delta \mathcal{S}_{2,k}+\Delta \mathcal{S}_{1,k}
+O(k^{-1} |\Delta\mathcal{S}_{2,k}|).
\end{split}
\end{equation*}
Therefore, when $c(a_x+a_y)>1$, we obtain
\begin{equation}\label{irene2}
\frac{\Delta G_{2,k}}{c(x+y)-1}-\Delta\mathcal{S}_{2,k}
=\frac{\Delta\mathcal{S}_{1,k}}{c(x+y)-1}+
O\left(k^{-1}\ln(k) |\Delta\mathcal{S}_{1,k}| \right).
\end{equation}
The relations \eqref{irene1}, \eqref{irene2} and the first limit in
\eqref{affermazione2} imply
\begin{align*}
\lim_n & np_n(x)p_n(y)
\Big(\ln(n) \mathcal{S}_{1,n}-\mathcal{S}_{2,n}\Big)
\\
&=\lim_n np_n(x)p_n(y)
\Big(
\frac{\ln(n)G_{1,n}}{c(x+y)-1}-\mathcal{S}_{2,n}
\Big)+
O\Big(\ln(n)n |p_n(x)p_n(y)|\sum_{k=m_0}^{n-1} k^{-1}|\Delta\mathcal{S}_{1,k}|
\Big)
\\
&= \lim_n np_n(x)p_n(y)
\Big(
\frac{G_{2,n}}{c(x+y)-1}-\mathcal{S}_{2,n}
\Big)
\\
&= (c(x+y)-1)^{-1}
\lim_n np_n(x)p_n(y)\mathcal{S}_{1,n}+
O\Big(n |p_n(x)p_n(y)|\sum_{k=m_0}^{n-1} k^{-1}\ln(k)|\Delta\mathcal{S}_{1,k}|
\Big)
\\
&= (c(x+y)-1)^{-1}
\lim_n np_n(x)p_n(y)\mathcal{S}_{1,n}=\frac{c^2}{(c(x+1)-1)^2},
\end{align*}
where we have used the fact that, by Lemma \ref{lemma-tecnico_1} and
relation \eqref{relazione-nota}, we have
$$
O\Big(\ln(n)n |p_n(x)p_n(y)|
\sum_{k=m_0}^{n-1} k^{-1}|\Delta\mathcal{S}_{1,k}|\Big)
=
O\Big(\frac{\ln(n)}{n^{c(a_x+a_y)-1}}
\sum_{k=m_0}^{n-1} \frac{1}{k^{1-(c(a_x+a_y)-2)}}
\Big)
\longrightarrow 0.
$$

For the last limit, we can set
$G_{3,k}:=c^2\ln^2(k)/[kp_{k}(x)p_{k}(y)]$ and, similarly as above,
observe that we have:
\begin{equation*}
\begin{split}
 \Delta G_{3,k}
&=\frac{c^2}{p_k(x)p_k(y)}
\left[
\left(\frac{\ln^2(k)}{k}-\frac{\ln^2(k-1)}{k-1}\right)
\left(1-(x+y)r_k+r_k^2xy\right) +
\frac{\ln^2(k)}{k}\left((x+y)r_k-r_k^2xy\right)
\right]
\\
&=\frac{c^2}{p_k(x)p_k(y)}\times
\\
&\left[
\left(-\frac{\ln^2(k)}{k^2}+2\frac{\ln(k)}{k^2}+O(k^{-3}\ln^2(k))\right)
\left(1-(x+y)r_k+r_k^2xy\right)+
\frac{\ln^2(k)}{k}\left((x+y)r_k-r_k^2xy\right)
\right]
\\
&=(c(x+y)-1)\Delta \mathcal{S}_{3,k}+2\Delta \mathcal{S}_{2,k}+
O(k^{-1} |\Delta\mathcal{S}_{3,k}| ).
\end{split}
\end{equation*}
Therefore, when $c(a_x+a_y)>1$, we obtain
\begin{equation}\label{irene3}
\frac{\Delta G_{3,k}}{c(x+y)-1}- \Delta \mathcal{S}_{3,k}
=\frac{2\Delta \mathcal{S}_{2,k}}{c(x+y)-1}
+O(k^{-1}\ln^2(k) |\Delta\mathcal{S}_{1,k}| ).
\end{equation}
By means of analogous computations as above, the relations
\eqref{irene1}, \eqref{irene2}, \eqref{irene3} and the
already proved second limit in \eqref{affermazione2} imply
\begin{align*}
\lim_n & np_n(x)p_n(y)
\left(\ln^2(n)\mathcal{S}_{1,n}-2\ln(n)\mathcal{S}_{2,n}
+\mathcal{S}_{3,n}\right)
\\
&= \lim_n np_n(x)p_n(y)
\Big(\frac{\ln^2(n)G_{1,n}}{c(x+y)-1}
-2\ln(n)\mathcal{S}_{2,n}+\mathcal{S}_{3,n}\Big)
\!+\!O\Big(
\ln^2(n)n |p_n(x)p_n(y)| \sum_{k=m_0}^{n-1} k^{-1}|\Delta\mathcal{S}_{1,k}|
\Big)
\\
&= \lim_n np_n(x)p_n(y)
\Big(\frac{\ln(n) G_{2,n}}{c(x+y)-1}
-2\ln(n)\mathcal{S}_{2,n}+\mathcal{S}_{3,n}\Big)
\\
& = \lim_n np_n(x)p_n(y)
\Big(\frac{\ln(n)G_{2,n}}{c(x+y)-1}
-2\frac{ \ln(n)(G_{2,n}-\mathcal{S}_{1,n}) }{c(x+y)-1}
+\mathcal{S}_{3,n}\Big)
\\
& \qquad + O\Big(
\ln(n)n |p_n(x)p_n(y)|
\sum_{k=m_0}^{n-1} k^{-1}\ln(k)|\Delta\mathcal{S}_{1,k}|
\Big)
\\
&= \lim_n np_n(x)p_n(y)
\Big(\frac{2\ln(n)\mathcal{S}_{1,n}}{c(x+y)-1}
-\frac{G_{3,n}}{c(x+y)-1}
+\mathcal{S}_{3,n}\Big)
\\
& = \frac{2}{c(x+y)-1}\lim_n np_n(x)p_n(y)
\Big(\ln(n)\mathcal{S}_{1,n}-\mathcal{S}_{2,n}\Big)
\!+\!O\Big(n |p_n(x)p_n(y)|
\sum_{k=m_0}^{n-1} k^{-1}\ln^2(k)|\Delta\mathcal{S}_{1,k}|
\Big)
\\
& = \frac{2}{c(x+y)-1}\lim_n np_n(x)p_n(y)
\Big(\ln(n)\mathcal{S}_{1,n}-\mathcal{S}_{2,n}\Big)=
\frac{2c^2}{(c(x+1)-1)^3},
\end{align*}
where we have used the fact that, by Lemma \ref{lemma-tecnico_1} and
relation \eqref{relazione-nota}, we have
$$
O\Big(\ln^2(n) n |p_n(x)p_n(y)|
\sum_{k=m_0}^{n-1} k^{-1}|\Delta\mathcal{S}_{1,k}|\Big)
=
O\Big(\frac{\ln^2(n)}{n^{c(a_x+a_y)-1}}
\sum_{k=m_0}^{n-1} \frac{1}{k^{1-(c(a_x+a_y)-2)}}
\Big)
\longrightarrow 0.
$$

ii) For the second part of the proof, note that by
condition~\eqref{ass:condition_r_n_1_appendix} on $(r_n)_n$, relation
\eqref{relazione-nota} and Lemma~\ref{lemma-tecnico_1}, when
$c(a_x+a_y)=1$, we have
$$
\sum_{k=m_0}^{n-1} r_k^{2u}
\frac{|p_{n}(x)|^u|p_{n}(y)|^u}{|p_{k}(x)|^u|p_{k}(y)|^u}=
O(n^{-u})\sum_{k=m_0}^{n-1} O(k^{-u}) =
\left\{\begin{aligned}
&O(\ln(n)/n)\quad&\mbox{for } u=1,\\
&O\left(n^{-u}\right)\quad&\mbox{for } u>1.
\end{aligned}\right.$$
For the case $c(a_x+a_y)>1$, note that for $u\geq 1$ and
$e\in\{0,1,2\}$, we have
\[\begin{aligned}
&\sum_{k=m_0}^{n-1} r_k^{2u}\ln^{e u}\Big(\frac{n}{k}\Big)
\frac{|p_{n}(x)|^u|p_{n}(y)|^u}{|p_{k}(x)|^u|p_{k}(y)|^u} =
\sum_{k=m_0}^{n-1} O(k^{-2u})\ln^{e u}\Big(\frac{n}{k}\Big)
O\Big(\Big(\frac{k}{n}\Big)^{uc(a_x+a_y)}\Big) =\\
&n^{-2u}\sum_{k=m_0}^{n-1} \ln^{e u}\Big(\frac{n}{k}\Big)
O\Big(\Big(\frac{k}{n}\Big)^{u(c(a_x+a_y)-2)}\Big).
\end{aligned}\]
Then, for $e=0$, using relation \eqref{relazione-nota}, it is easy to
see that
$$
n^{-2u}\sum_{k=m_0}^{n-1}O\left(\left(\frac{k}{n}\right)^{u(c(a_x+a_y)-2)}\right)
=\begin{cases}
O(n^{-uc(a_x+a_y)})\quad &\hbox{for } uc(a_x+a_y)<2u-1,\\
O(n^{-(2u-1)}\ln(n))\quad &\hbox{for } uc(a_x+a_y)=2u-1,\\
O(n^{-(2u-1)})\quad &\hbox{for } uc(a_x+a_y)>2u-1
\end{cases}
$$
(note that for $u=1$ only the third case is possible).
\\

\indent Now we consider the cases $e=1$ and $e=2$. Note that, setting
$\alpha:=2u-uc(a_x+a_y)\in\mathbb{R}$ and $\beta:=eu\geq1$, we have
that
$$
\frac{1}{n}\sum_{k=m_0}^{n-1} \ln^{\beta}\Big(\frac{n}{k}\Big)\,
O\Big(\Big(\frac{k}{n}\Big)^{-\alpha}\Big)
\ =\
O(1)\ +\ O\left(\int_{\frac{m_0-1}{n}}^{\epsilon}x^{-\alpha}\ln^{\beta}(x^{-1})dx\right),
$$
where $\epsilon\in(0,1)$ has been chosen such that
$g(x)=x^{-\alpha}\ln^{\beta}(x^{-1})$ is monotone in $(0,\epsilon]$ and we
  recall that $(m_0-1)\geq 1$. Then, we have that
$$
\int_{\frac{m_0-1}{n}}^{\epsilon}x^{-\alpha}\ln^{\beta}(x^{-1})dx
=\begin{cases}
O(n^{\alpha-1}\ln^{\beta}(n))\quad &\hbox{for } \alpha>1\\
O(\ln^{\beta+1}(n))\quad &\hbox{for } \alpha=1,\\
O(1)\quad &\hbox{for } \alpha<1.
\end{cases}
$$
Finally, we can conclude that, for the cases $e=1$ and $e=2$, we have
$$
n^{-2u}\sum_{k=m_0}^{n-1}\ln^{e u}\Big(\frac{n}{k}\Big)
O\Big(\Big(\frac{k}{n}\Big)^{u(c(a_x+a_y)-2)}\Big)
=\begin{cases}
O(n^{-uc(a_x+a_y)}\ln^{eu}(n))\quad &\hbox{for } uc(a_x+a_y)<2u-1,\\
O(n^{-(2u-1)}\ln^{eu+1}(n))\quad &\hbox{for } uc(a_x+a_y)=2u-1,\\
O(n^{-(2u-1)})\quad &\hbox{for } uc(a_x+a_y)>2u-1
\end{cases}
$$
(note again that for $u=1$ only the third case is possible).
\end{proof}

\begin{rem}\label{rem:relations_diff_v}
\rm Setting $v^{(e)}_{n,k}:=(n/k)\ln^e(n/k)F_{k+1,n}(x)F_{k+1,n}(y)$
for any $e\in\{0,1,2\}$ and $m_0-1\leq k\leq n-1$, and using the
relations \eqref{irene1}, \eqref{irene2}, \eqref{irene3} found in the
proof of Lemma~\ref{lemma-tecnico_2}, for $c(a_x+a_y)>1$ we have:
\[
\begin{aligned}
&|v^{(0)}_{n,k}-v^{(0)}_{n,k-1}|\ =\ n|p_n(x)p_n(y)|O(|\Delta G_{1,k}|)\ =
\ n|p_n(x)p_n(y)|O(|\Delta \mathcal{S}_{1,k}|)\ =\
O\Big(n r_k^2\frac{|p_{n}(x)||p_{n}(y)|}{|p_{k}(x)||p_{k}(y)|}\Big);
\\
&|v^{(1)}_{n,k}-v^{(1)}_{n,k-1}|\ =
\ n|p_n(x)p_n(y)|O(|\ln(n)\Delta G_{1,k}-\Delta G_{2,k}|) \\
& \quad = n|p_n(x)p_n(y)|O(|\ln(n)\Delta \mathcal{S}_{1,k}-\Delta \mathcal{S}_{2,k}|
+|\Delta \mathcal{S}_{1,k}|)\ =
O\Big(nr_k^{2}\left(\ln\left(\frac{n}{k}\right)+1\right)
\frac{|p_{n}(x)||p_{n}(y)|}{|p_{k}(x)||p_{k}(y)|}
\Big);
\\
&|v^{(2)}_{n,k}-v^{(2)}_{n,k-1}|\ =\
n|p_n(x)p_n(y)|O(|\ln^2(n)\Delta G_{1,k}-2\ln(n)\Delta G_{2,k}+\Delta G_{3,k}|)\
\\
&\quad =n|p_n(x)p_n(y)|
O\left(
|\ln^2(n)\Delta \mathcal{S}_{1,k}-2\ln(n)\Delta\mathcal{S}_{2,k}
+\Delta \mathcal{S}_{3,k}|
+|\ln(n)\Delta\mathcal{S}_{1,k}-\Delta\mathcal{S}_{2,k}|
\right)\\
&\quad = O\Big(
nr_k^2
\Big(\ln^2\Big(\frac{n}{k}\Big)+\ln\Big(\frac{n}{k}\Big)
\Big) \frac{|p_{n}(x)||p_{n}(y)|}{|p_{k}(x)||p_{k}(y)|}
\Big),
\end{aligned}
\]
Moreover, setting $v_{n,k}:=v^{(0)}_{n,k}/\ln(n)$ for any $m_0-1\leq
k\leq n-1$, in the case $c(a_x+a_y)=1$ we have:
 $|v_{n,k}-v_{n,k-1}|=O\left(r_k^2 k/\ln(n)\right)$ when $b_x+b_y\neq 0$
since Lemma \ref{lemma-tecnico_1} and
$$|v^{(0)}_{n,k}-v^{(0)}_{n,k-1}|\ =\ n|p_n(x)p_n(y)|O(|\Delta G_{1,k}|)\ =\
n|p_n(x)p_n(y)|O(|\Delta \mathcal{S}_{1,k}|)\ =\
O\Big(nr_k^{2}\frac{|p_{n}(x)||p_{n}(y)|}{|p_{k}(x)||p_{k}(y)|}
\Big);
$$
while $|v_{n,k}-v_{n,k-1}|=O\left(r_k^{2}/\ln(n)\right)$ when $b_x+b_y=0$ since
 Lemma \ref{lemma-tecnico_1} and
$$|v^{(0)}_{n,k}-v^{(0)}_{n,k-1}|\ =\ n|p_n(x)p_n(y)|O(|\Delta G_{1,k}|)\ =\
n|p_n(x)p_n(y)|O(k^{-1}|\Delta \mathcal{S}_{1,k}|)\ =\
O\Big(r_k^{2}\frac{n|p_{n}(x)||p_{n}(y)|}{k|p_{k}(x)||p_{k}(y)|}\Big).
$$
\end{rem}

\subsection{Technical computations for the proofs of Theorem~\ref{thm:asymptotics_theta_hat_1} and
Theorem~\ref{thm:asymptotics_theta_1_star}}

In this subsection we collect some technical computations necessary for the proofs of
Theorem~\ref{thm:asymptotics_theta_hat_1} and Theorem~\ref{thm:asymptotics_theta_1_star}.
Therefore, the notation and the assumptions used here are the same as those used in these theorems.\\

The first technical result is the following:

\begin{lem}\label{lem:terms_in_A_k_n}
Let the matrix $A_{k+1,n}$ be defined as in~\eqref{def:C_k_n_A_k_n}
for $m_0-1\leq k\leq n-1$.  Then, we have that
\begin{equation*}\begin{aligned}
&[A^{11}_{k+1,n}]_{jj}\ =\ F_{k+1,n}(\alpha_j),\\
&[A^{33}_{k+1,n}]_{jj}\ =\ a^{22}_{k+1,n}\ =\ F_{k+1,n}(c^{-1}),\\
&[A^{31}_{k+1,n}]_{jj}\ =\
\begin{cases}
\left(\frac{1-\alpha_j}{c\alpha_j-1}\right)
(F_{k+1,n}(c^{-1})-F_{k+1,n}(\alpha_j)),\ &\hbox{for } c\alpha_j\neq 1,\\
(1-c^{-1})F_{k+1,n}(c^{-1})\ln\left(\frac{n}{k}\right)+O(n^{-1}),\
&\hbox{for } c\alpha_j=1.
\end{cases}
\end{aligned}\end{equation*}
\end{lem}

\begin{proof}
By means of~\eqref{def:S_Q_S_R} and~\eqref{def:C_k_n_A_k_n}, after
standard calculations, the elements in $A_{k+1,n}$ for $m_0-1\leq
k\leq n-1$ can be written as follows:
$[A^{11}_{k+1,n}]_{jj}=F_{k+1,n}(\alpha_j)$, $[A^{33}_{k+1,n}]_{jj}=
a^{22}_{k+1,n}= F_{k+1,n}(c^{-1})$ and
$$
[A^{31}_{k+1,n}]_{jj}\ =\
(1-\alpha_j)
\frac{p_{n}(\alpha_j)}{p_{k}(c^{-1})}
S^j_{k+1,n},
$$
where
\begin{equation*}
S^j_{k+1,n}:=\sum_{l=k+1}^n\Big(\frac{r_lc^{-1}}{1-r_lc^{-1}}\Big)X^j_{l}
\qquad\mbox{and}\qquad
X^j_{l}:=\frac{p_{l}(c^{-1})}{p_{l}(\alpha_j)}.
\end{equation*}
Setting $\Delta X^j_{l}:=(X^j_{l}-X^j_{l-1})$, notice that we have
$$
\Delta X^j_{l}
=\Big(\frac{1-r_{l}c^{-1}}{1-r_{l}\alpha_j}-1\Big)X^j_{l-1}
=(c\alpha_j-1)\Big(\frac{r_lc^{-1}}{1-r_{l}\alpha_j}\Big)X^j_{l-1}
=(c\alpha_j-1)\Big(\frac{r_lc^{-1}}{1-r_lc^{-1}}\Big)X^j_l.
$$
Hence, in the case $c\alpha_j\neq 1$, we have that
$$
(X^j_{n}-X^j_{k})=\sum_{l=k+1}^{n}\Delta X^j_{l}
=(c\alpha_j-1)S^j_{k+1,n},
$$
which implies
$$
S^j_{k+1,n}=
\frac{X^j_{n}-X^j_{k}}{c\alpha_j-1}=(c\alpha_j-1)^{-1}
\Big(\frac{p_{n}(c^{-1})}{p_{n}(\alpha_j)}-\frac{p_{k}(c^{-1})}{p_{k}(\alpha_j)}
\Big).
$$
Using the above expression of $S^j_{k+1,n}$ in the definition of
$A^{31}_{k+1,n}$, we
obtain (for $c\alpha_j\neq 1$) that
\[
[A^{31}_{k+1,n}]_{jj}=
\frac{1-\alpha_j}{c\alpha_j-1}
\frac{p_{n}(\alpha_j)}{p_{k}(c^{-1})}
\Big(
\frac{p_{n}(c^{-1})}{p_{n}(\alpha_j)}-\frac{p_{k}(c^{-1})}{p_{k}(\alpha_j)}
\Big)=
\Big(\frac{1-\alpha_j}{c\alpha_j-1}\Big)
\Big(F_{k+1,n}(c^{-1})-F_{k+1,n}(\alpha_j)\Big).
\]
When $c\alpha_j=1$, observing that $X^j_{l}=1$ for any $l\geq 1$ and
using condition~\eqref{ass:condition_r_n_1_appendix} we get
$$
S^j_{k+1,n}
=
\sum_{l=k+1}^{n}\frac{r_lc^{-1}}{1-r_lc^{-1}}
=
\sum_{l=k+1}^n \frac{1}{l-1}+\sum_{l=k+1}^nO\Big(\frac{1}{l^2}\Big)
=
\sum_{l=k}^n\frac{1}{l}-\frac{1}{n}+O\Big(\sum_{l\geq k}\frac{1}{l^2}\Big)
=
\sum_{l=k}^n \frac{1}{l}+O(k^{-1}),
$$
where, for the last equality, we have used the fact that $k<n$ and
$\sum_{l\geq k}1/l^2=O(1/k)$. Then, using \eqref{eulero_mascheroni} for $a=0$,
we have
$$\sum_{l=k}^n \frac{1}{l}=\ln\Big(\frac{n}{k}\Big)+d_n-d_k=
\ln\Big(\frac{n}{k}\Big)+O(n^{-1})-O(k^{-1})=
\ln\Big(\frac{n}{k}\Big)+O(k^{-1})
$$
(where the last passage follows again by the fact that $k<n$).
Finally, since Lemma~\ref{lemma-tecnico_1} we have
$|F_{k+1,n}(c^{-1})|=O(k/n)$, we obtain (for $c\alpha_j=1$) that
$$
[A^{31}_{k+1,n}]_{jj}
=
(1-c^{-1})\frac{p_n(c^{-1})}{p_k(c^{-1})}
\Big(\ln(n/k)+O(1/k)\Big)
=
(1-c^{-1})F_{k+1,n}(c^{-1})\ln\Big(\frac{n}{k}\Big)+O(n^{-1}).
$$
\end{proof}

\subsubsection{Computations for the almost sure limits of the elements in \eqref{eq:matrix}}\label{subsubsection_appendix_technical_computation_1}

\begin{itemize}
\item
\textit{$a.s.-\lim_{n}n\sum_{k=m_0}^{n-1} r_k^2 [A^{1}_{k+1,n}B_{k+1}A^{1}_{k+1,n}]_{h,j}$}:\\
By using the first limit in~\eqref{eq:results_limits}, we have
\begin{multline*}
n\sum_{k=m_0}^{n-1} r_k^2 [B_{k+1}]_{h,j}[A^{1}_{k+1,n}]_{h,h}[A^{1}_{k+1,n}]_{j,j} =\
n\sum_{k=m_0}^{n-1} r_k^2 [B_{k+1}]_{h,j}F_{k+1,n}(\alpha_h)F_{k+1,n}(\alpha_j)\\
\stackrel{a.s}\longrightarrow\
\frac{c^2}{c(\alpha_h+\alpha_j)-1}(\mathbf{v}_h^{\top}\mathbf{v}_j)
Z_{\infty}(1-Z_{\infty}).
\end{multline*}
\item
\textit{$a.s.-\lim_{n}n\sum_{k=m_0}^{n-1} r_k^2
  [A^{3}_{k+1,n}B_{k+1}A^{3}_{k+1,n}]_{h,j}$}:\\
First, note that when
$c\alpha_h\neq 1$ and $c\alpha_j\neq 1$, we have that
$n\sum_{k=m_0}^{n-1}
r_k^2[B_{k+1}]_{h,j}[A^{3}_{k+1,n}]_{h,h}[A^{3}_{k+1,n}]_{j,j}$ has the
same limit as
\begin{multline*}
\frac{(1-c^{-1})^2}{(c\alpha_h-1)(c\alpha_j-1)}
\,n\sum_{k=m_0}^{n-1} r_k^2[B_{k+1}]_{h,j}F^2_{k+1,n}(c^{-1})\\
\begin{aligned}
&+
\ \frac{(1-\alpha_h)(1-\alpha_j)}{(c\alpha_h-1)(c\alpha_j-1)}
\,n\sum_{k=m_0}^{n-1} r_k^2[B_{k+1}]_{h,j}F_{k+1,n}(\alpha_h)F_{k+1,n}(\alpha_j)\\
&-\frac{(1-\alpha_h)(1-c^{-1})}{(c\alpha_h-1)(c\alpha_j-1)}
\,n\sum_{k=m_0}^{n-1} r_k^2[B_{k+1}]_{h,j}F_{k+1,n}(\alpha_h)F_{k+1,n}(c^{-1})\\
&-\frac{(1-\alpha_j)(1-c^{-1})}{(c\alpha_h-1)(c\alpha_j-1)}
\,n\sum_{k=m_0}^{n-1} r_k^2[B_{k+1}]_{h,j}F_{k+1,n}(\alpha_j)F_{k+1,n}(c^{-1}).
\end{aligned}
\end{multline*}
Then, when $c\alpha_h\neq 1$ and $c\alpha_j\neq 1$, using the first
limit in~\eqref{eq:results_limits} we obtain, after some standard
calculations,
$$
n\sum_{k=m_0}^{n-1} r_k^2 [B_{k+1}]_{h,j}[A^{3}_{k+1,n}]_{h,h}[A^{3}_{k+1,n}]_{j,j}\
\stackrel{a.s}\longrightarrow\
\frac{1+(c-1)(\alpha_h^{-1}+\alpha_j^{-1})}
{c(\alpha_h+\alpha_j)-1}(\mathbf{v}_h^{\top}\mathbf{v}_j)Z_{\infty}(1-Z_{\infty}).
$$
When $c\alpha_h=c\alpha_j=1$, we have that $n\sum_{k=m_0}^{n-1}
r_k^2[B_{k+1}]_{h,j}[A^{3}_{k+1,n}]_{h,h}[A^{3}_{k+1,n}]_{j,j}$ has the
same limit as
\begin{multline*}
(1-c^{-1})^2\,n\sum_{k=m_0}^{n-1} \ln^2(n/k)r_k^2[B_{k+1}]_{h,j}F^2_{k+1,n}(c^{-1})\\
\begin{aligned}
&+2c^{-1}(1-c^{-1})\,n\sum_{k=m_0}^{n-1} \ln(n/k)r_k^2[B_{k+1}]_{h,j}F^2_{k+1,n}(c^{-1})
\\
&+ c^{-2}\,n\sum_{k=m_0}^{n-1} r_k^2[B_{k+1}]_{h,j}F^2_{k+1,n}(c^{-1}),
\end{aligned}
\end{multline*}
from which, using the three limits in~\eqref{eq:results_limits}, we
obtain
$$n\sum_{k=m_0}^{n-1} r_k^2 [B_{k+1}]_{h,j}[A^{3}_{k+1,n}]_{h,h}[A^{3}_{k+1,n}]_{j,j}\
\stackrel{a.s}\longrightarrow\
(1+2c(c-1))(\mathbf{v}_h^{\top}\mathbf{v}_j)Z_{\infty}(1-Z_{\infty}).
$$
Finally, when $c\alpha_h\neq 1$ and $c\alpha_j=1$, we have that
$n\sum_{k=m_0}^{n-1}
r_k^2[B_{k+1}]_{h,j}[A^{3}_{k+1,n}]_{h,h}[A^{3}_{k+1,n}]_{j,j}$ has the
same limit as
\begin{multline*}
\frac{(1-c^{-1})^2}{(c\alpha_h-1)}\,
n\sum_{k=m_0}^{n-1} \ln(n/k)r_k^2[B_{k+1}]_{h,j}F^2_{k+1,n}(c^{-1})\\
\begin{aligned}
&+\frac{c^{-1}(1-c^{-1})}{(c\alpha_h-1)}\,
n\sum_{k=m_0}^{n-1} r_k^2[B_{k+1}]_{h,j}F^2_{k+1,n}(c^{-1})\\
&-\frac{(1-\alpha_h)(1-c^{-1})}{(c\alpha_h-1)}\,
n\sum_{k=m_0}^{n-1} \ln(n/k)r_k^2[B_{k+1}]_{h,j}F_{k+1,n}(\alpha_h)F_{k+1,n}(c^{-1})\\
&-\frac{c^{-1}(1-\alpha_h)}{(c\alpha_h-1)}\,
n\sum_{k=m_0}^{n-1} r_k^2[B_{k+1}]_{h,j}F_{k+1,n}(\alpha_h)F_{k+1,n}(c^{-1}),
\end{aligned}
\end{multline*}
which implies, using the first two limits in~\eqref{eq:results_limits}, that
$$
n\sum_{k=m_0}^{n-1} r_k^2 [B_{k+1}]_{h,j}[A^{3}_{k+1,n}]_{h,h}[A^{3}_{k+1,n}]_{j,j}\
\stackrel{a.s}\longrightarrow\
\frac{1+(c-1)(c+\alpha_h^{-1})}
{c\alpha_h}(\mathbf{v}_h^{\top}\mathbf{v}_j)Z_{\infty}(1-Z_{\infty}).
$$
The case $c\alpha_h=1$ and $c\alpha_j\neq1$ is
analogous. Therefore, we can summarize the limits in all the above
cases with the formula:
$$
\frac{1+(c-1)(\alpha_h^{-1}+\alpha_j^{-1})}
{c(\alpha_h+\alpha_j)-1}(\mathbf{v}_h^{\top}\mathbf{v}_j)Z_{\infty}(1-Z_{\infty}).
$$
\item
\textit{$a.s.-\lim_{n}n\sum_{k=m_0}^{n-1} r_k^2 (a^{2}_{k+1,n})^2b_{k+1}$}:
\\
Using the first limit in~\eqref{eq:results_limits}, we have
$$
n\sum_{k=m_0}^{n-1} r_k^2 (a^{2}_{k+1,n})^2b_{k+1}\ =\
(c^{-1}-1)^2n\sum_{k=m_0}^{n-1} r_k^2 b_{k+1}F^2_{k+1,n}(c^{-1})\
\stackrel{a.s}\longrightarrow\ (c-1)^2\|\mathbf{v}_1\|^2Z_{\infty}(1-Z_{\infty}).
$$
\item
\textit{$a.s.-\lim_{n}n\sum_{k=m_0}^{n-1} r_k^2 [A^{1}_{k+1,n}B_{k+1}A^{3}_{k+1,n}]_{h,j}$}:
\\
First, when $c\alpha_j\neq 1$ notice that
$n\sum_{k=m_0}^{n-1} r_k^2[B_{k+1}]_{h,j}[A^{1}_{k+1,n}]_{h,h}[A^{3}_{k+1,n}]_{j,j}$
has the same limit as
\[
\frac{1-c^{-1}}{c\alpha_j-1}\!
n\!\sum_{k=m_0}^{n-1}\! r_k^2[B_{k+1}]_{h,j}F_{k+1,n}(\alpha_h)F_{k+1,n}(c^{-1})
-\frac{1-\alpha_j}{c\alpha_j-1}
n\!\sum_{k=m_0}^{n-1}\! r_k^2[B_{k+1}]_{h,j}F_{k+1,n}(\alpha_h)F_{k+1,n}(\alpha_j),
\]
and hence, after standard calculations, we obtain
$$
n\sum_{k=m_0}^{n-1} r_k^2 [B_{k+1}]_{h,j}[A^{1}_{k+1,n}]_{h,h}[A^{3}_{k+1,n}]_{j,j}\
\stackrel{a.s}\longrightarrow\
\frac{\alpha_h^{-1}(c-1)+c}
{c(\alpha_h+\alpha_j)-1}(\mathbf{v}_h^{\top}\mathbf{v}_j)Z_{\infty}(1-Z_{\infty}).
$$
When $c\alpha_j=1$, $n\sum_{k=m_0}^{n-1}
r_k^2[B_{k+1}]_{h,j}[A^{1}_{k+1,n}]_{h,h}[A^{3}_{k+1,n}]_{j,j}$ has the
same limit as
\[
(1-c^{-1})
n\sum_{k=m_0}^{n-1} \ln(n/k)r_k^2[B_{k+1}]_{h,j}F_{k+1,n}(\alpha_h)F_{k+1,n}(c^{-1}) 
+ c^{-1}n\sum_{k=m_0}^{n-1} r_k^2[B_{k+1}]_{h,j}F_{k+1,n}(\alpha_h)F_{k+1,n}(c^{-1}),
\]
and hence
$$
n\sum_{k=m_0}^{n-1} r_k^2 [B_{k+1}]_{h,j}[A^{1}_{k+1,n}]_{h,h}[A^{3}_{k+1,n}]_{j,j}\
\stackrel{a.s}\longrightarrow\
\frac{\alpha_h^{-1}(c-1)+c}{c\alpha_h}(\mathbf{v}_h^{\top}\mathbf{v}_j)
Z_{\infty}(1-Z_{\infty}).$$
Therefore we can summarize the limits of the above two cases with the formula
$$
\frac{\alpha_h^{-1}(c-1)+c}
{c(\alpha_h+\alpha_j)-1}(\mathbf{v}_h^{\top}\mathbf{v}_j)Z_{\infty}(1-Z_{\infty}).
$$
\item
\textit{$a.s.-\lim_{n}n\sum_{k=m_0}^{n-1} r_k^2
  a^{2}_{k+1,n}[\mathbf{b}_{k+1}^{\top}A^{1}_{k+1,n}]_{j}$}: \\
Notice
that
$$
n\sum_{k=m_0}^{n-1} r_k^2
[\mathbf{b}_{k+1}]_j[A^{1}_{k+1,n}]_{jj}a^{2}_{k+1,n}=
(c^{-1}-1)n\sum_{k=m_0}^{n-1} r_k^2
[\mathbf{b}_{k+1}]_jF_{k+1,n}(\alpha_j)F_{k+1,n}(c^{-1}),
$$
which implies that
$$
n\sum_{k=m_0}^{n-1} r_k^2 [\mathbf{b}_{k+1}]_j[A^{1}_{k+1,n}]_{jj}a^{2}_{k+1,n}\
\stackrel{a.s}\longrightarrow\
\frac{1-c}{\alpha_j}(\mathbf{v}_1^{\top}\mathbf{v}_j)Z_{\infty}(1-Z_{\infty}).
$$
\item
\textit{$a.s.-\lim_{n}n\sum_{k=m_0}^{n-1} r_k^2
  a^{2}_{k+1,n}[\mathbf{b}_{k+1}^{\top}A^{3}_{k+1,n}]_{j}$}:
\\
First,
when $c\alpha_j\neq 1$, notice that\\ $n\sum_{k=m_0}^{n-1}
r_k^2[\mathbf{b}_{k+1}]_j[A^{3}_{k+1,n}]_{jj}a^{2}_{k+1,n}$ has the
same limit as
\[
\begin{aligned}
&\frac{(1-c^{-1})(1-\alpha_j)}{c\alpha_j-1}
n\sum_{k=m_0}^{n-1} r_k^2[\mathbf{b}_{k+1}]_jF_{k+1,n}(\alpha_j)F_{k+1,n}(c^{-1})\\
&- \frac{(1-c^{-1})^2}{c\alpha_j-1}n
\sum_{k=m_0}^{n-1} r_k^2[\mathbf{b}_{k+1}]_j F_{k+1,n}^2(c^{-1}),
\end{aligned}
\]
which implies after some calculations
$$n\sum_{k=m_0}^{n-1} r_k^2 [\mathbf{b}_{k+1}]_j[A^{3}_{k+1,n}]_{jj}a^{2}_{k+1,n}\
\stackrel{a.s}\longrightarrow\
\frac{1-c}{\alpha_j}(\mathbf{v}_1^{\top}\mathbf{v}_j)Z_{\infty}(1-Z_{\infty}).
$$
When $c\alpha_j=1$, $n\sum_{k=m_0}^{n-1}
r_k^2[\mathbf{b}_{k+1}]_j[A^{3}_{k+1,n}]_{jj}a^{2}_{k+1,n}$ has the
same limit as
\[
-(1-c^{-1})^2n\sum_{k=m_0}^{n-1}\ln(n/k)r_k^2[\mathbf{b}_{k+1}]_jF^2_{k+1,n}(c^{-1}) 
-c^{-1}(1-c^{-1})n\sum_{k=m_0}^{n-1} r_k^2[\mathbf{b}_{k+1}]_jF^2_{k+1,n}(c^{-1}),
\]
from which we can obtain
$$
n\sum_{k=m_0}^{n-1} r_k^2[\mathbf{b}_{k+1}]_j[A^{3}_{k+1,n}]_{jj}a^{2}_{k+1,n}\
\stackrel{a.s}\longrightarrow\
c(1-c)(\mathbf{v}_1^{\top}\mathbf{v}_j)Z_{\infty}(1-Z_{\infty}).
$$
Therefore, we can summarize the limits of the above two cases with the formula
$$
\frac{1-c}{\alpha_j}(\mathbf{v}_1^{\top}\mathbf{v}_j)Z_{\infty}(1-Z_{\infty}).
$$
\end{itemize}

\subsubsection{Computations for the almost sure limits of the elements in~\eqref{eq:matrix_star}}\label{subsubsection_appendix_technical_computation_1_star}

\begin{itemize}
\item
\textit{$a.s.-\lim_{n}\frac{n}{\ln(n)}\sum_{k=m_0}^{n-1} r_k^2
  [A^{1}_{k+1,n}B_{k+1}A^{1}_{k+1,n}]_{h,j}$}: \\
By
using~\eqref{eq:results_limit_star}, we have
\begin{multline*}
\frac{n}{\ln(n)}
\sum_{k=m_0}^{n-1} r_k^2 [B_{k+1}]_{h,j}[A^{1}_{k+1,n}]_{h,h}[A^{1}_{k+1,n}]_{j,j}
=
\frac{n}{\ln(n)}
\sum_{k=m_0}^{n-1} r_k^2 [B_{k+1}]_{h,j}F_{k+1,n}(\alpha_h)F_{k+1,n}(\alpha_j)
\\
\stackrel{a.s}\longrightarrow
(\mathbf{v}_h^{\top}\mathbf{v}_j)Z_{\infty}(1-Z_{\infty})
\begin{cases}
c^2\ &\mbox{ if }b_{\alpha_h}+b_{\alpha_j}=0,\\
0\ &\mbox{ if }b_{\alpha_h}+b_{\alpha_j}\neq 0.
\end{cases}
\end{multline*}
\item
\textit{$a.s.-\lim_{n}\frac{n}{\ln(n)}\sum_{k=m_0}^{n-1} r_k^2
  [A^{3}_{k+1,n}B_{k+1}A^{3}_{k+1,n}]_{h,j}$}:\\
Since
$c(\alpha_h+\alpha_j)=1$ implies $c\alpha_h\neq 1$ and $c\alpha_j\neq
1$, we have that
$$\frac{n}{\ln(n)}\sum_{k=m_0}^{n-1} r_k^2[B_{k+1}]_{h,j}
[A^{3}_{k+1,n}]_{h,h}[A^{3}_{k+1,n}]_{j,j}$$
has the same limit as
\begin{multline*}
\frac{(1-c^{-1})^2}{(c\alpha_h-1)(c\alpha_j-1)}\,
\frac{n}{\ln(n)}\sum_{k=m_0}^{n-1} r_k^2[B_{k+1}]_{h,j}F^2_{k+1,n}(c^{-1})\\
\begin{aligned}
&+\ \frac{(1-\alpha_h)(1-\alpha_j)}{(c\alpha_h-1)(c\alpha_j-1)}\,
\frac{n}{\ln(n)}
\sum_{k=m_0}^{n-1} r_k^2[B_{k+1}]_{h,j}F_{k+1,n}(\alpha_h)F_{k+1,n}(\alpha_j)\\
&-\frac{(1-\alpha_h)(1-c^{-1})}{(c\alpha_h-1)(c\alpha_j-1)}\,
\frac{n}{\ln(n)}
\sum_{k=m_0}^{n-1} r_k^2[B_{k+1}]_{h,j}F_{k+1,n}(\alpha_h)F_{k+1,n}(c^{-1})\\
&-\frac{(1-\alpha_j)(1-c^{-1})}{(c\alpha_h-1)(c\alpha_j-1)}\,
\frac{n}{\ln(n)}
\sum_{k=m_0}^{n-1} r_k^2[B_{k+1}]_{h,j}F_{k+1,n}(\alpha_j)F_{k+1,n}(c^{-1}),
\end{aligned}
\end{multline*}
which is equal to
$$
o(1) + \left(\frac{(\alpha_h-1)(\alpha_j-1)}{c^2\alpha_h\alpha_j}\right)
\frac{n}{\ln(n)}
\sum_{k=m_0}^{n-1} r_k^2[B_{k+1}]_{h,j}F_{k+1,n}(\alpha_h)F_{k+1,n}(\alpha_j).
$$
Hence, we have that
\begin{multline*}
\frac{n}{\ln(n)}
\sum_{k=m_0}^{n-1} r_k^2 [B_{k+1}]_{h,j}[A^{3}_{k+1,n}]_{h,h}[A^{3}_{k+1,n}]_{j,j}\
\\
\stackrel{a.s}\longrightarrow
(\mathbf{v}_h^{\top}\mathbf{v}_j)Z_{\infty}(1-Z_{\infty})
\begin{cases}
\frac{(\alpha_h-1)(\alpha_j-1)}{\alpha_h\alpha_j}\
&\mbox{ if }b_{\alpha_h}+b_{\alpha_j}=0,\\
0\
&\mbox{ if }b_{\alpha_h}+b_{\alpha_j}\neq 0.
\end{cases}
\end{multline*}
\item
\textit{$a.s.-\lim_{n}\frac{n}{\ln(n)}\sum_{k=m_0}^{n-1} r_k^2b_{k+1}
  (a^{2}_{k+1,n})^2$}: \\
Since the calculations are analogous to those in
Subsection~\ref{subsubsection_appendix_technical_computation_1}, we
have $$\frac{n}{\ln(n)}\sum_{k=m_0}^{n-1} r_k^2 b_{k+1}(a^{2}_{k+1,n})^2
\stackrel{a.s}\longrightarrow 0.$$
\item
\textit{$a.s.-\lim_{n}\frac{n}{\ln(n)}\sum_{k=m_0}^{n-1} r_k^2
  [A^{1}_{k+1,n}B_{k+1}A^{3}_{k+1,n}]_{h,j}$}:\\
Since $c(\alpha_h+\alpha_j)=1$ implies $c\alpha_j\neq 1$, we have that
$$\frac{n}{\ln(n)}\sum_{k=m_0}^{n-1}
r_k^2[B_{k+1}]_{h,j}[A^{1}_{k+1,n}]_{h,h}[A^{3}_{k+1,n}]_{j,j}$$ has the
same limit as
\begin{align*}
\left(\frac{1-c^{-1}}{c\alpha_j-1}\right)&
\frac{n}{\ln(n)}
\sum_{k=m_0}^{n-1} r_k^2[B_{k+1}]_{h,j}F_{k+1,n}(\alpha_h)F_{k+1,n}(c^{-1})
\\
&\qquad\qquad- \left(\frac{1-\alpha_j}{c\alpha_j-1}\right)
\frac{n}{\ln(n)}
\sum_{k=m_0}^{n-1} r_k^2[B_{k+1}]_{h,j}F_{k+1,n}(\alpha_h)F_{k+1,n}(\alpha_j)
\\
&=o(1)- \left(\frac{1-\alpha_j}{c\alpha_j-1}\right)
\frac{n}{\ln(n)}
\sum_{k=m_0}^{n-1} r_k^2[B_{k+1}]_{h,j}F_{k+1,n}(\alpha_h)F_{k+1,n}(\alpha_j).
\end{align*}
Hence, we have
\begin{multline*}
\frac{n}{\ln(n)}
\sum_{k=m_0}^{n-1} r_k^2[B_{k+1}]_{h,j}[A^{1}_{k+1,n}]_{h,h}[A^{3}_{k+1,n}]_{j,j}\\
\stackrel{a.s}\longrightarrow
(\mathbf{v}_h^{\top}\mathbf{v}_j)
Z_{\infty}(1-Z_{\infty})
\begin{cases}
\frac{c^2(\alpha_j-1)}{c\alpha_j-1}=
\frac{c(1-\alpha_j)}{\alpha_h}\
&\mbox{ if }b_{\alpha_h}+b_{\alpha_j}=0,\\
0\
&\mbox{ if } b_{\alpha_h}+b_{\alpha_j}\neq 0.
\end{cases}
\end{multline*}
\item
\textit{$a.s.-\lim_{n}\frac{n}{\ln(n)}\sum_{k=m_0}^{n-1} r_k^2
  a^{2}_{k+1,n}[\mathbf{b}_{k+1}^{\top}A^{1}_{k+1,n}]_{j}$}:\\
Since
the calculations are analogous to those in Subsection~\ref{subsubsection_appendix_technical_computation_1}, we
have $$\frac{n}{\ln(n)}\sum_{k=m_0}^{n-1} r_k^2
[\mathbf{b}_{k+1}]_ja^{2}_{k+1,n}[A^{1}_{k+1,n}]_{jj}
\stackrel{a.s}\longrightarrow 0.$$
\item
\textit{$a.s.-\lim_{n}\frac{n}{\ln(n)}\sum_{k=m_0}^{n-1} r_k^2
  a^{2}_{k+1,n}[\mathbf{b}_{k+1}^{\top}A^{3}_{k+1,n}]_{j}$}:\\
Since
the calculations are analogous to those in Subsection~\ref{subsubsection_appendix_technical_computation_1}, we
have $$\frac{n}{\ln(n)}\sum_{k=m_0}^{n-1} r_k^2
[\mathbf{b}_{k+1}]_ja^{2}_{k+1,n}[A^{3}_{k+1,n}]_{jj}
\stackrel{a.s}\longrightarrow0.$$
\end{itemize}

\section{Stable convergence and its variants}\label{app-B}

This brief appendix contains some basic definitions and results
concerning stable convergence and its variants. For more details, we
refer the reader to \cite{crimaldi-2009, crimaldi-libro,
  cri-let-pra-2007, hall-1980} and the references therein.\\

\indent Let $(\Omega, {\mathcal A}, P)$ be a probability space, and let
$S$ be a Polish space, endowed with its Borel $\sigma$-field. A {\em
  kernel} on $S$, or a random probability measure on $S$, is a
collection $K=\{K(\omega):\, \omega\in\Omega\}$ of probability
measures on the Borel $\sigma$-field of $S$ such that, for each
bounded Borel real function $f$ on $S$, the map
$$
\omega\mapsto
K\!f(\omega)=\int f (x)\, K(\omega)(dx)
$$
is $\mathcal A$-measurable. Given a sub-$\sigma$-field $\mathcal H$ of
$\mathcal A$, a kernel $K$ is said $\mathcal H$-measurable if all the
above random variables $K\!f$ are $\mathcal H$-measurable.\\

\indent On $(\Omega, {\mathcal A},P)$, let $(Y_n)_n$ be a sequence of
$S$-valued random variables, let $\mathcal H$ be a sub-$\sigma$-field
of $\mathcal A$, and let $K$ be a $\mathcal H$-measurable kernel on
$S$. Then we say that $Y_n$ converges {\em $\mathcal H$-stably} to
$K$, and we write $Y_n\longrightarrow K$ ${\mathcal H}$-stably, if
$$
P(Y_n \in \cdot \,|\, H)\stackrel{weakly}\longrightarrow
E\left[K(\cdot)\,|\, H \right]
\qquad\hbox{for all } H\in{\mathcal H}\; \hbox{with } P(H) > 0,
$$where $K(\cdot)$ denotes the random variable
  defined, for each Borel set $B$ of $S$, as $\omega\mapsto
  K\!I_B(\omega)=K(\omega)(B)$.  In the case when ${\mathcal
  H}={\mathcal A}$, we simply say that $Y_n$ converges {\em stably} to
$K$ and we write $Y_n\longrightarrow K$ stably. Clearly, if
$Y_n\longrightarrow K$ ${\mathcal H}$-stably, then $Y_n$ converges in
distribution to the probability distribution $E[K(\cdot)]$. Moreover,
the $\mathcal H$-stable convergence of $Y_n$ to $K$ can be stated in
terms of the following convergence of conditional expectations:
\begin{equation}\label{def-stable}
E[f(Y_n)\,|\, {\mathcal H}]\stackrel{\sigma(L^1,\, L^{\infty})}\longrightarrow
K\!f
\end{equation}
for each bounded continuous real function $f$ on $S$. \\

\indent In \cite{cri-let-pra-2007} the notion of $\mathcal H$-stable
convergence is firstly generalized in a natural way replacing in
(\ref{def-stable}) the single sub-$\sigma$-field $\mathcal H$ by a
collection ${\mathcal G}=({\mathcal G}_n)_n$ (called conditioning
system) of sub-$\sigma$-fields of $\mathcal A$ and then it is
strengthened by substituting the convergence in
$\sigma(L^1,L^{\infty})$ by the one in probability (i.e. in $L^1$,
since $f$ is bounded). Hence, according to \cite{cri-let-pra-2007}, we
say that $Y_n$ converges to $K$ {\em stably in the strong sense}, with
respect to ${\mathcal G}=({\mathcal G}_n)_n$, if
\begin{equation}\label{def-stable-strong}
E\left[f(Y_n)\,|\,{\mathcal G}_n\right]\stackrel{P}\longrightarrow K\!f
\end{equation}
for each bounded continuous real function $f$ on $S$.\\

\indent Finally, a strengthening of the stable convergence in the
strong sense can be naturally obtained if in (\ref{def-stable-strong})
we replace the convergence in probability by the almost sure
convergence: given a conditioning system ${\mathcal G}=({\mathcal
  G}_n)_n$, we say that $Y_n$ converges to $K$ in the sense of the
{\em almost sure conditional convergence}, with respect to ${\mathcal
  G}$, if
\begin{equation*}
E\left[f(Y_n)\,|\,{\mathcal G}_n\right]\stackrel{a.s.}\longrightarrow K\!f
\end{equation*}
for each bounded continuous real function $f$ on
  $S$. The almost sure conditional convergence has been introduced in
  \cite{crimaldi-2009} and, subsequently, employed by others in the
  urn model literature (e.g. \cite{aletti-2009, z}).  \\

We now conclude this section recalling two convergence results that we
need in our proofs. \\

From \cite[Proposition~3.1]{cri-pra}, we can get the following
result.

\begin{theo}\label{thm:triangular}
Let $({\mathbf T}_{n,k})_{n\geq 1, 1\leq k\leq k_n}$ be a triangular
array of $d$-dimensional real random vectors, such that, for each
fixed $n$, the finite sequence $({\mathbf T}_{n,k})_{1\leq k\leq k_n}$
is a martingale difference array with respect to a given filtration
$({\mathcal G}_{n,k})_{k\geq 0}$. Moreover, let $(t_n)_n$ be a
sequence of real numbers and assume that the following
conditions hold:
\begin{itemize}
\item[(c1)] ${\mathcal G}_{n,k}{\underline{\subset}} {\mathcal G}_{n+1,
  k}$ for each $n$ and $1\leq k\leq k_n$;
\item[(c2)] $\sum_{k=1}^{k_n} (t_n{\mathbf
  T}_{n,k})(t_n{\mathbf T}_{n,k})^{\top}=t_n^2\sum_{k=1}^{k_n} {\mathbf
  T}_{n,k}{\mathbf T}_{n,k}^{\top} \stackrel{P}\longrightarrow \Sigma$,
  where $\Sigma$ is a random positive semi\-defi\-ni\-te matrix;
\item[(c3)] $\sup_{1\leq k\leq k_n} |t_n{\mathbf T}_{n,k}|
\stackrel{L^1}\longrightarrow 0$.
\end{itemize}
Then $t_n\sum_{k=1}^{k_n}{\mathbf T}_{n,k}$ converges stably to the
Gaussian kernel ${\mathcal N}(\mathbf{0}, \Sigma)$.
\end{theo}

The following result combines together stable convergence and 
stable convergence in the strong sense.

\begin{theo}\cite[Lemma 1]{ber-cri-pra-rig}\label{blocco}
Suppose that $C_n$ and $D_n$ are $S$-valued random variables, that $M$
and $N$ are kernels on $S$, and that ${\mathcal G}=({\mathcal G}_n)_n$
is a filtration satisfying for all $n$
$$
\sigma(C_n)\underline\subset{\mathcal G}_n\quad\hbox{and }\quad
\sigma(D_n)\underline\subset
\sigma\left({\textstyle\bigcup_n}{\mathcal G}_n\right)
$$

\noindent If $C_n$ stably converges to $M$ and $D_n$ converges to $N$
stably in the strong sense, with respect to $\mathcal G$, then
$$
(C_n, D_n)\longrightarrow M \otimes N \qquad\hbox{stably}.
$$
(Here, $M\otimes N$ is the kernel on $S\times S$ such that $(M
\otimes N )(\omega) = M(\omega) \otimes N(\omega)$ for all $\omega$.)
\end{theo}


\end{document}